\documentclass{article}
\usepackage{amsthm}
\usepackage{amsfonts}
\usepackage{amsmath}
\usepackage{paralist}
\usepackage{graphics} 
\usepackage{epsfig} 
\usepackage{graphicx}
\usepackage{epstopdf}
\usepackage[colorlinks=true]{hyperref}
\hypersetup{urlcolor=blue, citecolor=red}
\usepackage{todonotes}
\usepackage{mathtools} 

\usepackage{caption}
\usepackage{subcaption}

\newcommand{\R}{\mathbb{R}}

\newcommand{\T}{\mathbb{T}}

\newcommand{\RR}{\mathbb{R}}

\newcommand{\indic}{\mathbf{1}}

\newcommand{\TT}{\mathbb{T}}

\newcommand{\grad}{\nabla}

\newcommand{\diver}{{\rm{div}}}
\newcommand{\p}{\partial}
\newcommand{\dt}{\Delta t}

\DeclareMathOperator{\EE}{\mathbb{E}}
\DeclareMathOperator{\PP}{\mathbb{P}}

\newcommand{\cF}{{\mathcal F}}
\newcommand{\cG}{{\mathcal G}}

\newcommand{\cJ}{{\mathcal J}}
\newcommand{\cA}{{\mathcal A}}
\newcommand{\cB}{{\mathcal B}}

\newcommand{\cC}{{\mathcal C}}

\newcommand{\cP}{{\mathcal P}}

\newcommand{\cT}{{\mathcal T}}

\newcommand{\cK}{{\mathcal K}}

\newcommand{\ds}{\displaystyle}

\newcommand {\argmin}{\mathrm{arg min} }

\newcommand{\beq}{\begin{equation}}
\newcommand{\beqa}{\begin{eqnarray}}
\newcommand{\bea} {\begin{array}{ll}}
\newcommand{\beqan}{\begin{eqnarray*}}
\newcommand{\eeq}{\end{equation}}
\newcommand{\eeqa}{\end{eqnarray}}
\newcommand{\eeqan}{\end{eqnarray*}}
\newcommand{\eea} {\end{array}}
\newtheorem{theorem}{Theorem}

\newtheorem{lemma}[theorem]{Lemma}

\newtheorem{remark}[theorem]{Remark}

\numberwithin{equation}{section}
\title{Mean Field Type Control with Congestion (II):  \\ An Augmented Lagrangian Method}
\author{Yves Achdou \thanks { Univ. Paris Diderot, Sorbonne Paris Cit{\'e}, Laboratoire Jacques-Louis Lions, UMR 7598, UPMC, CNRS, F-75205 Paris, France.
 achdou@ljll.univ-paris-diderot.fr},
Mathieu Lauri{\`e}re \thanks { Univ. Paris Diderot, Sorbonne Paris Cit{\'e}, Laboratoire Jacques-Louis Lions, UMR 7598, UPMC, CNRS, F-75205 Paris, France.}
}

\begin{document}
\maketitle

\begin{abstract}
	This work deals with  a numerical method for solving a mean-field type control problem with congestion.
It is the continuation of an article by the same authors,  in which suitably defined weak solutions of  the system of partial differential equations 
arising from the model were discussed and existence and uniqueness were proved. Here, the focus is put on numerical methods: 
a monotone finite difference scheme is proposed and shown to have a variational interpretation. Then an Alternating Direction Method of Multipliers
for solving the  variational problem is addressed. It is based on an augmented Lagrangian.
Two kinds of boundary conditions are considered: periodic conditions and more realistic boundary conditions associated to state constrained problems. Various test cases and numerical results are presented.
\end{abstract}
\footnotetext[1]{ The  authors   were partially supported by  ANR projects ANR-12-MONU-0013 and ANR-16-CE40-0015-01.}

\section{Introduction}
In the recent years,  an important research activity has been devoted to the study of stochastic differential games  with a large number of players.
In their pioneering articles  \cite{MR2269875,MR2271747,MR2295621}, J-M. Lasry and P-L. Lions have introduced the notion of {\sl mean field games},
sometimes refered to  as MFGs for short,
which describe the   asymptotic behavior of  stochastic differential games (Nash equilibria) as the number $N$ of players
tends to infinity.  In  such models, it is assumed that the agents are all identical and that 
an individual agent can hardly influence the outcome of the game.  Moreover, each individual strategy is influenced by some averages of functions of 
the states of the other agents. In the limit when $N\to +\infty$, a given agent feels the presence of the other agents through the 
statistical distribution of the states. Since perturbations of a single agent's strategy does not influence the statistical distribution of the states, 
the latter acts as a parameter  in the control problem to be solved by each  agent.  
\\
Another kind of asymptotic regime is obtained by  assuming that all the agents use the same distributed feedback strategy
 and by passing  to the limit as $N\to \infty$ before optimizing the common feedback. Given a common feedback strategy, the asymptotics are 
given by McKean-Vlasov theory, see \cite{MR0221595,MR1108185} : the dynamics of a given agent is found by solving  a stochastic differential equation whose coefficients depend on 
a mean field, namely the statistical distribution of the states, which may also affect the objective function. Since the feedback strategy is common to all agents, perturbations of the latter affect the mean field.  Then, having each player optimize its objective function amounts to solving a control problem
 driven by  McKean-Vlasov dynamics. The latter is named control of McKean-Vlasov dynamics by R. Carmona and F. Delarue \cite{MR3045029,MR3091726} and mean field type control by A. Bensoussan et al, \cite{MR3037035,MR3134900}.
\\
When the dynamics of the players are independent stochastic processes, both mean field games and control of  McKean-Vlasov dynamics 
naturally lead to a  system of coupled partial differential equations, a forward Kolmogorov equation
 and a backward Hamilton-Jacobi--Bellman equation.
For mean field games, the latter system  has been studied by Lasry and Lions in  \cite{MR2269875,MR2271747,MR2295621}. Besides, many important aspects of the mathematical theory on MFGs developed by  the same authors  are not published in journals or books, but can be found in the videos of  the lectures of P-L. Lions at Coll{\`e}ge de France: see the web site of Coll{\`e}ge de France, \cite{PLL}. One can also see \cite{MR3195844} for a brief survey.
 The analysis of the system of partial differential equations arising from mean field type control can be performed with rather similar arguments to
those used for MFGs,
 see \cite{MR3392611} for a work devoted to classical solutions. 
\bigskip

The class  of   MFGs with congestion effects  was introduced and studied by P-L. Lions in  \cite{PLL} in 2011, see also \cite{MR3135339,YAJML, MR3392611} for some numerical simulations,
 to model  situations in  which the cost of displacement  increases in the regions where the density of agents is large.
A striking fact is that in general,  MFGs with congestion cannot be cast into an optimal control problem driven by a partial differential equation, 
in contrast with simpler cases studied by Cardaliaguet et al in \cite{cardaliaguet2014second} for example.  Note that  \cite{cardaliaguet2014second} is inspired 
by the literature on optimal transport, see e.g. \cite{BenamouBrenier2000monge,MR3116016}.
  In contrast with MFGs, mean field type control problems have a genuine variational structure
 i.e. thay can always be seen as problems of optimal control driven by  partial differential equations.
 In \cite{MR3498932}, inspired by  \cite{cardaliaguet2014second} ,  we took advantage of the latter observation to deal with mean field type control with congestion and possibly degenerate diffusions. We introduced a pair of primal/dual optimization problems leading  to a suitable weak formulation of the system 
of partial differential equations for which there exists a unique solution. 

 Note that the variational approach is not the only possible one to deal with weak solutions, see \cite{MR2295621} and 
the nice article of A. Porretta,  \cite{porretta2014},  on weak solutions of Fokker-Planck equations
 and of  MFGs  systems of partial differential equations. Similarly,  \cite{YAAP16} 
contains existence and uniqueness results for suitably defined weak solutions of the systems  of PDEs arising from  MFGs with congestion effects,
 which do not rely on any variational interpretation.

The goal of the present paper is to design a numerical algorithm in order to compute solutions to some mean field control problems including congestions effects. 
Although the scheme is of the same nature as those already proposed in \cite{MR2679575,MR3135339} for MFGs and relies on a monotone discrete Hamiltonian  constructed 
using upwinding, the novelty here lies in  the  natural interpretation of the discrete scheme in terms of an optimal control problem. This allows us to characterize the solution  as the saddle point of a Lagrangian.  We can then define an augmented Lagrangian and propose an Alternating Direction Method of Multipliers (ADMM) to compute the solution.
 This type of algorithm is described in~\cite{FortinGlowinski2000augmented} and was  applied in \cite{BenamouBrenier2000monge} 
to optimal transport and more recently in \cite{BenamouCarlier2015ALG2} to some MFGs with a variational structure.  
In \cite{BenamouCarlier2015ALG2}, Benamou and Carlier restricted themselves to first order MFGs,
 because,  in the second order case, the ADMM requires the solution of a degenerate fourth order 
linear partial differential equation and leads to systems of linear equations with very high condition numbers. 
This issue was recently addressed by R. Andreev, see \cite{andreev2016preconditioning}, who proposed suitable multilevel preconditioners and extended the augmented Lagrangian approach to second order MFGs. \\ 
As emphasized for instance by Benamou and Carlier \cite{BenamouCarlier2015ALG2}, the advantage of such approaches is
 that that they ensure that the mass remains nonnegative and that they are adapted to weak solutions (in the sense that the Bellman equation may not hold where the density is zero). However, in the context of MFGs, they can only be applied to 
problems having a variational structure, therefore not to the congestion models of P-L. Lions. 
\\
This restriction does not exist for mean field type control problems,  since as already mentioned,  the latter always have a variational structure.
Another important point for the models  considered here is that Hamiltonian  becomes singular as the density
 vanishes; nevertheless, we shall see that the present algorithm correctly handles situations when the density of states local vanishes, see Remark~\ref{rem:convergenceADMM} below.
\\
The paper is organized as follows. In the remaining part of the introduction, we recall the mean field type control problem with congestion 
 and the notations introduced in \cite{MR3498932}. Section~\ref{sec:scheme-periodic} is devoted to an Alternating Direction Method of Multipliers (ADMM) for this problem in the periodic setting, {\it i.e.} when the domain is a $d$-dimensional torus. Then, in section~\ref{sec:state-constraints}, we extend the latter  method to the case of state constrained problems. Finally, various test cases and numerical results are presented in section~\ref{sec:numerics}.

\subsection{Model and assumptions}\label{subsec:model}
\label{sec:model-assumptions}
The model considered in the present work leads to the following system of partial differential equations:
\begin{eqnarray}
\label{eq:HJB-mftc-cong}
&  \ds
\frac{\partial u} {\partial t} (t,x) + \nu \Delta u(t,x)  \ds+  H (x, m(t,x), D u(t,x)) 
 + m(t,x)
 \frac{\partial  H} {\partial m}  
(x,m(t,x), D u(t, x)) 
  &= 0 ,
\\
\label{eq:FP-mftc-cong}
	& \ds  \frac{\partial m} {\partial t} (t,x)  
- \nu \Delta m(t,x)  
\ds 
 +
 \diver\Bigl( m(t,\cdot) \frac{\partial H} {\partial p} (\cdot, m(t,\cdot),D u(t,\cdot))\Bigr)(x) &=0,
\end{eqnarray}
with the initial and terminal conditions 
\begin{equation}
\label{eq:3}
  m(0,x)=m_0(x)\quad \hbox{and}\quad u(T,x) = u_T(x).
\end{equation}
\begin{remark}
	For the implementation (see \S~\ref{subsec:algoADMM}), we will restrict ourselves to the first order system, i.e. we will take $\nu = 0$.
The method can be extended to second order system by using suitable multilevel preconditioners for solving the fourth order systems of linear equations
which arise in this case, see \cite{andreev2016preconditioning}.
\end{remark}
\paragraph{Standing assumptions.} We now list the assumptions on the Hamiltonian $H$,  the initial and terminal conditions $m_0$ and $u_T$.
These conditions are supposed to hold in all what follows.
\begin{description}
\item[H1] The Hamiltonian $H:  \T^d\times (0,+\infty) \times \R^d \to \R $ is of the form
  \begin{equation}
    \label{eq:4}
    H(x,m,p)= -\frac {|p|^\beta}{m^{\alpha}} +\ell (x,m),
  \end{equation}
  with $1<\beta \le 2$ and $0\le \alpha<1$, and where $\ell $ is continuous cost function that will be discussed below. It is clear that $H$ is concave with respect to $p$.
  Calling $\beta^*$ the conjugate exponent of $\beta$, {\it i.e.} $\beta^* =\beta /(\beta-1)$, it is useful to note that
  \begin{eqnarray}
\label{eq:5}
  H(x,m, p)=  \inf_{\xi \in \R^d }   \left( \xi\cdot p +L(x,m,\xi)  \right) ,\\
\label{eq:6}  L(x,m,\xi)=(\beta-1) \beta^{-\beta^*}  m^{\frac \alpha {\beta- 1}} |\xi|^{\beta^*} +\ell (x,m),
  \end{eqnarray}
where $L$ is convex with respect to $\xi$, and that
\begin{equation}\label{eq:7}
  L(x,m,\xi)= \sup_{p\in \R^d} (-\xi\cdot p + H(x,m,p)).
\end{equation}
Hereafter, we shall always make the convention that $mH(x,m,p)=0$ if $m=0$.
  \item[H2] (conditions on the cost $\ell$) The function $\ell: \T^d \times \R_+\to \R$ is continuous with respect to both variables and
continuously differentiable with respect to $m$ if $m>0$. We also assume that  $m\mapsto m\ell(x, m)$ is strictly convex, and that 
there exist $q>1$ and two positive constants $C_1$ and $C_2$ such that
\begin{eqnarray}
\label{eq:8}  \frac 1 {C_1 } m^{q-1}- C_1 \le \ell(x,m)\le C_1   m^{q-1}+ C_1,\\
\label{eq:9} \frac 1 {C_2 } m^{q-1}- C_2\le m \frac {\partial\ell }  {\partial m}  (x,m) \le C_2  m^{q-1}+ C_2.
\end{eqnarray}
Moreover, there exists a real function $\underline \ell : \T^d \to \R$ such that:
\begin{equation}
  \label{eq:10}
\ell (x,m)\ge \underline\ell(x),\quad \forall x\in \T^d, \; \forall m\ge 0. 
\end{equation}
The convexity assumption on $m\mapsto m\ell(x, m)$ implies that $m\mapsto mH(x,m, p)$ is strictly convex with respect to $m$.\\
Moreover, we assume that there exists a constant $C_3\ge 0$ such that 
\begin{equation}
  \label{eq:11}
| \ell(x,m)-\ell(y,m)|\le C_3(1+ m^{q-1}) |x-y|.
\end{equation}
\item[H3] We assume that $\beta \ge  q^*$.
\item[H4](initial and terminal conditions) We assume that $m_0$ is of class $\cC^1$ on $\T^d$, that $u_T$ is of class $\cC^2$ on $\T^d$ and that $m_0> 0$ and $\int_{\T^d} m_0(x)dx=1$.
\item[H5]  $\nu$ is a non negative real number.
\end{description}

\subsection{A heuristic justification of (\ref{eq:HJB-mftc-cong})-(\ref{eq:3})}
Consider a probability space $(\Omega, \cT, \cP)$ and a filtration $\cF^t$ generated by a $D$-dimensional standard Wiener process $(W_t)$ and
the stochastic process $(X_t)_{t\in [0,T]}$   in $\R^d$, adapted to $\cF^t$, which solves the stochastic differential equation
\begin{equation}\label{eq:12}
  d X_t = \xi_t \;d t + \nu \;d W_t \qquad \forall t\in [0,T],
\end{equation}
given the initial state $X_0$, which is a random variable $\cF^0$-measurable, whose probability density is  $m_0$, and a bounded stochastic process $\xi_t$ adapted to $\cF_t$ (the control).
More precisely,  we will take
\begin{equation}\label{eq:13}
\xi_t=v(t,X_t), \end{equation}
where $v(t,\cdot)$ is a continuous function on $\T^d$.
As explained in \cite{MR3134900}, page 13, if the feedback function $v$ is smooth enough, then the probability distribution $m_{t}$ of $X_t$ has a density with respect to the Lebesgue measure,  $m_{v}(t,\cdot)\in \PP\cap L^1(\T^d)$ for all $t$,
and   $m_v$ is solution of the Fokker-Planck equation
 \begin{equation}
\label{eq:14}
\frac{\partial m_v} {\partial t} (t,x)  -    \nu \Delta m(t,x)  +
\diver \Big( m_v(t,\cdot)  v(t,\cdot)  \Big) (x)=0,\;\; 
 \end{equation}
for $t\in (0,T]$ and $ x\in \T^d$,
with the initial condition
\begin{equation}
  \label{eq:15}
m_v(0,x)= m_0(x),\quad x\in \T^d.
\end{equation}
 We define the objective function
\begin{equation}
  \label{eq:16}
  \begin{split}
    \cJ(v) &= \EE\left[ \int_0^T L( X_t,m_v(t, X_t), \xi_t) dt  + u_T(X_T)\right]\\
&=\ds \int_{[0,T]\times\T^d} L( x, m_v(t,x),v(t,x)) m_v(t,x) dx dt + \int_{\T^d} u_T(x) m_v(T,x) dx.    
  \end{split}
\end{equation}
The goal is to minimize $\cJ(v)$ subject to (\ref{eq:14}) and (\ref{eq:15}). Following A. Bensoussan, J. Frehse and P. Yam  in  \cite{MR3134900},   it can be seen   that if there exists a smooth feedback function $v^* $ achieving
$\cJ( v^*)= \min \cJ(v)$  and such that $ m_{v^*}>0$    then
\begin{displaymath}
  v^*(t,x)={\rm{argmin}}_{v}  \Bigl( L(x, m_{v^*}(t,x),v)+ \nabla u(t,x)\cdot v(t,x)\Bigr)
\end{displaymath}
and $(m_{v^*},u)$ solve~(\ref{eq:HJB-mftc-cong}), (\ref{eq:FP-mftc-cong}) and (\ref{eq:3}). 
The condition $ m>0$ is necessary for the equation~(\ref{eq:HJB-mftc-cong}) to be defined pointwise.

The issue with the latter argument
is that we do not know how to guarantee a priori that 
 $m$ will not vanish in some region of $(0,T)\times \T^d$. 
In \cite{MR3498932}, we proposed a theory of weak solutions of (\ref{eq:HJB-mftc-cong})-(\ref{eq:3}), in order to cope with the cases when $m$ may vanish. 

\subsection{Two optimization problems}
\label{sec:two-optim-probl}
Let us recall the two optimization problems and the notations that we introduced in~\cite{MR3498932}. 
The first optimization is described as follows. Consider the set $\cK_0$:
\begin{displaymath}
  \cK_0=\left\{ \phi \in \cC^2 ([0,T]\times \T^d): \phi(T,\cdot)= u_T \right\}
\end{displaymath}
and the functional $\cA$ on $\cK_0$:
\begin{equation}
  \label{eq:pbA}
\cA(\phi)=\inf_{
  \begin{array}[c]{l}
m\in L^1((0,T)\times \T^d)\\ m\ge 0    
  \end{array}
}  \cA(\phi,m)
\end{equation}
where 
\begin{equation}
  \label{eq:18}
 \cA(\phi,m)=
  \begin{array}[t]{l}
\ds \int_0^T \int_{\T^d}    m(t,x)\left(\frac{\partial \phi} {\partial t} (t,x)  +\nu \Delta \phi\; (t,x)   +  H
(x, m(t,x), D \phi(t,x) )          \right)    dx dt    \\ \ds +\int_{\T^d}    m_0(x)\phi(0,x) dx.
  \end{array}\end{equation}
Then the first problem reads
\begin{equation}
  \label{eq:19}
  \sup_{\phi\in \cK_0} \cA(\phi).
\end{equation}
For the second optimization problem, we consider the set  $\cK_1$:
\begin{equation}
  \label{eq:20}
  \cK_1=\left\{
    \begin{array}[c]{l}
(m,z) \in L^1 ((0,T)\times \T^d) \times L^1 ((0,T)\times \T^d;\R^d)  : \\ m\ge 0 \hbox{ a.e}
\\
	\ds \frac{\partial m} {\partial t}   -  \nu \Delta m  +  \diver   z  =0 ,
\qquad m(0,\cdot)= m_0 
    \end{array}
\right\}
\end{equation}
where the boundary value problem is satisfied in the sense of distributions.
We also define
\begin{equation}
\label{eq:21}
\widetilde L (x,m, z)= \left\{
  \begin{array}[c]{rl}
m L(x, m, \frac z m ) \quad &\hbox{ if } m>0\\
0    \quad &\hbox{ if } (m,z)=(0,0)\\
+\infty \quad &\hbox{ otherwise. } 
  \end{array}
    \right.
\end{equation}
Note that $(m,z) \mapsto \widetilde L(x,m,z)$ is LSC on $\R\times \R^d$.
Using (\ref{eq:6}), Assumption  ${\rm H2}$ and the results of \cite{MR3392611} paragraph 3.2, it can be proved that  
 $(m,z) \mapsto \widetilde L(x,m,z)$ is convex on $\R\times \R^d$, because $0<\alpha<1$. It can also be checked that
\begin{equation}\label{eq:22}
  \widetilde L(x,m,z)=\left\{
    \begin{array}[c]{ll}
\ds \sup_{p\in \R^d} (-z\cdot p + mH(x,m,p))\quad \quad &\hbox{if } m>0 \hbox{ or } (m,z)=(0,0)      \\
+\infty &\hbox{otherwise.}
    \end{array}\right.
\end{equation}
Since $L$ is bounded from below, $\ds \int_0^T \int_{\T^d} \widetilde L (x,m(t,x), z(t,x)) dx dt $ is well defined in $\R \cup \{+\infty\}$ for all $(m, z) \in \cK_1$.
 Then, the second problem is: 
 \begin{equation}
   \label{eq:pbB}
  \inf_{(m,z)\in \cK_1} \cB(m,z),
 \end{equation}
where if $\ds \int_0^T \int_{\T^d} \widetilde L (x,m(t,x), z(t,x)) dx dt<+\infty$,
\begin{equation}
  \label{eq:24}
\ds \cB(m, z)=  \int_0^T \int_{\T^d} \widetilde L (x,m(t,x), z(t,x)) dx dt  + \int_{\T^d}     m(T,x) u_T(x) dx,
\end{equation}
 and if not,
\begin{equation}
  \label{eq:25}
\cB(m, z)= +\infty.
\end{equation}
\\
 To give a meaning to the second integral in (\ref{eq:24}), we define $w(t,x)= \frac {z(t,x)}{m(t,x)}$ if $m(t,x)>0$ and 
$w(t,x)=0$ otherwise. From (\ref{eq:6}) and (\ref{eq:8}), we see that   $\int_0^T \int_{\T^d} \widetilde L (x,m(t,x), z(t,x)) dx dt <+\infty$ implies that
 $m ^{1+\frac \alpha {\beta-1}} |w|^{\beta ^*}\in L^1( (0,T)\times \T^d)$,   which implies that 
 $m  |w|^{\frac \beta  {\beta -1 +\alpha}}\in L^1( (0,T)\times \T^d)$. 
In that case, the boundary value problem in (\ref{eq:20}) can be rewritten as follows:
\begin{equation}
  \label{eq:26}
	\ds \frac{\partial m} {\partial t}   - \nu \Delta m +  \diver   (mw)  =0 ,\quad\quad m(0,\cdot)= m_0,
\end{equation}
and we can use Lemma 3.1 in \cite{cardaliaguet2014second} in order to obtain the following:
\begin{lemma}
\label{sec:two-optim-probl-1}
If $(m,w)\in \cK_1$ is such that $\int_0^T \int_{\T^d} \widetilde L (x,m(t,x), z(t,x)) dx dt <+\infty$, then
  the map $t\mapsto m(t)$  for $t\in (0,T)$ and $t\mapsto m_0$ for $t<0$  is H{\"o}lder continuous a.e. for the weak * topology of $\cP(\T^d)$.
\end{lemma}
Lemma~\ref{sec:two-optim-probl-1} implies that the measure $m(t)$ is defined for all $t$, so the second integral in (\ref{eq:24}) has a meaning. In \cite{MR3498932} the following result has been proven.
\begin{lemma}
  \label{sec:two-optim-probl-2}
  The two problems $\cA$ and $\cB$ are in duality, in the sense that:
  \begin{equation}
    \label{eq:27}
\sup _{\phi\in \cK_0} \cA(\phi)= \min_{(m,z)\in \cK_1} \cB(m,z).
  \end{equation}
Moreover the latter minimum is achieved by a unique $(m^*, z^*)\in \cK_1$, and $m^*\in L^q((0,T)\times \T^d)$.
\end{lemma}
The proof is based on the following observations: firstly, using the Fenchel-Moreau theorem, see e.g. \cite{MR1451876}, problem $\cA$ can be written:
\begin{equation}
  \label{eq:33}
  \sup_{\phi\in \cK_0} \cA(\phi)=   - \inf_{\phi\in E_0} \left( \cF(\phi) + \cG (\Lambda(\phi)) \right),
\end{equation}
where the functional $\cF$ and $\chi_T$ are defined on $E_0=\cC^2([0,T]\times \T^d)$ by:
  \begin{equation*}
    \cF(\phi)= \chi_T(\phi)- \int_{\T^d} m_0(x)\phi(0,x)dx, 
    \qquad
    \chi_T(\phi)=
    \begin{cases}
    0 & \hbox{ if } \phi|_{t=T}=u_T \\
    +\infty & \hbox{ otherwise,}
    \end{cases}
  \end{equation*}
and the functional $\cG$ is defined on $E_1=\cC^0([0,T]\times \T^d)\times \cC^0([0,T]\times \T^d;\R^d)$ by:
\begin{equation}
  \label{eq:28}
\cG(a, b)= -\inf_{ \begin{array}[c]{l}
m\in L^1((0,T)\times \T^d)\\ m\ge 0    
  \end{array}} \int_0^T \int_{\T^d}    m(t,x) \left( a(t,x)+  H
(x, m(t,x), b(t,x) ) \right)    dx dt   .
\end{equation}
and $\Lambda$ is the linear operator $\Lambda: E_0\to E_1$ defined by
\begin{displaymath}
  \Lambda(\phi)= \left(\frac{\partial \phi} {\partial t}   +   \nu \Delta \phi, D\phi \right).
\end{displaymath}
Secondly, problem $\cB$ can be written
\begin{displaymath}
 \inf_{(m,z)\in \cK_1} \cB(m,z)
 = \inf_{(m,z)\in E_1^* } \left(      \cF^* (\Lambda^* (m,z)) + \cG^* (-m,-z) \right),
\end{displaymath}
where $E_1^*$ is the topological dual of $E_1$ {\it i.e.} the set of Radon measures $(m,z)$ on $(0,T)\times \T^d$ with values in $\R\times \R^d$. If $E_0^*$ is the dual space of $E_0$, the operator $\Lambda^*: E_1^* \to E_0^*$ is the adjoint of $\Lambda$. The maps $\cF^*$ and $\cG^*$ are the Legendre-Fenchel conjugates of $\cF$ and $\cG$. 

The conclusion of the proof then relies on Fenchel-Rockafellar duality theorem, see \cite{MR1451876}.

To design our Augmented Lagrangian algorithm, we will introduce discrete counterparts of problems $\cA$ and $\cB$, and also of the operators $\cF, \cG, \Lambda, \cF^*, \cG^*$ and $\Lambda^*$.

\section{Numerical scheme in the periodic setting}\label{sec:scheme-periodic}
\subsection{Discretization}
 In the sequel, $\R_+ = [0,+\infty)$,  and for any $x \in \R$, $x^+ = \max(x,0), x^- = \max(0,-x)$ are respectively the positive and the negative parts of $x$.
\\
We focus on the two-dimensional case, i.e. $d=2$. Let $\T^2_h$ be a uniform grid on the unit two-dimensional torus with mesh step $h$ 
such that $1/h$ is an integer $N_h$. We note by $x_{i,j}$ the point in $\T^2_h$ of coordinates $(i h, j h)$,
where $i,j$ are understood modulo $N_h$ if needed. For a  positive integer $N_T$, 
consider $\Delta t = T/N_T$ and $t_n = n\Delta t$, $n = 0,\dots, N_T$. \\
We note $N = (N_T+1) N_h^2$ the total number of points in the space-time grid. It will  sometimes be convenient to also use $N' = N_T N_h^2$. 
A grid function $f$ is a family of real numbers $(f_{i,j}^n)$ for  $n\in\{0,\dots,N_T\}$ and $i,j \in  \{0,\dots,N_h-1\}$. In the periodic setting, we agree that $f_{i \pm N_h,j}^n = f_{i,j \pm N_h}^n = f_{i,j}^n$.\\
For any positive integer $K$ and any $f,g \in \R^{K N}$, we define
$$
	\langle f, g\rangle_{\ell^2(\R^{KN})} = h^2 \Delta t \sum_{k=1}^K \sum_{i,j,n} (f^k)_{i,j}^n (g^k)_{i,j}^n, \quad ||f||_{\ell^2(\R^{KN})}= \sqrt{\langle f, f\rangle_{\ell^2(\R^{KN})}},
$$
and we use a similar definition for vectors in $\R^{KN'}$.
 When the  space of interest is clear from the context, we  use the notations $\langle ., .\rangle_{\ell^2}$ and $||.||_{\ell^2}$.
 The scalar product in $\R^ 2$ will be noted by $u\cdot v$. 
\\
 The discrete versions  of the data  are  ${\widetilde m}_{i,j}^0 = \frac 1 {h^2} \int_ {|x- x_{i,j}|_{\infty}<h/2} m_0(x) dx $
 and ${\widetilde u}_{i,j}^{N_T} = u_T(x_{i,j})$,  $i,j\in\{0,\dots,N_h-1\}$.
\\
We also introduce the one sided finite differences: for any $\phi \in \RR^{N_h^2}$,
\begin{equation*}
	(D_1^+ \phi)_{i,j} = \frac{\phi_{i+1,j} - \phi_{i,j}}{h}, \quad (D_2^+ \phi)_{i,j} = \frac{\phi_{i,j+1} - \phi_{i,j}}{h}, \qquad \forall \, i,j\in\{0,\dots,N_h-1\}.
\end{equation*}
We let $[\grad_h \phi ]_{i,j}$ be the collection of the four possible one sided finite differences at $x_{i,j}$:
\begin{equation*}
	[\grad_h \phi ]_{i,j} = \Big( (D_1^+ \phi)_{i,j}, (D_1^+ \phi)_{i-1,j}, (D_2^+ \phi)_{i,j}, (D_2^+ \phi)_{i,j-1} \Big) \in \RR^4,
\end{equation*}
and $\Delta_h$ be the discrete Laplacian, defined by
\begin{equation*}
	(\Delta_h \phi)_{i,j} = - \frac{1}{h^2} \big( 4 \phi_{i,j} - \phi_{i+1,j} - \phi_{i-1,j}  - \phi_{i,j+1} - \phi_{i,j-1}\big), \qquad \forall \, i,j\in\{0,\dots,N_h-1\}.
\end{equation*}
Finally, for any $(\phi^n)_{n \in\{0,\dots, N_T\}} \in \R^{N_T}$, consider the discrete time derivative:
\begin{equation*}
	(D_t \phi)^n = \frac{\phi^n - \phi^{n-1}}{\Delta t}, \qquad \forall \, n\in\{1,\dots,N_T\}.
\end{equation*}

The discrete Hamiltonian is the form $H_h(x,m, [D_h \phi])$ where $H_h : \T^2\times \R_+\times \R^4\to \R$ is given by
\begin{equation*}
	H_h(x,m,p_1,p_2,p_3,p_4) = - m^{-\alpha} \left( (p_1^-)^2 + (p_2^+)^2 + (p_3^-)^2 + (p_4^+)^2 \right)^{\beta/2} + \ell(x,m).
\end{equation*}
Although  $H_h$ does not depend explicitly on $h$, 
we use the index $h$ to distinguish the discrete Hamiltonian from the original one, namely $H$.
 We see that $H_h$ has the following properties:
\begin{itemize}
	\item \emph{monotonicity} : $H_h(x,m,p_1,p_2,p_3,p_4)$ is nondecreasing with respect to $p_1$ and $p_3$, and nonincreasing with respect to $p_2$ and $p_4$
	\item \emph{consistency} : $H_h(x,m,p_1,p_1,p_2,p_2) = H(x,m,p)$, $\forall x\in \T^2, m \geq 0, p=(p_1, p_2) \in \RR^2$
	\item \emph{differentiability} :  $H_h(x,m,p_1,p_2,p_3,p_4)$ is of class $\mathcal C^1$ w.r.t. $p_1, p_2, p_3, p_4$
	\item \emph{concavity} : $(p_1, p_2, p_3, p_4) \mapsto H_h(x,m, p_1, p_2, p_3, p_4)$ is concave.
\end{itemize}
Apart from the dependency on $m$, the discrete Hamiltonian $H_h$  is constructed in the same way as in the 
finite difference schemes proposed in \cite{MR2679575,MR3135339}
for mean field games.  The properties stated above made it possible to prove existence and uniqueness (under some additional assumptions) for the 
solutions of the discrete problems in  the MFGs'case, and to prove convergence to  either classical or weak solutions see \cite{MR3097034,MR3452251}.
The monotone character of the discrete Hamiltonian played a key role in all the latter results; this is precisely the  reason
 why we  prefer this kind of discrete Hamiltonian to the central schemes chosen in \cite{BenamouCarlier2015ALG2}.
 Note that a similar scheme was also used for MFGs
with congestion, see \cite{YAJML}.

\subsection{Discrete version of problem $\mathcal A_h$}
We introduce the discrete version of problem \eqref{eq:pbA}:
\begin{align}
	\mathcal A_h = 
	\sup_{\phi\in\R^N} \min_{m \in (\R_+)^N} \Big\{ &h^2 \Delta t \sum_{n=1}^{N_T} \sum_{i,j=0}^{N_h-1}  m_{i,j}^n \Big[(D_t \phi_{i,j})^n + \nu(\Delta_h \phi^{n-1})_{i,j}   +   H_h\left(x_{i,j}, m_{i,j}^n, [\grad_h \phi^{n-1}  ]_{i,j}
 \right)       \Big] \notag \\
	& -\chi_T(\phi) + h^2 \sum_{i,j=0}^{N_h-1} {\widetilde m}_{i,j}^0 \phi_{i,j}^0 \Big\} \label{eq:pbA1}
\end{align}
where 
\begin{equation*}
\chi_T(\phi) =
	\begin{cases}
		+\infty &\hbox{ if $\exists \, i,j$ s.t. $\phi_{i,j}^{N_T} \neq {\widetilde u}_{i,j}^{N_T}$} \\
		0 &\hbox{ otherwise.}
	\end{cases}
\end{equation*}
We can formulate $\mathcal A_h$ in terms of a convex problem as follows
\begin{align}
	\mathcal A_h 
	&= - \inf_{\phi \in \R^N} \big\{ \mathcal F_h(\phi) +  \mathcal G_h (\Lambda_h (\phi)) \big\},  \label{eq:pbA2}
\end{align}
where $\Lambda_h : \R^N \to \R^{5N'}$ is defined by~: $\forall \, n\in\{1,\dots,N_T\}, \forall \, i,j\in\{0,\dots,N_h-1\}$,
\begin{align*}
	(\Lambda_h(\phi))_{i,j}^n = \, 
		&\left( (D_t \phi_{i,j})^n +\nu\left(\Delta_h \phi^{n-1}\right)_{i,j}, [\grad_h \phi^{n-1}]_{i,j}  \right),
\end{align*}
(with the notation introduced above)
and $\mathcal F_h : \R^N \to \R \cup \{+\infty\}$ and $\mathcal G_h : \R^{5N'} \to \R \cup \{+\infty\}$ are the two lower semi-continuous proper functions defined by: for all $\phi \in \RR^N$, for all $(a,b,c,\tilde b, \tilde c) \in \RR^{5N'}$:
\begin{align*}
 \mathcal F_h(\phi) = \, &\chi_T(\phi) - h^2 \sum_{i,j=0}^{N_h-1} {\widetilde m}_{i,j}^0 \phi_{i,j}^0, \\
	\hbox{and } \quad 	\mathcal G_h(a,b,c,\tilde b,\tilde c) = &- \min_{m \in (\R_+)^N} \left\{ h^2 \Delta t \sum_{n=1}^{N_T} \sum_{i,j=0}^{N_h-1} m_{i,j}^n \left( a_{i,j}^{n} + H_h(x_{i,j}, m_{i,j}^n, b_{i,j}^{n}, c_{i,j}^{n}, \tilde b_{i,j}^{n}, \tilde c_{i,j}^{n})\right) \right\}\\
	= &-  h^2 \Delta t \sum_{n=1}^{N_T} \sum_{i,j=0}^{N_h-1} K_h(x_{i,j}, a_{i,j}^n, b_{i,j}^{n}, c_{i,j}^{n}, \tilde b_{i,j}^{n}, \tilde c_{i,j}^{n}),
\end{align*}
with 
\begin{align*}
	K_h(x, a_0, p_1, p_2, p_3, p_4) = &\min_{m \in \R_+} \big\{ m(a_0 + H_h(x, m, p_1, p_2, p_3, p_4) )\big\}.
\end{align*}
Note that $K_h$ is nonpositive. 
By definition of $H_h$, $K_h$ is concave in $(a_0, p_1, p_2, p_3, p_4)$,
 hence $\mathcal G_h$ is  convex.

\subsection{The dual version of problem $\mathcal A_h$}\label{sec:discreteDuality}
We will also need the dual version of problem $\mathcal A_h$.
From \eqref{eq:pbA2}, by Fenchel-Rockafellar theorem (see e.g. \cite{MR1451876}, Corollary 31.2.1), we deduce that the dual problem of $\mathcal A_h$ is:
\begin{equation}\label{eq:pbAdual1}
	\min_{\sigma \in \R^{5N'}} \big\{\mathcal F_h^*(\Lambda_h^*(\sigma)) + \mathcal G_h^*(-\sigma)\big\},
\end{equation}
where $\mathcal G_h^*$ and $\mathcal F_h^*$ are respectively the Legendre-Fenchel conjugates of $\mathcal G_h$ and $\mathcal F_h$, defined by:
 for all $\sigma \in \RR^{5N'}$, all $x\in\TT_h^2$, and all $\mu \in \RR^N$,
\begin{align}
	\mathcal F_h^*(\mu) = &\sup_{\phi \in \R^N} \Big\{ \langle\mu , \phi\rangle_{\ell^2(\R^{N})} - \mathcal F_h(\phi) \Big\},\notag
	\\
        \mathcal G_h^*(-\sigma)
	= &\,\max_{q \in \R^{5N'}} \Big\{ - \langle\sigma , q\rangle_{\ell^2(\R^{5N'})}  - \mathcal G_h(q)  \Big\} \notag \\
	= &\,\max_{q} \Big\{ h^2 \Delta t  \sum_{n=1}^{N_T} \sum_{i,j=0}^{N_h-1}\Big[ - \sigma_{i,j}^n \cdot q_{i,j}^n + K_h(x_{i,j},  q_{i,j}^{n})\Big] \Big\}  
	= \, h^2 \Delta t  \sum_{n=1}^{N_T} \sum_{i,j=0}^{N_h-1} \tilde L_h(x_{i,j}, \sigma_{i,j}^n) \label{eq:GstarLbar}
      \end{align}
with      
\begin{displaymath}
  \tilde L_h(x,\sigma_0) = \max_{q_0 \in \RR^{5}} \big\{ -\sigma_0 \cdot q_0 + K_h(x,q_0) \big\}, \quad \forall \sigma_0 \in \R^5.
\end{displaymath}
Finally $\Lambda_h^*: \R^{5N'} \to \R^N$ denotes the adjoint of $\Lambda_h$, defined by: for all $(m,y,z,\tilde y,\tilde z) \in \RR^{5N'}$, and all $\phi \in \RR^N$:
\begin{align}
	\langle \Lambda_h^*(m,y,z,\tilde y,\tilde z) , \phi \rangle_{\ell^2(\R^{N})}
	&= \langle (m,y,z,\tilde y,\tilde z) , \Lambda_h(\phi) \rangle_{\ell^2(\R^{5N'})}
	\label{eq:defLambdastar}\\
	&= h^2 \Delta t  \sum_{n=1}^{N_T}  \sum_{i,j=0}^{N_h-1} \left[ m_{i,j}^n \left( \frac{\phi_{i,j}^n - \phi_{i,j}^{n-1}}{\Delta t} + \nu \left(\Delta_h \phi^{n-1}\right)_{i,j} \right) \right.
	\notag\\
		&\qquad\qquad\qquad\qquad + y_{i,j}^n \frac{\phi_{i+1,j}^{n-1}- \phi_{i,j}^{n-1}}{h} + z_{i,j}^n\frac{\phi_{i,j}^{n-1}- \phi_{i-1,j}^{n-1}}{h}
	\notag\\
		&\qquad\qquad\qquad\qquad \left.+ \tilde y_{i,j}^n \frac{\phi_{i,j+1}^{n-1}- \phi_{i,j}^{n-1}}{h} + \tilde z_{i,j}^n\frac{\phi_{i,j}^{n-1}- \phi_{i,j-1}^{n-1}}{h} \right]
	 \notag\\
	&= h^2 \Delta t \sum_{i,j=0}^{N_h-1} \left[  \sum_{n=0}^{N_T-1} \frac{m_{i,j}^{n} - m_{i,j}^{n+1}}{\Delta t}\phi_{i,j}^n + \frac{1}{\Delta t} m_{i,j}^{N_T} \phi_{i,j}^{N_T} - \frac{1}{\Delta t} m_{i,j}^{0} \phi_{i,j}^{0}  \right]
	 \notag\\
	&\qquad + h^2 \Delta t \sum_{i,j=0}^{N_h-1} \sum_{n=0}^{N_T-1} \nu \left(\Delta_h m^{n+1}\right)_{i,j}  \phi_{i,j}^n 
	 \notag\\
	&\qquad + h^2 \Delta t  \sum_{n=0}^{N_T-1} \sum_{i,j=0}^{N_h-1} \left[ -\frac{y_{i,j}^{n+1}- y_{i-1,j}^{n+1}}{h}\phi_{i,j}^{n}  
	 -\frac{z_{i+1,j}^{n+1}- z_{i,j}^{n+1}}{h} \phi_{i,j}^{n} \right]
	  \notag\\
	 &\qquad + h^2 \Delta t  \sum_{n=0}^{N_T-1} \sum_{i,j=0}^{N_h-1} \left[ -\frac{\tilde y_{i,j}^{n+1}- \tilde y_{i,j-1}^{n+1}}{h}\phi_{i,j}^{n}
	  -\frac{\tilde z_{i,j+1}^{n+1}- \tilde z_{i,j}^{n+1}}{h} \phi_{i,j}^{n} \right],
	 \notag
 \end{align}
where we used discrete integration by parts and the periodic boundary condition.
Hence
\begin{align*}
	\mathcal F_h^*(\Lambda_h^*(m,y,z,\tilde y,\tilde z)) = 
	\begin{cases}
		h^2 \sum_{i,j=0}^{N_h-1}  m_{i,j}^{N_T} {\widetilde u}_{i,j}^{N_T} &\hbox{ if $(m,y,z,\tilde y, \tilde z)$ satisfies 
\eqref{eq:FPdiscrete}} \hbox{ (see below)}, \\
		+\infty &\hbox{otherwise},
	\end{cases}
\end{align*}
with $\forall \, i,j\in\{0,\dots,N_h-1\}$, $m^0_{i,j} = \widetilde m^0_{i,j}$, and $\forall \, n\in\{0,\dots,N_T-1\}$:
\begin{equation}\label{eq:FPdiscrete}
\frac{m_{i,j}^{n+1} - m_{i,j}^{n}}{\Delta t} - \nu\left(\Delta_h m^{n+1}\right)_{i,j} + \frac{y_{i,j}^{n+1} - y_{i-1,j}^{n+1}}{h} + \frac{z_{i+1,j}^{n+1} - z_{i,j}^{n+1}}{h} + \frac{\tilde y_{i,j}^{n+1} - \tilde y_{i-1,j}^{n+1}}{h} + \frac{\tilde z_{i,j+1}^{n+1} - \tilde z_{i,j}^{n+1}}{h} = 0.
\end{equation}
Hence,  the dual of problem $\mathcal A_h$ takes the form:
\begin{align}
	\min_{(m,y,z, \tilde y, \tilde z) \in \R^{5N'}} &\left\{ h^2 \Delta t  \sum_{n=1}^{N_T} \sum_{i,j=0}^{N_h} \tilde L_h(x_{i,j}, m_{i,j}^n, y_{i,j}^n, z_{i,j}^n, \tilde y_{i,j}^n, \tilde z_{i,j}^n) + h^2 \sum_{i,j=0}^{N_h}  m_{i,j}^{N_T} {\widetilde u}_{i,j}^{N_T} \right\} \label{eq:pbAdual2} \\
	&\hbox{ subject to \eqref{eq:FPdiscrete}.}\notag
\end{align}
Let us now compute an equivalent expression for $\tilde L_h$. We note $\mathcal D_{\tilde L_h}$ the domain where, for any $x \in \TT_h^2$, $\tilde L_h(x,\cdot,\cdot,\cdot,\cdot,\cdot) < +\infty$, that is:
\begin{equation*}
	\mathcal D_{\tilde L_h} = \{(0,0,0,0,0)\} \cup \{(m,y,z,\tilde y, \tilde z) \in\R^5 ~:~ m>0, y \geq 0, z \leq 0, \tilde y \geq 0, \tilde z \leq 0\}.
\end{equation*}
For $x \in \TT_h^2$ and $(m,y,z,\tilde y, \tilde z) \in \mathcal D_{\tilde L_h}$, we have
\begin{align}
	\tilde L_h(x,m,y,z,\tilde y,\tilde z)
	= &\max_{a,b,c,\tilde b,\tilde c} \big\{ -(m, y, z, \tilde y, \tilde z)  \cdot (a,b,c,\tilde b,\tilde c) + K_h(x,a,b,c,\tilde b,\tilde c) \big\}
	\notag \\
	= &\max_{a,b,c,\tilde b,\tilde c} \big\{ -(m, y, z, \tilde y, \tilde z)  \cdot (a,b,c,\tilde b,\tilde c) + \min_{\mu \geq 0} \{\mu a + \mu H_h(x,\mu,b,c,\tilde b,\tilde c)\} \big\}
	\notag \\
	= &\max_{b,c,\tilde b,\tilde c} \max_{a} \min_{\mu \geq 0} \big\{ a (\mu-m) -(y, z, \tilde y, \tilde z)  \cdot (b,c,\tilde b,\tilde c) + \mu H_h(x,\mu,b,c,\tilde b,\tilde c) \big\}
	\notag \\
	= & \max_{b,c,\tilde b,\tilde c} \big\{ -(y, z, \tilde y, \tilde z)  \cdot (b,c,\tilde b,\tilde c)  + mH_h(x,m,b,c,\tilde b,\tilde c) \big\}. \label{eq:lastMaxLtilde}
\end{align}
where the last equality holds by Fenchel-Moreau Theorem (note that for all $x,b, c, \tilde b$ and $\tilde c$, $m \mapsto m H_h(x,m,b,c,\tilde b,\tilde c)$ is convex and l.s.c.).
\\
From~\eqref{eq:lastMaxLtilde}, we can  express $\tilde L_h(x,m,y,z,\tilde y,\tilde z)$ as follows:
\begin{equation*}
	\tilde L_h(x,m,y,z,\tilde y,\tilde z)
	=
	\begin{cases}
 	(\beta-1)\beta^{-\beta^*} \frac{ \left( y^2 + z^2 + \tilde y^2 + \tilde z^2 \right)^{\beta^*/2}} {m^{(\beta^*-1)(1-\alpha)}}  + m \ell(x,m)
&\hbox{ if $m>0, y \geq 0, z \leq 0, \tilde y \geq 0, \tilde z \leq 0$} \\
		0 &\hbox{ if $m=y=z=\tilde y=\tilde z=0$} \\
		+\infty &\hbox{ otherwise.}
	\end{cases}
\end{equation*}

\subsection{Augmented Lagrangian}
\label{sec:augmented-lagrangian}
Let us go back to the primal formulation of $\mathcal A_h$. Trying to  directly find a minimum of \eqref{eq:pbA2} is difficult because $\phi$
 is involved in both $\mathcal F_h$ and $\mathcal G_h$. Instead, we can artificially separate the arguments of  $\mathcal F_h$ and $\mathcal G_h$, by 
introducing a new argument $q$  of $\mathcal G_h$ and adding the constraint that  $q=\Lambda_h(\phi)$. With this approach, the problem 
becomes
\begin{align}
	\mathcal A_h 
	= - &\inf_{\phi \in \R^N} \inf_{q \in \R^{5N'}}\big\{ \mathcal F_h(\phi) + \mathcal G_h (q) \big\} \notag \\
	&\hbox{ subject to $q = \Lambda_h(\phi)$.} \label{eq:pbA3}
\end{align}
The Lagrangian corresponding to this constrained optimization problem is
\begin{equation}\label{eq:discreteL}
	\mathcal L_h (\phi, q, \sigma) = \mathcal F_h(\phi) + \mathcal G_h(q) - \langle \sigma , (\Lambda_h(\phi) - q)  \rangle_{\ell^2(\R^{5N'})}
\end{equation}
where $\sigma$ is the dual variable corresponding to the constraint in \eqref{eq:pbA3}.

Finding a minimizer of \eqref{eq:pbA3} is equivalent to finding a saddle-point of $\mathcal L_h$, so the goal is now to obtain $(\phi, q, \sigma)$ achieving
\begin{equation*}
	\inf_{\phi\in\R^N} \inf_{q\in\R^{5N'}} \sup_{\sigma\in\R^{5N'}} \mathcal L_h (\phi, q, \sigma).
\end{equation*}

We are going to use an alternating direction algorithm based on an augmented Lagrangian. Augmented Lagrangian algorithms consist of adding  a penalty term to the Lagrangian (whereas penalty methods add a penalty term to the objective), and solving a sequence of unconstrained optimization problems. They were first discussed in~\cite{MR0271809,MR0272403} and later in~\cite{FortinGlowinski2000augmented}. We will see that, under appropriate assumptions, the algorithm produces a sequence that converges to the solution of the original constrained problem, for every choice of a positive penalty parameter.
Therefore, unlike penalty methods, it is not necessary to have the penalty parameter tend to infinity in order to obtain the solution of the original constrained problem.

 For $r>0$, we introduce the augmented Lagrangian:
\begin{align}\label{eq:defAugLag}
	\mathcal L_h^r (\phi, q, \sigma) = \mathcal L_h (\phi, q, \sigma) + \frac r 2 \left|\left| \Lambda_h(\phi) - q \right|\right|_{\ell^2(\R^{5N'})}^2.
\end{align}
Note that the saddle-points of $\mathcal L_h$ and $\mathcal L_h^r$ are the same (see e.g. \cite{FortinGlowinski2000augmented}, \cite{GabayMercier1976dual}).
\\
In the sequel, we propose an Alternating Direction Method of Multipliers (ADMM) based on $	\mathcal L_h^r$.

\subsection{Alternating Direction Method of Multipliers for $\mathcal L_h^r$}\label{subsec:algoADMM}
\paragraph{Assumption} 
Hereafter, we take $\nu=0$.  In other words, we focus on deterministic mean field type control problem.  We have already explained in the introduction that it is possible to address the case $\nu>0$ with the same method, with some additional difficulties that are dealt with in   \cite{andreev2016preconditioning} and that we do not wish to tackle in the present paper.

\medskip

For general considerations on augmented Lagrangians and Alternating Direction Method of Multipliers, the reader is referred  to 
\cite{eckstein2012augmented}.
The algorithm that we use is a variant of the  algorithm referred to as  {\sl ALG2}   in \cite{BenamouCarlier2015ALG2}, 
a terminology used initially by Fortin and Glowinski in   \cite{FortinGlowinski2000augmented}, Chapter 3, Section~3,
to distinguish between  two possible  alternating direction methods of multipliers (ADMM).  As explained in~\cite{FortinGlowinski2000augmented}, Chapter~3, Remark~3.5, an iteration  of  {\sl ALG2}  is cheaper than one iteration of the other algorithm, namely {\sl ALG1}. This explains our choice of {\sl ALG2}.\\
The ADMM constructs a sequence of approximations of the solution, and  each iteration  is split into three steps.
 For simplicity, we will note $q = (a,b,c, \tilde b, \tilde c)\in\RR^{5N'}$ and $\sigma = (m,y,z, \tilde y, \tilde z)\in\RR^{5N'}$. 
Starting from an initial guess $(\phi^0, q^0, \sigma^0)$, we generate  a sequence indexed by $k\geq 0$:
\begin{align}
	& \ds \phi^{k+1} \in \argmin_{\phi \in \R^N} \Big\{
		\mathcal F_h(\phi) - \langle \sigma^k , \Lambda_h(\phi)  \rangle_{\ell^2(\R^{5N'})}  + \frac r 2 \left|\left| \Lambda_h(\phi) - q^k \right|\right|_{\ell^2(\R^{5N'})}^2
	 \Big\}, \label{eq:update-phi}\\
	& \ds q^{k+1} \in \argmin_{q \in \R^{5N'}} \Big\{
		 \mathcal G_h(q) + \langle \sigma^k , q  \rangle_{\ell^2(\R^{5N'})}  + \frac r 2 \left|\left| \Lambda_h(\phi^{k+1}) - q \right|\right|_{\ell^2(\R^{5N'})}^2
	 \Big\}, \label{eq:update-q}\\
	 &\ds \sigma^{k+1} = \sigma^k - r \Big(\Lambda_h(\phi^{k+1}) - q^{k+1} \Big). \label{eq:update-sigma}
\end{align}
The first two equations are proximal problems, whereas the last one is an explicit update. The link between proximal problems, ADMM and augmented Lagrangians is well known (see for instance~\cite{BoydParikh2014proximal}, Section 4 of Chapter 4).
We detail below how to implement this algorithm in our case. Updating $\phi$ and $q$ will be done respectively by solving a  boundary value problem involving a discrete 
version of partial differential equation for  $\phi$ and by reducing the proximal problem for $q$ to a single equation in $\RR_+$ (see~\eqref{eq:finalEqMu}) at each grid node.

 \begin{remark}\label{rem:convergenceADMM}
 As explained for instance in~\cite{EcksteinBertsekas1992douglas}, the augmented Lagrangian ADMM is a special case of the Douglas-Rachford splitting method for finding the zeros of the sum of two maximal monotone operators. As a consequence, the convergence of our ADMM algorithm holds since the following two conditions are satisfied (see Theorem 8 in~\cite{EcksteinBertsekas1992douglas}, following the contributions of~\cite{GabayMercier1976dual,LionsMercier1979splitting}; see also Section~5 of Chapter~3 in \cite{FortinGlowinski2000augmented}).
 \begin{enumerate}
 	\item $\Lambda_h$ has full column rank when it is considered on the space $\{ \phi \in \R^{5N} \,:\, \phi^{N_T} \equiv \widetilde u^{N_T}\}$.
 Indeed, $\Lambda_h$ is a discrete time-space gradient operator, so it is injective over the space of functions with fixed final values.
	\item  There exists a pair $(\phi, \sigma)$ that satisfies the following primal-dual extremality relations: 
		\begin{equation*}
			\Lambda_h^* \sigma \in \partial \mathcal F_h(\phi), \quad -\sigma \in \partial \mathcal G_h(\Lambda_h \phi)
		\end{equation*}
		where $\Lambda_h^*$ is defined by \eqref{eq:defLambdastar} in Section~\ref{sec:discreteDuality},
 and $\partial \mathcal F_h, \partial \mathcal G_h$ denote the subdifferentials of $\mathcal F_h$ and $\mathcal G_h$. 
 This is equivalent to the  existence of a solution to the discrete problem, 
which can be obtained as in ~\cite{AchdouCamilliCapuzzoDolcetta2012Planning}, Section 3.1. 
\end{enumerate}
By~\eqref{eq:update-q}, $\sigma^{k} - r(\Lambda_h(\phi^{k+1}) - q^{k+1}) \in  \partial \mathcal G_h(q^{k+1})$,
 and \eqref{eq:update-sigma} means that $\sigma^{k+1}\in \partial \mathcal G_h(q^{k+1})$ for every $k$.
\\
Furthermore, we see that $\mathcal G_h(q^{k+1}) <+\infty$ implies that $\mathcal G_h^*(\sigma^{k+1}) = \sigma^{k+1} \cdot q^{k+1} - \mathcal G_h (q^{k+1})$. So $\sigma^{k+1}$ is in the domain of $\mathcal G_h^*$, which implies, by~\eqref{eq:GstarLbar}, $\tilde L_h(x_{i,j},(\sigma^{k+1})_{i,j}^n) < +\infty$ for every $(i,j,n)$. In particular $(m^{k+1})_{i,j}^n \geq 0$, and $(z^{k+1})_{i,j}^n = 0$ whenever $(m^{k+1})_{i,j}^n = 0$. This explains why ALG2 gives consistent results even if the density vanishes, as already observed in~\cite{BenamouBrenier2000monge}.
 \end{remark}
When $k$ tends to $+\infty$, $\Lambda_h(\phi^k)$ and $q^k$ converge to the same limit. Moreover, the increment  $\sigma^{k+1} -\sigma^k $ of the  dual variable
 is the difference between these two terms (scaled by $r$). For the numerical convergence criteria, besides the error between $\Lambda_h(\phi^k)$ and $q^k$, we will use the $\ell^2$ norm of the 
residuals of the discrete versions of the HJB equation, (see \S~\ref{subsec:convergenceCrit}).
\\
We now give some more details on the three different steps:
 
\paragraph{Step 1 : update of $\phi$ :}
To alleviate the notations, let us drop the superscript $k$ and note in this step $q = (a,b,c, \tilde b, \tilde c) = (a^k,b^k,c^k, \tilde b^k, \tilde c^k) = q^k$ and $\sigma = (m,y,z, \tilde y, \tilde z) = (m^k,y^k,z^k,\tilde y^k,\tilde z^k) = \sigma^k$. 
Note that  $\mathcal F_h(\phi)$ is finite   if and only if   $\phi^{N_T} = \tilde u^{N_T}$.
We are therefore looking for  $\phi \in \R^N$ satisfying:
\begin{enumerate}
	\item $\forall \, i,j = 0, \dots, N_h-1, \quad \phi_{i,j}^{N_T} = {\widetilde u}_{i,j}^{N_T}$,
	\item for any $\psi \in \R^N$
\begin{equation}\label{eq:cond2phi}
	0 = D \mathcal F_h(\phi) \psi - \sigma \cdot D\Lambda_h(\phi)\psi + r \left(\Lambda_h(\phi) - q\right) \cdot  D\Lambda_h(\phi)\psi. 
\end{equation}
\end{enumerate}
 If  $\phi$ satisfies the first condition, then the second condition~\eqref{eq:cond2phi} can be written as follows:
\begin{align*}
	0
	=
	& -h^2 \sum_{i,j=0}^{N_h-1} {\widetilde m}_{i,j}^0 \psi_{i,j}^0 \\
	&+ h^2 \Delta t \sum_{i,j=0}^{N_h-1} \left(r \frac{{\widetilde u}_{i,j}^{N_T} - \phi_{i,j}^{N_T-1}}{\Delta t} - ra_{i,j}^{N_T}  - m_{i,j}^{N_T}\right) \frac{ - \psi_{i,j}^{N_T-1}}{\Delta t} \\
	&+ h^2 \Delta t \sum_{n=1}^{N_T-1} \sum_{i,j=0}^{N_h-1} \left(r \frac{\phi_{i,j}^{n} - \phi_{i,j}^{n-1}}{\Delta t} - ra_{i,j}^{n}  - m_{i,j}^{n}\right) \frac{\psi_{i,j}^n - \psi_{i,j}^{n-1}}{\Delta t} \\
	& +h^2 \Delta t \sum_{n=1}^{N_T} \sum_{i,j=0}^{N_h-1} \left(r\frac{\phi_{i+1,j}^{n-1} - \phi_{i,j}^{n-1}}{h} - rb_{i,j}^{n}  - y_{i,j}^{n}\right) \frac{\psi_{i+1,j}^{n-1} - \psi_{i,j}^{n-1}}{h} \\
	& +h^2 \Delta t \sum_{n=1}^{N_T} \sum_{i,j=0}^{N_h-1} \left(r\frac{\phi_{i,j}^{n-1} - \phi_{i-1,j}^{n-1}}{h} - rc_{i,j}^{n}  - z_{i,j}^{n}\right) \frac{\psi_{i,j}^{n-1} - \psi_{i-1,j}^{n-1}}{h} \\
	& +h^2 \Delta t \sum_{n=1}^{N_T} \sum_{i,j=0}^{N_h-1} \left(r\frac{\phi_{i,j+1}^{n-1} - \phi_{i,j}^{n-1}}{h} - r\tilde b_{i,j}^{n}  - \tilde y_{i,j}^{n}\right) \frac{\psi_{i,j+1}^{n-1} - \psi_{i,j}^{n-1}}{h} \\
	& +h^2 \Delta t \sum_{n=1}^{N_T} \sum_{i,j=0}^{N_h-1} \left(r\frac{\phi_{i,j}^{n-1} - \phi_{i,j-1}^{n-1}}{h} - r \tilde c_{i,j}^{n}  - \tilde z_{i,j}^{n}\right) \frac{\psi_{i,j}^{n-1} - \psi_{i,j-1}^{n-1}}{h}.
\end{align*}
By discrete integration by parts and periodicity, the right hand side can be written as follows:
\begin{align*}
	\,&h^2 \Delta t \sum_{i,j=0}^{N_h-1} \left( r\frac{2\phi_{i,j}^{N_T-1}- {\widetilde u}_{i,j}^{T} - \phi_{i,j}^{N_T-2}}{(\Delta t)^2} - r\frac{a_{i,j}^{N_T-1} - a_{i,j}^{N_T}}{\Delta t} - \frac{m_{i,j}^{N_T-1} - m_{i,j}^{N_T}}{\Delta t} \right) \psi_{i,j}^{N_T-1} \\
	&+ h^2 \Delta t \sum_{i,j=0}^{N_h-1} \left( r\frac{-\phi_{i,j}^{1} + \phi_{i,j}^{0}}{(\Delta t)^2} + r\frac{a_{i,j}^{1}}{\Delta t} + \frac{m_{i,j}^{1}}{\Delta t} - \frac{m_{i,j}^0}{\Delta t}\right) \psi_{i,j}^{0} \\
	&+ h^2 \Delta t \sum_{n=1}^{N_T-2}\sum_{i,j=0}^{N_h-1} \left( r\frac{2 \phi_{i,j}^{n} - \phi_{i,j}^{n-1} - \phi_{i,j}^{n+1}}{(\Delta t)^2} - r\frac{a_{i,j}^{n} - a_{i,j}^{n+1}}{\Delta t} - \frac{m_{i,j}^{n} - m_{i,j}^{n+1}}{\Delta t} \right) \psi_{i,j}^{n} \\
	& + h^2 \Delta t \sum_{n=0}^{N_T-1} \sum_{i,j=0}^{N_h-1} \Big(2r\frac{4\phi_{i,j}^{n} - \phi_{i-1,j}^{n} - \phi_{i+1,j}^{n}  - \phi_{i,j-1}^{n} - \phi_{i,j+1}^{n}}{h^2}  \\
	&\qquad\qquad\qquad - r\frac{b_{i-1,j}^{n+1} - b_{i,j}^{n+1}}{h}  - r\frac{c_{i,j}^{n+1} - c_{i+1,j}^{n+1}}{h} - r\frac{\tilde b_{i,j-1}^{n+1} - \tilde b_{i,j}^{n+1}}{h}  - r\frac{\tilde c_{i,j}^{n+1} - \tilde c_{i,j+1}^{n+1}}{h} \\
	&\qquad\qquad\qquad  -  \frac{y_{i-1,j}^{n+1} - y_{i,j}^{n+1}}{h} - \frac{z_{i,j}^{n+1} - z_{i+1,j}^{n+1}}{h}  -  \frac{\tilde y_{i,j-1}^{n+1} - \tilde y_{i,j}^{n+1}}{h} - \frac{\tilde z_{i,j}^{n+1} - \tilde z_{i,j+1}^{n+1}}{h}  \Big) \psi_{i,j}^{n}.
\end{align*}

We deduce that $\phi$ must satisfy the finite difference equation~:
$\forall i,j=0, \dots, N_h-1$,
\begin{align}
	&r\frac{2\phi_{i,j}^{N_T-1} - \phi_{i,j}^{N_T-2}}{(\Delta t)^2} + 2r\frac{4\phi_{i,j}^{N_T-1} - \phi_{i-1,j}^{N_T-1} - \phi_{i+1,j}^{N_T-1}  - \phi_{i,j-1}^{N_T-1} - \phi_{i,j+1}^{N_T-1}}{h^2} 
	\notag \\
	=\, & r\frac{{\widetilde u}_{i,j}^{N_T}}{(\Delta t)^2}
	\notag \\
	&-\left( \frac{m_{i,j}^{N_T} - m_{i,j}^{N_T-1}}{\Delta t}  + \frac{y_{i,j}^{N_T} - y_{i-1,j}^{N_T}}{h} + \frac{z_{i+1,j}^{N_T} - z_{i,j}^{N_T}}{h}  + \frac{\tilde y_{i,j}^{N_T} - \tilde y_{i,j-1}^{N_T}}{h} + \frac{\tilde z_{i,j+1}^{N_T} - \tilde z_{i,j}^{N_T}}{h} \right)
	\notag \\
	&- r\left( \frac{a_{i,j}^{N_T} - a_{i,j}^{N_T-1}}{\Delta t} + \frac{b_{i,j}^{N_T} - b_{i-1,j}^{N_T}}{h}  + \frac{c_{i+1,j}^{N_T} - c_{i,j}^{N_T}}{h} + \frac{\tilde b_{i,j}^{N_T} - \tilde b_{i,j-1}^{N_T}}{h}  + \frac{\tilde c_{i,j+1}^{N_T} - \tilde c_{i,j}^{N_T}}{h} \right)
	\label{eq:step1eqPhitT}
\end{align}
and $\forall n = 1, \dots, N_T-2, \forall i,j=0, \dots, N_h-1$,
\begin{align}
	&r\frac{2 \phi_{i,j}^{n} - \phi_{i,j}^{n-1} - \phi_{i,j}^{n+1}}{(\Delta t)^2} +2 r\frac{4\phi_{i,j}^{n} - \phi_{i-1,j}^{n} - \phi_{i+1,j}^{n}  - \phi_{i,j-1}^{n} - \phi_{i,j+1}^{n}}{h^2}
	\notag \\
	& = -\left( \frac{m_{i,j}^{n+1} - m_{i,j}^{n}}{\Delta t} +  \frac{y_{i,j}^{n+1} - y_{i-1,j}^{n+1}}{h} + \frac{z_{i+1,j}^{n+1} - z_{i,j}^{n+1}}{h} +  \frac{\tilde y_{i,j}^{n+1} - \tilde y_{i,j-1}^{n+1}}{h} + \frac{\tilde z_{i,j+1}^{n+1} - \tilde z_{i,j}^{n+1}}{h} \right)
	\notag \\
	&\qquad -r\left( \frac{a_{i,j}^{n+1} - a_{i,j}^{n}}{\Delta t} + \frac{b_{i,j}^{n+1} - b_{i-1,j}^{n+1}}{h} + \frac{c_{i+1,j}^{n+1} - c_{i,j}^{n+1}}{h}  + \frac{\tilde b_{i,j}^{n+1} - \tilde b_{i,j-1}^{n+1}}{h} + \frac{\tilde c_{i,j+1}^{n+1} - \tilde c_{i,j}^{n+1}}{h}\right)
	\label{eq:step1eqPhi}
\end{align}
with periodic boundary conditions and the condition  at $t=0$: $\forall i=1,\dots, N_h$
\begin{align}
 & r\frac{\phi_{i,j}^{0} - \phi_{i,j}^{1}}{(\Delta t)^2}  +2 r\frac{4\phi_{i,j}^{0} - \phi_{i-1,j}^{0} - \phi_{i+1,j}^{0}  - \phi_{i,j-1}^{0} - \phi_{i,j+1}^{0}}{h^2}
 \notag \\
 =  &\, -\left(\frac{m_{i,j}^{1} - m_{i,j}^0}{\Delta t} +   \frac{y_{i,j}^{1} - y_{i-1,j}^{1}}{h} + \frac{z_{i+1,j}^{1} - z_{i,j}^{1}}{h} +  \frac{\tilde y_{i,j}^{1} - \tilde y_{i,j-1}^{1}}{h} + \frac{\tilde z_{i,j+1}^{1} - \tilde z_{i,j}^{1}}{h} \right)
 \notag \\
&-r\left( \frac{a_{i,j}^{1}}{\Delta t}  + \frac{b_{i,j}^{1} - b_{i-1,j}^{1}}{h} + \frac{c_{i+1,j}^{1} - c_{i,j}^{1}}{h}  + \frac{\tilde b_{i,j}^{1} - \tilde b_{i,j-1}^{1}}{h} + \frac{\tilde c_{i,j+1}^{1} - \tilde c_{i,j}^{1}}{h}\right).
	\label{eq:step1eqPhit0}
\end{align}
\begin{remark}
The system (\ref{eq:step1eqPhitT})-(\ref{eq:step1eqPhit0}) is the discrete version of a boundary value problem 
in $(0,T)\times \T^2$, involving a second order linear elliptic  partial differential equation.
\end{remark}

In our implementation,  (\ref{eq:step1eqPhitT})-(\ref{eq:step1eqPhit0}) is solved by using BiCGStab iterations, see  \cite{MR1149111}.

\paragraph{Step 2 : update of $q = (a,b,c, \tilde b, \tilde c)$ :}
To alleviate the notations, let us note in this step $\phi = \phi^{k+1}$ and $\sigma = (m,y,z, \tilde y, \tilde z) = (m^k,y^k,z^k,\tilde y^k,\tilde z^k) = \sigma^k$. Then we are looking for $q \in \R^{5N}$ satisfying:
\begin{align*}
	q  \, \in \, &\argmin_{q \in \R^{5N'}} \Big\{
		\mathcal G_h(q) + \langle \sigma , q \rangle_{\ell^2(\R^{5N'})}  + \frac r 2 \left|\left| \Lambda_h(\phi) - q \right|\right|_{\ell^2(\R^{5N'})}^2
	 \Big\},
\end{align*}
where $q \mapsto \mathcal G_h(q) - \langle \sigma , q \rangle_{\ell^2}  + \frac r 2 \left|\left| \Lambda_h(\phi) - q \right|\right|_{\ell^2}^2$ is convex.
This amounts to solving  a five-dimensional optimization problem at each grid node, i.e. for each $n \in \{1, \dots, N_T\}, i,j \in \{0, \dots, N_h-1\}$,
to finding $(a,b,c, \tilde b, \tilde c) \in \R^5$ that minimizes
\begin{align}
	&-K_h(x_{i,j}, a, b, c, \tilde b, \tilde c) + m_{i,j}^n a + y_{i,j}^n b + z_{i,j}^n c  + \tilde y_{i,j}^n \tilde b + \tilde z_{i,j}^n \tilde c \notag\\
	&+ \frac r 2 \left[ \left(\frac{\phi_{i,j}^n - \phi_{i,j}^{n-1}}{\Delta t}  - a\right)^2 + \left( \frac{\phi_{i+1,j}^{n-1} - \phi_{i,j}^{n-1}}{h} - b\right)^2 + \left( \frac{\phi_{i,j}^{n-1} - \phi_{i-1,j}^{n-1}}{h}   - c\right)^2 \right.\notag\\
	&\qquad\qquad + \left.\left( \frac{\phi_{i,j+1}^{n-1} - \phi_{i,j}^{n-1}}{h} - \tilde b\right)^2 + \left( \frac{\phi_{i,j}^{n-1} - \phi_{i,j-1}^{n-1}}{h}   - \tilde c\right)^2\right]. \label{eq:tmpMinPbK}
\end{align}
This task is not trivial because the definition of $K_h$ itself involves a minimization. However, we can simplify the expression~\eqref{eq:tmpMinPbK} as follows: we notice that, for all $x\in\TT_h^2$, and all $\sigma' \in \mathcal D_{\tilde L_h}$,
\begin{equation*}
	\tilde L_h(x, \sigma') = \max_{q \in \RR^5}\Big\{ - \sigma' \cdot q  + K_h(x,q)\Big\} = (-K_h)^*(x,-\sigma').
\end{equation*}
Hence,  by Fenchel-Moreau's theorem, for any $q \in \R^{5}$,
\begin{equation*}
	-K_h(x,q) = \max_{\sigma'\in \R^5} \Big\{ -  \sigma' \cdot q - \tilde L_h(x, \sigma')\Big\}.
\end{equation*}
Note that any maximizer $\sigma' = (\mu,\eta,\zeta,\tilde\eta,\tilde\zeta)$ should be in $\mathcal D_{\tilde L_h}$, that is, should satisfy either $(\mu,\eta,\zeta,\tilde\eta,\tilde\zeta) = (0,0,0,0,0)$, or $\mu>0, \eta \geq 0, \zeta \leq 0, \tilde \eta\geq 0$ and $\tilde\zeta\leq 0$.
Plugging this into \eqref{eq:tmpMinPbK} leads us to the following saddle-point problem:
\begin{equation}\label{pb:infmax1}
	\inf_{q \in \RR^5} \max_{\sigma' \in \mathcal D_{\tilde L_h}} U(q,\sigma')
\end{equation}
where
$$
	U(q,\sigma') = (\sigma_{i,j}^n - \sigma') \cdot q - \tilde L_h(x_{i,j}, \sigma')  + \frac r 2 \left|\left| \Lambda_h(\phi)_{i,j}^n - q \right|\right|_{2}^2,
$$
is concave in $\sigma' = (\mu,\eta,\zeta,\tilde \eta,\tilde \zeta)$, and convex in $q = (a,b,c, \tilde b, \tilde c)$. 
 The following lemma allows us  to swap the inf and the max in the expression above:
\begin{lemma}\label{lem:swapInfMax}
	$$
	\inf_{q \in \RR^5} \sup_{\sigma' \in \R^5
        } U(q,\sigma')
	=
	\sup_{\sigma' \in \R ^5 
} \inf_{q \in \RR^5} U(q,\sigma').
	$$
\end{lemma}
\begin{proof}
Let us show that  Corollary~37.1.3 in~\cite{MR1451876} can be applied. For simplicity, we note
\begin{equation*}
	U(q,\sigma') = -\sigma' \cdot q + \sigma_{i,j}^n \cdot q + \frac r 2 \left|\left| \Lambda_h(\phi)_{i,j}^n - q \right|\right|_{2}^2 -  \tilde L_h(x_{i,j},\sigma')
\end{equation*}
 (recall that $\phi_{i,j}^n$ and $\sigma_{i,j}^n  = (m_{i,j}^n ,y_{i,j}^n ,z_{i,j}^n , \tilde y_{i,j}^n , \tilde z_{i,j}^n )$ are fixed in this step). Then for $(\lambda,\theta) \in \R^5 \times \R^5$, $\underline U^*$ is defined as in \cite{MR1451876} by
\begin{align*}
	\underline U^*(\lambda, \theta) &= \sup_q \inf_{\sigma'} \left\{ q \cdot \lambda + \sigma' \cdot \theta - U(q,\sigma') \right\}.
\end{align*}
We first remark that  $\underline U^* (\lambda, \theta)<+\infty$. Indeed,  $\tilde L_h(x,0) = 0$ for all $x$, and  we deduce that for any $q$,
\begin{align*}
	\inf_{\sigma'} \left\{ q \cdot \lambda + \sigma' \cdot \theta - U(q,\sigma') \right\} 
	&= \inf_{\sigma'} \left\{ q \cdot \lambda +  \sigma' \cdot \theta + \sigma' \cdot q - \sigma_{i,j}^n \cdot q - \frac r 2 \left|\left| \Lambda_h(\phi)_{i,j}^n - q \right|\right|_{2}^2 +  \tilde L_h(x_{i,j},\sigma') \right\} \\
	&\leq  q \cdot( \lambda - \sigma_{i,j}^n) - \frac r 2 \left|\left| \Lambda_h(\phi)_{i,j}^n - q \right|\right|_{2}^2,
\end{align*}
and the supremum over $q$ of this last quantity is finite.\\
 Moreover $\underline U^*(\lambda, \theta)>-\infty$. Indeed, recall that from  H2,
 for every $x_{i,j}$, $\tilde L_h(x_{i,j}, \cdot)$ is bounded from below  by  $c= \inf_{m\ge 0} ( \frac {m^q} {C_1}   -C_1 m)>-\infty$. So, for $q = -\theta$,
\begin{align*}
	\inf_{\sigma'} \left\{  q \cdot \lambda + \sigma' \cdot \theta - U(q,\sigma') \right\} 
	&= - \theta \cdot (\lambda -\sigma_{i,j}^n) - \frac r 2 \left|\left| \Lambda_h(\phi)_{i,j}^n + \theta \right|\right|_{2}^2 + 
 \inf_{\sigma'} \left\{ \tilde L_h(x_{i,j},\sigma') \right\}\\
	&\geq - \theta \cdot (\lambda -\sigma_{i,j}^n) - \frac r 2 \left|\left| \Lambda_h(\phi)_{i,j}^n + \theta \right|\right|_{2}^2 + c,
\end{align*}
hence the supremum over $q$  is also bounded from below by this last term, which is finite.
\\
Therefore, we can apply Corollary~37.1.3 of~\cite{MR1451876} and get the conclusion.
\end{proof}
From Lemma~\ref{lem:swapInfMax}, we obtain that the problem \eqref{pb:infmax1} is equivalent to
\begin{align*}
	\max_{\sigma' \in \mathcal D_{\tilde L_h}} \min_{q\in\RR^5}  \left\{ q \cdot (\sigma_{i,j}^n - \sigma')  + \frac r 2 \left\| \Lambda_h(\phi)_{i,j}^n - q \right\|_2^2
			- \tilde L_h(x_{i,j}, \sigma')\right\}.
\end{align*}
Considering the minimization, the first order optimality conditions give, for $\sigma' \in \mathcal D_{\tilde L_h}$:
\begin{align}
 q = \frac 1 r (\sigma' - \sigma_{i,j}^n) + (\Lambda_h \phi)_{i,j}^n. \label{eq:qFctSigmaPrime}
 \end{align}
Using the expression of $q = (a,b,c, \tilde b, \tilde c)$ as a function of $\sigma' = (\mu,\eta,\zeta,\tilde \eta,\tilde \zeta)$, the saddle-point problem takes the form:
\begin{equation}\label{eq:mainMaxPbStep2}
	\max_{\sigma' \in \mathcal D_{\tilde L_h}}  W(\sigma'),
\end{equation}
with $W(\sigma') =  -\frac1{2r}\|\sigma' - \sigma_{i,j}^n\|_2^2 + (\Lambda_h \phi)_{i,j}^n \cdot (\sigma_{i,j}^n - \sigma') - \tilde L_h(x_{i,j}, \sigma')$, which implicitly depends on the point $(i,j,n)$ under consideration.
\\
Assume that the maximum is attained for some $\sigma' \in \mathcal D_{\tilde L_h} \setminus \{0\}$. Then, the first order conditions for the maximization give (noting  $\Lambda_h^{(p)}$ the $p$-th coordinate of $\Lambda_h$, $p \in \{1,\dots,5\}$):
 \begin{align}
   &\ds  -\frac1{r}(\mu - m_{i,j}^n) - (\Lambda_h^{(1)}\phi)_{i,j}^n - \p_\mu \tilde L_h(x_{i,j}, \mu, \eta, \zeta, \tilde \eta, \tilde \zeta)= 0\notag\\
& \Leftrightarrow \quad \ds \mu - m_{i,j}^n +r (\Lambda_h^{(1)}\phi)_{i,j}^n - \frac{r \beta^{-\beta^*}(1-\alpha)} {\mu^{(\beta^*-1)(1-\alpha)+1}}(\eta^2 + \zeta^2 + \tilde \eta^2 + \tilde \zeta^2)^{\beta^*/2}  + r \ell(x_{i,j},\mu) + r \mu \p_\mu \ell(x_{i,j},\mu) = 0, \label{eq:firsteqmu}\\
&	\ds  \min\left(\eta,  \frac1{r}(\eta - y_{i,j}^n)  + (\Lambda_h^{(2)}\phi)_{i,j}^n + \p_\eta \tilde L_h(x_{i,j}, \mu, \eta, \zeta, \tilde \eta, \tilde \zeta)\right) = 0  \notag\\
& \Leftrightarrow \quad \ds
\min\left(\eta,  -y_{i,j}^n + r (\Lambda_h^{(2)}\phi)_{i,j}^n +
\eta  + \frac{r \beta^{1-\beta^*}} {\mu^{(\beta^*-1)(1-\alpha)}}(\eta^2 + \zeta^2 + \tilde \eta^2 +
 \tilde \zeta^2)^{\beta^*/2 - 1} \eta \right)=0,
\label{eq:firsteqeta}\\
& \notag\\
& \ds  \max\left ( \zeta, 
   \frac1{r}(\zeta - z_{i,j}^n)  + (\Lambda_h^{(3)}\phi)_{i,j}^n +\p_\zeta \tilde L_h(x_{i,j}, \mu, \eta, \zeta, \tilde \eta, \tilde \zeta) \right)= 0  \notag\\
&\Leftrightarrow \quad	 \ds     \max\left( \zeta,   -z_{i,j}^n +r (\Lambda_h^{(3)}\phi)_{i,j}^n
+\zeta + \frac{r \beta^{1-\beta^*}} {\mu^{(\beta^*-1)(1-\alpha)}}(\eta^2 + \zeta^2 + \tilde \eta^2 + \tilde \zeta^2)^{\beta^*/2 - 1} \zeta \right)
=0,
\label{eq:firsteqzeta}\\
& \notag\\ \ds   
& \ds  \min\left(\tilde \eta, \frac1{r}(\tilde \eta - \tilde y_{i,j}^n)  + (\Lambda_h^{(4)}\phi)_{i,j}^n+ \p_{\tilde\eta} \tilde L_h(x_{i,j}, \mu, \eta, \zeta, \tilde \eta, \tilde \zeta) \right) =0  \notag\\
&\Leftrightarrow \quad	
	 \ds  \min\left(\tilde \eta,     -\tilde y_{i,j}^n + r (\Lambda_h^{(4)}\phi)_{i,j}^n+ \tilde\eta + \frac{r \beta^{1-\beta^*}} {\mu^{(\beta^*-1)(1-\alpha)}}(\eta^2 + \zeta^2 + \tilde \eta^2 + \tilde \zeta^2)^{\beta^*/2 - 1} \tilde \eta \right)=0 ,
		\label{eq:firsteqetat}\\
& \notag\\
& \ds   \max\left ( \tilde \zeta,  \frac1{r}(\tilde \zeta - \tilde z_{i,j}^n)  + (\Lambda_h^{(5)}\phi)_{i,j}^n + \p_{\tilde\zeta} \tilde L_h(x_{i,j}, \mu, \eta, \zeta, \tilde \eta, \tilde \zeta) \right)= 0  \notag\\
& \Leftrightarrow \quad	 \ds    \max\left ( \tilde \zeta, - \tilde z_{i,j}^n +r (\Lambda_h^{(5)}\phi)_{i,j}^n   +\tilde \zeta + \frac{r \beta^{1-\beta^*}} {\mu^{(\beta^*-1)(1-\alpha)}}(\eta^2 + \zeta^2 + \tilde \eta^2 + \tilde \zeta^2)^{\beta^*/2 - 1} \tilde \zeta \right)=0.
		\label{eq:firsteqzetat}
 \end{align}
We can therefore express $\eta,\zeta,\tilde\eta,\tilde\zeta$ as functions of $\mu$: let us define $\Sigma(\mu) = \eta^2 + \zeta^2 + \tilde \eta^2 + \tilde \zeta^2$ and
$$
 \chi(\mu) = \left( \mu - m_{i,j}^n + r (\Lambda_h^{(1)}\phi)_{i,j}^n + r \ell(x_{i,j},\mu) + r \mu \p_\mu \ell(x_{i,j},\mu) \right)\frac {\mu^{(\beta^*-1)(1-\alpha)+1}} {r \beta^{-\beta^*}(1-\alpha)}.
 $$
Although $\Sigma(\mu)$ and $\chi (\mu)$ depend on $(i,j,n)$,  we drop these indices for simplicity. 
Let us define $P_\chi = \{ \mu >0 \,:\, \chi(\mu) \geq 0\}$.
From \eqref{eq:firsteqmu} we obtain that  for any $\mu \in P_\chi$,
\begin{align}\label{eq:SigmaTmp}
 	 \Sigma(\mu) = \left(\chi(\mu)\right)^{2/\beta^*}.
\end{align}
 Moreover \eqref{eq:firsteqeta}--\eqref{eq:firsteqzetat}  yield
 \begin{eqnarray}
   \label{eq:1}
\ds 	\eta 	&=  \ds \frac{\left(y_{i,j}^n-r (\Lambda_h^{(2)}\phi)_{i,j}^n\right)^+}{1+\frac{r \beta^{1-\beta^*}} {\mu^{(\beta^*-1)(1-\alpha)}}\Sigma(\mu)^{\beta^*/2 - 1}},\quad\quad  	\tilde \eta 
	 &= \ds  \frac{\left( \tilde y_{i,j}^n - r (\Lambda_h^{(4)}\phi)_{i,j}^n \right)^+}{1+\frac{r \beta^{1-\beta^*}} {\mu^{(\beta^*-1)(1-\alpha)}}\Sigma(\mu)^{\beta^*/2 - 1}}, \\
\label{eq:2}
 	\zeta 
	&\ds = -\frac{\left( z_{i,j}^n - r (\Lambda_h^{(3)}\phi)_{i,j}^n\right)^- }{1+\frac{r \beta^{1-\beta^*}} {\mu^{(\beta^*-1)(1-\alpha)}}\Sigma(\mu)^{\beta^*/2 - 1}}, \quad\quad   	\tilde \zeta 
	&\ds = - \frac{\left( \tilde z_{i,j}^n - r (\Lambda_h^{(5)}\phi)_{i,j}^n \right)^-}{1+\frac{r \beta^{1-\beta^*}} {\mu^{(\beta^*-1)(1-\alpha)}}\Sigma(\mu)^{\beta^*/2 - 1}}.
 \end{eqnarray}
 Using (\ref{eq:1})-(\ref{eq:2}) and the definition of $\Sigma$, we find that $\mu \in P_\chi$ satisfies:
 \begin{equation}\label{eq:finalEqMu}
 	\Xi(\mu) = 0,
 \end{equation}
 where
 $$
 	\Xi(\mu) = \Sigma(\mu) \left(1+\frac{r \beta^{1-\beta^*}} {\mu^{(\beta^*-1)(1-\alpha)}}\chi(\mu)^{1 - 2/\beta^*}\right)^2  - \gamma_{i,j}^n,
 $$
 with $\Sigma(\mu)$ given by \eqref{eq:SigmaTmp} and
 $$
 	\gamma_{i,j}^n = \left(y_{i,j}^n - r (\Lambda_h^{(2)}\phi)_{i,j}^n\right)^+ -  \left(z_{i,j}^n - r (\Lambda_h^{(3)}\phi)_{i,j}^n\right)^- +  \left(\tilde y_{i,j}^n - r (\Lambda_h^{(4)}\phi)_{i,j}^n\right)^+ -  \left(\tilde z_{i,j}^n - r (\Lambda_h^{(5)}\phi)_{i,j}^n\right)^-.
 $$
 Equation~\eqref{eq:finalEqMu} involves only the unknown $\mu$. Let us show that it admits at most one solution in $P_\chi$.
 \begin{lemma}\label{lem:step-2-Pxi}
The function $\Xi$ is strictly increasing and  $P_\chi$ is a right-unbounded interval. 
 There exists at most one solution of (\ref{eq:finalEqMu}) in $ P_\chi$.
 \end{lemma}
 \begin{proof}
 For completeness we first study the function $\chi$.\\
 Let us note $N(\mu) = \left( \mu - m_{i,j}^n + r (\Lambda_h^{(1)}\phi)_{i,j}^n + r \ell(x_{i,j},\mu) + r \mu \p_\mu \ell(x_{i,j},\mu) \right)$, 
so  $\chi(\mu) = N(\mu) \frac {\mu^{(\beta^*-1)(1-\alpha)+1}} {r \beta^{-\beta^*}(1-\alpha)}$. 
We see that  $N'>0$ since, from H2, $\mu \mapsto \mu \ell(x_{i,j},\mu)$ is strictly convex. Therefore $\chi$ is a strictly increasing function.
 Moreover, from (\ref{eq:8}) and (\ref{eq:9}), $\chi(\mu) \to +\infty$ as $\mu \to +\infty$. Thus, $P_\chi$ is a right-unbounded interval.
There exists at most one number  $\mu_0\in P_\chi$ such that $\chi(\mu_0)=0$, and if it exists $P_\chi=[\mu_0, +\infty)$.
\\ 
We rewrite $\Xi$ using~\eqref{eq:SigmaTmp}:
 \begin{align*}
	 	\Xi(\mu) &= \chi(\mu)^{2/\beta^*} \left(1+\frac{r \beta^{1-\beta^*}} {\mu^{(\beta^*-1)(1-\alpha)}}\chi(\mu)^{1 - 2/\beta^*}\right)^2  - \gamma_{i,j}^n = \Theta(\mu)^2  - \gamma_{i,j}^n
\end{align*}
with the notation $
	\Theta(\mu) = \chi(\mu)^{1/\beta^*}+\frac{r \beta^{1-\beta^*}} {\mu^{(\beta^*-1)(1-\alpha)}}\chi(\mu)^{1 - 1/\beta^*}$.
Then, the derivative of $\Xi$ on $P_\chi$ is  $\Xi'(\mu) = 2 \Theta(\mu) \Theta'(\mu)$ with
\begin{align*}
	\Theta'(\mu) &= \frac 1 {\beta^*} \chi'(\mu)\chi(\mu)^{\frac{1}{\beta^*}-1} + r \beta^{1-\beta^*} \mu^{-(\beta^*-1)(1-\alpha)} \chi(\mu)^{-\frac{1}{\beta^*}}\left[ -(\beta^*-1)(1-\alpha) \mu^{-1}\chi(\mu) + \left(1-\frac{1}{\beta^*}\right) \chi'(\mu)\right].
\end{align*}
On the set $P_\chi$, $\chi \geq 0$ and $\Theta \geq 0$ with equality only at $\mu_0$ if it exists;
 moreover $\chi'>0$. Hence to obtain that $\Xi'>0$ on $P_\chi\setminus \{\mu_0\}$ , it remains to show that the last term in $\Theta'(\mu)$ is nonnegative.
 For any $\mu \in P_\chi$,
\begin{align*}
	&-(\beta^*-1)(1-\alpha) \mu^{-1}\chi(\mu) + \left(1-\frac{1}{\beta^*}\right) \chi'(\mu) \\
	= \,&\frac{\mu^{(\beta^*-1)(1-\alpha)}} {r \beta^{-\beta^*}(1-\alpha)} \left[ -(\beta^*-1)(1-\alpha) N(\mu) + \left(1-\frac{1}{\beta^*}\right) \Big[N'(\mu)\mu + N(\mu)((\beta^*-1)(1-\alpha)+1) \Big]\right] \\
	= \,&\frac{\mu^{(\beta^*-1)(1-\alpha)}} {r \beta^{-\beta^*}(1-\alpha)} \left[ \frac{\alpha}{\beta} N(\mu) + \left(1-\frac{1}{\beta^*}\right) N'(\mu)\mu\right] 
	> 0.
\end{align*}
Hence $\Xi'(\mu) \geq 0$ with equality only at $\mu_0$ if it exists. We conclude that  there is at most one solution to~\eqref{eq:finalEqMu} in $P_\chi$.
\end{proof}
Our algorithm to solve~\eqref{eq:mainMaxPbStep2} and find the maximizer of $W$ in $\mathcal D_{\tilde L_h}$ is therefore as follows:
 \begin{enumerate}
 	\item Look for a solution $\mu^*$ to~\eqref{eq:finalEqMu}: 
 \begin{enumerate}
 	\item Compute the left end $\underline\mu$  of $P_\chi$.
This is done as follows: first, check whether $\chi(0) \geq 0$. If yes, then set $\underline\mu = 0$. 
Otherwise, look for $\underline \mu=\mu_0$ by a bissection method in $[0, \tilde \mu]$ where $\tilde \mu$ is sufficiently large such that $\chi(\tilde \mu) >0$ 
($\tilde \mu$ exists since $\chi(\mu) \to +\infty$ as $\mu \to +\infty$). 
	\item Check whether $\Xi\left(\underline\mu\right) > 0$. If yes, then there is no solution to~\eqref{eq:finalEqMu} and the maximizer of~\eqref{eq:mainMaxPbStep2} is $\sigma' = (0,0,0,0,0)$. In this case, stop here.
	\item Otherwise, continue and compute $\mu^*$ solving~\eqref{eq:finalEqMu}.
 To do so, we use a bissection method in $[\underline \mu, \overline \mu]$, where $\overline \mu$ is large enough such that $\Xi\left(\overline\mu\right) >0$
 (this is possible because $\Xi\to +\infty$ as $\mu \to +\infty$).
 \end{enumerate} 
	\item Given $\mu^*$, compute $\eta^*, \zeta^*, \tilde\eta^*, \tilde\zeta^*$ given by~(\ref{eq:1})-(\ref{eq:2})
	\item The maximizer of~\eqref{eq:mainMaxPbStep2} is either $(\mu^*, \eta^*, \zeta^*, \tilde\eta^*, \tilde\zeta^*)$ or $(0,0,0,0,0)$. Take the one giving the largest value for $W$ (the explicit value for $W(0,0,0,0,0)$ is $-\frac1{2r}(\sigma_{i,j}^n)^2 + (\Lambda_h \phi)_{i,j}^n \cdot \sigma_{i,j}^n$ by definition of $W$).
\end{enumerate}

Finally we can deduce $q=(a,b,c, \tilde b, \tilde c)$ using \eqref{eq:qFctSigmaPrime}.

\paragraph{Step 3 : update of $\sigma = (m,y,z, \tilde y, \tilde z)$ :}
The last step is simply~:
\begin{equation*}
	\sigma^{k+1} = \sigma^k - r \Big(\Lambda_h(\phi^{k+1}) - q^{k+1} \Big).
\end{equation*}
It consists of a loop on all the nodes of the time-space grid.

\subsection{Convergence criteria}\label{subsec:convergenceCrit}
At convergence of the ADMM, the  grid functions $(u^{n})_{i,j}$ and $(m^{n})_{i,j}$ satisfy some discrete versions of (\ref{eq:HJB-mftc-cong})-(\ref{eq:FP-mftc-cong}) that will be written below (the discrete Bellman equation holds at the nodes where  $m$ is positive). The latter discrete equations are obtained by writing the optimality conditions for (\ref{eq:pbA2}) and (\ref{eq:pbAdual1}).
They are reminiscent of the finite difference schemes used in~\cite{YAJML} for  MFG problems with congestion.
  Recall that since we study a mean field type control problem, \eqref{eq:HJB-mftc-cong}  differs from the Bellman  equation of the MFG system: it has an additional term involving the derivative of $H$ w.r.t $m$.
\\
 To study numerically the convergence of the ADMM, we may use the $\ell^2$ norm of the residuals of the above mentionned discrete equations.

\paragraph{Discrete HJB equation.} %
When $\nu = 0$, the discrete version of the HJB equation~\eqref{eq:HJB-mftc-cong} 
is obtained by applying the following semi-implicit Euler scheme: for  $n \in \{0, \dots, N_T-1\}, i,j \in \{0, \dots, N_h-1\}$,
\begin{align}\label{eq:discreteHJB}
\ds \frac{u_{i,j}^{n+1} - u_{i,j}^{n}}{\Delta t}
 + H_h(x_{i,j}, m^{n+1}_{i,j}, [\grad_h u^{n}]_{i,j} )
 + m^{n+1}_{i,j} \frac{\partial H_h}{\partial m}(x_{i,j}, m^{n+1}_{i,j}, [\grad_h u^{n}]_{i,j} )&= 0.
 \end{align}
At convergence of the ADMM, this equation holds at all the grid nodes where $m^{n+1}_{i,j}>0$. 
\paragraph{\bf Discrete transport operator.}
In order to approximate equation \eqref{eq:FP-mftc-cong}, we 
multiply the nonlinear term in  \eqref{eq:FP-mftc-cong} by a test -function $w$
and integrate over $\Omega$, as one would do when writing the weak formulation of  \eqref{eq:FP-mftc-cong}: this yields the integral
 $\int_{\Omega} \diver\left(m \frac {\partial H}{\partial p}(\cdot,m,\nabla u) \right)(x)\; w(x)\, dx$, in which $m$ appears twice.
By integration by parts (using the periodic boundary condition),
\[\ds I=- \int_{\Omega} m(x) \frac {\partial H}{\partial p}(x, m(x),\nabla u(x)) \cdot \nabla w(x),\ \]
which  will be approximated by
\[ - h^2\sum_{i,j}    m_{i,j} \nabla_q H_h (x_{i,j}, m_{i,j},[\grad_h u]_{i,j}) \cdot [\grad_h w]_{i,j}.\]
We define the  transport  operator $\cT$ by
\[h^2 \sum_{i,j} \cT_{i,j}(u,m,\widetilde m) w_{i,j} = - h^2\sum_{i,j}    m_{i,j} \nabla_q H_h (x_{i,j},\widetilde m_{i,j},[\grad_h u]_{i,j}) \cdot [\grad_h w]_{i,j},\]
where we have doubled the $m$ variable, to keep the notations used in \cite{YAJML}.  This identity completely characterizes $ \cT_{i,j}(u,m,\widetilde m)$: for any point $x_{i,j}$,
\begin{equation}
  \begin{split}
  &  \cT_{i,j}(u,m,\widetilde m)= \\   & \frac 1 h\left(
  \begin{array}[c]{l}
\ds
\left(
  \begin{array}[c]{l}
   \ds  +m_{i,j}  \frac {\partial H_h} {\partial q_1} (x_{i,j}, \widetilde m_{i,j}, [\grad_h u]_{i,j})
     - m_{i-1,j}  \frac {\partial H_h} {\partial q_1}(x_{i-1,j}, \widetilde m_{i-1,j}, [\grad_h u]_{i-1,j}) \\ \ds
     +m_{i+1,j}  \frac {\partial H_h} {\partial q_2} (x_{i+1,j}, \widetilde m_{i+1,j}, [\grad_h u]_{i+1,j})
	- \ds  m_{i,j}  \frac {\partial H_h} {\partial q_2} (x_{i,j}, \widetilde m_{i,j}, [\grad_h u]_{i,j})
  \end{array}
\right)
\\ +\\
\ds 
\left(
  \begin{array}[c]{l}
\ds + m_{i,j}  \frac {\partial H_h} {\partial q_3} (x_{i,j},\widetilde m_{i,j}, [\grad_h u]_{i,j})
	- \ds  m_{i,j-1}\frac {\partial H_h} {\partial q_3} (x_{i,j-1},\widetilde m_{i,j-1}, [\grad_h u]_{i,j-1}) \\
	+ \ds  m_{i,j+1}\frac {\partial H_h} {\partial q_4} (x_{i,j+1}, \widetilde m_{i,j+1},[\grad_h u]_{i,j+1})
	- m_{i,j}  \frac {\partial H_h} {\partial q_4} (x_{i,j},\widetilde m_{i,j}, [\grad_h u]_{i,j})
  \end{array}
\right)
  \end{array}
\right).
  \end{split}
\end{equation}

\paragraph{Discrete Kolmogorov equation.}
With the notations introduced above, the following discrete Kolmogorov equation corresponding to~\eqref{eq:FP-mftc-cong}:
\begin{equation}
\label{eq:discreteFP}
\ds   \frac {m^{n+1}_{i,j}- m^{n}_{i,j}} {\dt}
+ \cT_{i,j}(u^{n},m^{n+1},m^{n+1} )= 0,
\end{equation}
arises in the optimality condition for (\ref{eq:pbA2}) and (\ref{eq:pbAdual1}).

\paragraph{Criteria of convergence.} 
To study numerically the convergence of the ADMM, we take the approximate solution $(\phi^k, q^k, \sigma^k)$ obtained at the $k$-th iteration
and compute the $\ell^2$ norm (and the $m$-weighted  $\ell^2$ norm) of the residual $w$ defined as:
\begin{equation*}
w_{i,j}^n = \ds (D_t (\phi^k)_{i,j})^{n+1} 
 + H_h(x_{i,j}, (m^k)^{n+1}_{i,j}, [\grad_h (\phi^k)^{n}]_{i,j} )
 + (m^k)^{n+1}_{i,j} \frac{\partial H_h}{\partial m}(x_{i,j}, (m^k)^{n+1}_{i,j}, [\grad_h (\phi^k)^{n}]_{i,j} ),
\end{equation*}
if  $(m^k)^{n+1}_{i,j}\not =0$, $w_{i,j}^n= 0$ otherwise.
In our implementation, we do not compute  the residuals of the discrete Kolmogorov equation, although this would also give a good criteria of convergence.\\
Another criteria that we use is the error between $\Lambda(\phi^k)$ and $q^k$:
\begin{equation*}
	||\Lambda(\phi^k) - q^k||_{\ell^2(\R^{5N'})}.
\end{equation*}
Since $\left(\Lambda(\phi^k)\right)_k$ and $(q^k)_k$ are converging to the same limit, this term tends to $0$.
Finally, we  can also check that $||\phi^{k+1} - \phi^k||_{\ell^2(\R^n)}$ and $||m^{k+1} - m^k||_{\ell^2(\R^n)}$ tend to $0$.

\section{State constraints}\label{sec:state-constraints}

The goal of this section is to extend the previous algorithm to mean field type control with state constraints, hence with different 
boundary conditions which are relevant  to model  walls  or obstacles  and   more realistic from the point of view of applications. In general the  problems will not be periodic any longer. 

\begin{remark}
For brevity, we restrict ourselves to a one dimensional interval, but the same approach can be applied 
to two-dimensional domains. In section~\ref{sec:numerics}, we will show bidimensional numerical simulations.
\end{remark}

\paragraph{Model and assumptions.}
We consider the interval $\Omega = (0,1) \subset \R$. At a point $x\in \partial \Omega$,  we note by $n(x)$ the outward normal vector (actually a scalar in dimension one).
 Let us take $0\leq\alpha<1$ and $1<\beta\leq 2$ with conjugate exponent $\beta^* = \beta/(\beta-1)$. 
As in the previous sections, $\ell$ is a continuous cost function with the same assumptions as in~\S~\ref{subsec:model}. 
We also fix $u_T \in \cC^2(\Omega)$ and $m_0 \in \cC^1(\Omega)$ such that $m_0 \geq 0$ and $\int_\Omega m_0(x) dx=1$.
 We also assume that $\nu = 0$; in other words, we focus on deterministic mean field type control problem. 
\\
We consider on $\overline \Omega$ the same Lagrangian as in~\S~\ref{subsec:model}~:
$
	 L(x,m,\xi) = (\beta - 1)\beta^{-\beta^*} m^{\frac \alpha{\beta -1}}|\xi|^{\beta^*} + \ell(x,m).
$
The Hamiltonian takes a different form on the boundary than inside the domain:
\begin{align*}
	 H(x,m,p) &=
	 \begin{cases}
	 	\inf_{\xi \in \RR} \big\{\xi \cdot p + L(x,m,\xi)\big\} = -m^{-\alpha} |p|^{\beta} + \ell(x,m) &\mbox{ if } x \in \Omega, \\
		\inf_{\xi \in \RR \,:\, \xi \cdot n \leq 0} \big\{\xi \cdot p + L(x,m,\xi)\big\} &\mbox{ if } x \in \partial \Omega = \{0,1\}.
	\end{cases}
\end{align*}
Note that on the boundary, the infimum is taken over dynamics staying in  $\overline \Omega$.

\subsection{Numerical Scheme}
\paragraph{Discretization.}
Let ${\overline\Omega}_h$ be a uniform grid on $\overline \Omega = [0,1]$ with mesh step $h$ such that $1/h$ is an integer $N_h$. Let $\Omega_h = {\overline\Omega}_h \backslash \{0,1\}$. Note by $x_{i}$ the point in ${\overline\Omega}_h$ of coordinate $i h$, $i \in \{0, \dots, N_h\}$. Let $T>0$, $N_T$ be a positive integer, $\Delta t = T/N_T$, and $t_n = n\Delta t$, for $n \in \{0,\dots, N_T\}$.  Moreover, we note $N = (N_T+1) (N_h+1)$ the total number of points in the space-time grid and $M = (N_T+1) (N_h-1)$ the number of points inside the space domain, and $N'=N_T(N_h+1)$. 
As above, the discrete data are noted ${\widetilde m}^0$ and ${\widetilde u}^{N_T}$.
We define the discrete Hamiltonian:
\begin{equation*}
	H_h(x,m,p_1,p_2) = - m^{-\alpha} \left( (p_1^-)^2 + (p_2^+)^2 \right)^{\beta/2} + \ell(x,m)
\end{equation*}
for any $x \in \Omega_h$, $m \geq 0$, $p_1 \in \RR$ and $p_2 \in \RR$.  On the boundary, we define~:
\begin{equation*}
	H_{h,0} (m,p_1) = - m^{-\alpha} |p_1^-|^{\beta} + \ell(0,m), \quad \hbox{ and } \quad H_{h,1} (m, p_2) = - m^{-\alpha} |p_2^+|^{\beta} + \ell(1,m).
\end{equation*}
\begin{remark}\label{sec:discretization}
	In dimension two, if the domain is $(0,1)^2$ for example, the Hamiltonian would take a special form on each segment of the boundary and at each corner.
\end{remark}

\paragraph{Assumption} 
As above, we introduce the discrete first order right sided finite difference operator: for any $\phi \in \RR^{N_h+1}$,
$
	(D^+ \phi)_{i} = \frac{1}{h}(\phi_{i+1} - \phi_{i}),$ for all $i\in\{0,\dots,N_h-1\}.
$
We let $[\grad_h \phi ]_{i}$ be the collection of the two possible one sided finite differences at $x_{i} \in \Omega_h, i \in\{1,\dots,N_h-1\}$:
$
	[\grad_h \phi ]_{i} = \Big( (D_1^+ \phi)_{i}, (D_1^+ \phi)_{i-1} \Big) \in \RR^2,
$
and we let $\Delta_h$ be the discrete Laplacian:
$
	(\Delta_h \phi)_{i} = - \frac{1}{h^2} \big( 2 \phi_{i} - \phi_{i+1} - \phi_{i-1}\big),$ for all $i\in\{1,\dots,N_h-1\}.
$
For any $(\phi^n)_{n \in\{0,\dots, N_T\}} \in \R^{N_T+1}$, we note the discrete first order finite difference operator in time:
$
	(D_t \phi)^n = \frac{1}{\Delta t}(\phi^n - \phi^{n-1}),$ for all $n\in\{1,\dots,N_T\}.
$

\paragraph{Discrete  problem $\mathcal A_h$.}
We consider the discrete optimization problem:
\begin{align}
	\mathcal A_h = 
	\sup_{\phi\in\R^N} \min_{m \in (\R_+)^N} \Big\{ &h \Delta t \sum_{n=1}^{N_T} \sum_{i=0}^{N_h}  m_{i}^n (D_t \phi_{i})^n
	+ h \Delta t \sum_{n=1}^{N_T} \sum_{i=1}^{N_h-1}  m_{i}^nH_h\left(x_{i}, m_{i}^n, (D^+ \phi^{n-1})_{i} , (D^+ \phi^{n-1})_{i-1}  \right) \notag \\
	&+ \Delta t \sum_{n=1}^{N_T} m_{0}^nH_{h,0} \left(m_{0}^n , (D^+ \phi^{n-1})_{0} \right)+ \Delta t \sum_{n=1}^{N_T} m_{N_h}^nH_{h,1} \left(m_{N_h}^n , (D^+ \phi^{n-1})_{N_h-1} \right) \notag \\
	& -\chi_T(\phi) + h \sum_{i=0}^{N_h} {\widetilde m}_{i}^0 \phi_{i}^0 \Big\} \label{eq:pbA1-ST}
\end{align}
where, as before, $\chi_T(\phi) = 0$ if $\phi_{i}^{N_T}  = {\widetilde u}_{i}^{N_T}$ for all $i \in \{0,\dots,N_h\}$, and $\chi_T(\phi) = +\infty$ otherwise.

We can formulate $\mathcal A_h$ as a convex minimization problem:
\begin{align}
	\mathcal A_h 
	&= - \inf_{\phi \in \R^N} \big\{ \mathcal F_h(\phi) +  \mathcal G_h (\Lambda_h (\phi)) \big\},  \label{eq:pbA2-ST}
\end{align}
where $\Lambda_h : \R^N \to \R^{3N'}$ is defined by: $\forall \, n\in\{1,\dots,N_T\}$,
\begin{eqnarray*}
  (\Lambda_h(\phi))_{i}^n &=& 
		\left( (D_t \phi_{i})^n , [\grad_h \phi^{n-1}]_{i}  \right) \qquad \qquad\forall \, i\in\{1,\dots,N_h-1\}, \\
	(\Lambda_h(\phi))_{0}^n &=& 
		\left( (D_t \phi_{0})^n , (D_1^+ \phi^{n-1})_{0}, 0\right), \\
	(\Lambda_h(\phi))_{N_h}^n &=&
		\left( (D_t \phi_{N_h})^n , 0 , (D_1^+ \phi^{n-1})_{N_h-1} \right),
\end{eqnarray*}
(for more homogeneity in the notations, we have added a dummy $0$ in $(\Lambda_h(\phi))_{0}^n$ and $(\Lambda_h(\phi))_{N_h}^n$),
and $\mathcal F_h : \R^N \to \R \cup \{+\infty\}$ and $\mathcal G_h : \R^{3N} \to \R \cup \{+\infty\}$ are the two  proper functions defined by
\begin{align*}
 \mathcal F_h(\phi) = \, &\chi_T(\phi) - h \sum_{i=0}^{N_h} {\widetilde m}_{i}^0 \phi_{i}^0, \\
	\hbox{and } \quad 	\mathcal G_h(a,b,c) = &- \min_{m \in (\R_+)^N} \Big\{ h \Delta t \sum_{n=1}^{N_T} \sum_{i=0}^{N_h} m_{i}^n a_{i}^{n}
	+  h \Delta t \sum_{n=1}^{N_T} \sum_{i=1}^{N_h-1} m_{i}^n H_h(x_{i}, m_{i}^n, b_{i}^{n}, c_{i}^{n}) \\
	&\qquad\qquad +  \Delta t \sum_{n=1}^{N_T} m_{0}^n H_{h,0}(m_{0}^n, b_{0}^{n}) + \Delta t \sum_{n=1}^{N_T} m_{N_h}^n H_{h,1}(m_{N_h}^n, c_{N_h}^{n})
	 \Big\}\\
	= &-  h \Delta t \sum_{n=1}^{N_T} \sum_{i=1}^{N_h-1} K_h(x_{i}, a_{i}^n, b_{i}^{n}, c_{i}^{n}) -  \Delta t \sum_{n=1}^{N_T} K_{h,0}(a_{0}^n, b_{0}^{n}) - \Delta t \sum_{n=1}^{N_T} K_{h,1}(a_{N_h}^{n}, c_{N_h}^{n}),
      \end{align*}
with   $K_h(x, a, p_1, p_2) = \min_{m \in \R_+} \big\{ m(a + H_h(x, m, p_1, p_2) )\big\}$, $K_{h,0}(a, p_1) = \min_{m \in \R_+} \big\{ m(a + H_{h,0}(m, p_1) )\big\}$,
				and $ K_{h,1}(a, p_2) = \min_{m \in \R_+} \big\{ m(a + H_{h,1}(m, p_2) )\big\}$.
Note that $K_h,K_{h,0}$ and $K_{h,1}$ are nonpositive and concave in $(a, p_1, p_2), (a, p_1)$  and $(a, p_2)$. Hence $\mathcal G_h$ is indeed convex.

\paragraph{Dual version of problem $\mathcal A_h$.}\label{sec:discreteDuality-ST}
From \eqref{eq:pbA2-ST}, by Fenchel-Rockafellar theorem (see e.g. \cite{MR1451876}, Corollary 31.2.1), we deduce that the dual problem of $\mathcal A_h$ is:
\begin{equation}\label{eq:pbAdual1-ST}
	\min_{\sigma \in \R^{3N'}} \big\{\mathcal F_h^*(\Lambda_h^*(\sigma)) + \mathcal G_h^*(-\sigma)\big\},
\end{equation}
where $\mathcal G_h^*$ and $\mathcal F_h^*$ are the Legendre-Fenchel conjugates of $\mathcal G_h$ and $\mathcal F_h$ respectively, defined by
$
\mathcal F_h^*(\mu) = \sup_{\phi \in \R^N} \Big\{ \langle \mu , \phi \rangle_{\ell^2(\R^{N})}  - \mathcal F_h(\phi) \Big\},  
$
and 
\begin{align}
	\mathcal G_h^*(-\sigma)
	= &\max_{q \in \R^{3N'}} \Big\{ - \langle \sigma , q \rangle_{\ell^2(\R^{3N'})} - \mathcal G_h(q)  \Big\} \notag \\
	= &\max_{q} \Big\{ h \Delta t  \sum_{n=1}^{N_T} \sum_{i=1}^{N_h-1}\Big[ - \sigma_{i}^n \cdot q_{i}^n + K_h(x_{i},  q_{i}^{n})\Big]  + \Delta t  \sum_{n=1}^{N_T} \Big[ - m_{0}^n a_{0}^n  - y_{0}^n c_{0}^n  + K_{h,0}( a_{0}^{n},  b_{0}^{n})\Big] \notag \\
	 & \qquad + \Delta t  \sum_{n=1}^{N_T} \Big[ - m_{N_h}^n a_{N_h}^n  - z_{N_h}^n c_{N_h}^n  + K_{h,1}( a_{N_h}^{n},  c_{N_h}^{n})\Big] \Big\}\notag  \\
	= & h \Delta t  \sum_{n=1}^{N_T} \sum_{i=1}^{N_h-1} \tilde L_h(x_{i}, \sigma_{i}^n)
	 	 + \Delta t  \sum_{n=1}^{N_T} \tilde L_{h,0}(m_{0}^n, y_{0}^n)
		 + \Delta t  \sum_{n=1}^{N_T} \tilde L_{h,1}(m_{N_h}^n, z_{N_h}^n), \label{eq:GstarLbar-ST} 
               \end{align}
with $\tilde L_h(x,\sigma_0) = \max_{q_0 \in \R^{3}} \big\{ -\sigma_0 \cdot q_0 + K_h(x,q_0) \big\}$,  $\forall \sigma_0 \in \R^3$,
$\tilde L_{h,0}(m, y) = \max_{(a,b) \in \R^2} \big\{ -(a,b) \cdot (m, y) + K_{h,0}(a,b) \big\}$, $\forall y \in \R$, and
$\tilde L_{h,1}(m, z) = \max_{(a,c) \in \R^2} \big\{ -(a,c) \cdot (m, z) + K_{h,1}(a,c) \big\}$, $\forall z \in \R$.
\\
Finally, $\Lambda_h^*: \R^{3N'} \to \R^N$ is the adjoint of $\Lambda_h$, defined by
\begin{align*}
	&\left\langle \Lambda_h^*(m,y,z), \phi \right\rangle_{\ell^2(\R^{N})}
	= \left\langle (m,y,z), \Lambda_h(\phi) \right\rangle_{\ell^2(\R^{3N'})}
	\\
	&= h \Delta t  \sum_{n=1}^{N_T} \left[  \sum_{i=0}^{N_h} m_{i}^n \frac{\phi_{i}^n - \phi_{i}^{n-1}}{\Delta t} 
	+  \sum_{i=0}^{N_h-1} (y_{i}^n + z_{i+1}^n)\frac{\phi_{i+1}^{n-1}- \phi_{i}^{n-1}}{h} \right]
		\\
	&= h \Delta t  \sum_{n=0}^{N_T-1}  \left[ - \sum_{i=0}^{N_h} \frac{m_{i}^{n+1} - m_{i}^{n}}{\Delta t} \phi_{i}^n  \right]  + h \sum_{i=0}^{N_h}  m_{i}^{N_T} \phi_{i}^{N_T} - h \sum_{i=0}^{N_h}  m_{i}^{0} \phi_{i}^{0}
	\\
	&\qquad + h \Delta t  \sum_{n=0}^{N_T-1}  \left[ \sum_{i=1}^{N_h-1} \left(-\frac{y_{i}^{n+1}- y_{i-1}^{n+1}}{h} - \frac{z_{i+1}^{n+1}- z_{i}^{n+1}}{h} \right)\phi_{i}^{n}  + (y_{N_h-1}^{n+1} + z_{N_h}^{n+1}) \frac {\phi_{N_h}^{n-1}}{h} - (y_{0}^{n+1}+z_{1}^{n+1}) \frac {\phi_{0}^{n}}{h} \right].
\end{align*}
Hence
\begin{align*}
	\mathcal F_h^*(\Lambda_h^*(m,y,z)) = 
	\begin{cases}
		h \sum_{i=0}^{N_h}  m_{i}^{N_T} {\widetilde u}_{i}^{N_T} &\hbox{ if $(m,y,z)$ satisfies \eqref{eq:FPdiscrete-ST} (see below)} \\
		+\infty &\hbox{otherwise}
	\end{cases}
\end{align*}
with : $\forall \, i\in\{0,\dots,N_h\}$, $m_{i}^0  = {\widetilde m}^0_{i}$, and :
\begin{equation}\label{eq:FPdiscrete-ST}
	\begin{dcases}
		\frac{m_{i}^{n+1} - m_{i}^{n}}{\Delta t} + \frac{y_{i}^{n+1} - y_{i-1}^{n+1}}{h} + \frac{z_{i+1}^{n+1} - z_{i}^{n+1}}{h} = 0, \quad \forall \, n\in\{0,\dots,N_T-1\}, \forall \, i\in\{1,\dots,N_h-1\} \\
		\frac{m_{0}^{n+1} - m_{0}^{n}}{\Delta t} + \frac{y_{0}^{n+1}}{h} + \frac{z_{1}^{n+1}}{h} = 0, 
\quad
		\frac{m_{N_h}^{n+1} - m_{N_h}^{n}}{\Delta t} - \frac{y_{N_h-1}^{n+1}}{h} - \frac{z_{N_h}^{n+1}}{h} = 0, \quad \forall \, n\in\{0,\dots,N_T-1\}.
	\end{dcases}
\end{equation}

So the dual of problem $\mathcal A_h$ rewrites~:
\begin{align}
	\min_{(m,y,z) \in \R^{3N'}} &\left\{ h \Delta t  \sum_{n=1}^{N_T} \sum_{i=1}^{N_h-1} \tilde L_h(x_{i}, m_{i}^n, y_{i}^n, z_{i}^n) 
	+ \Delta t  \sum_{n=1}^{N_T} \tilde L_{h,0}(m_{0}^n, y_{0}^n)
		 + \Delta t  \sum_{n=1}^{N_T} \tilde L_{h,1}(m_{N_h}^n, z_{N_h}^n) + h \sum_{i=0}^{N_h}  m_{i}^{N_T} {\widetilde u}_{i}^{N_T} \right\} \label{eq:pbAdual2-ST} \\
	&\hbox{ subject to \eqref{eq:FPdiscrete-ST}.}\notag
\end{align}
Let us now compute an equivalent expression for $\tilde L_h$. We note $\mathcal D_{\tilde L_h}$ the domain where $\tilde L_h$ is finite, that is:
\begin{equation*}
	\mathcal D_{\tilde L_h} = \{(0,0,0)\} \cup \{(m,y,z) \in\R^3 ~:~ m>0, y \geq 0, z \leq 0\}.
\end{equation*}
and similarly for $\tilde L_{h,0}$ and $\tilde L_{h,1}$:
\begin{align*}
	&\mathcal D_{\tilde L_{h,0}} = \{(0,0)\} \cup \{(m,y) \in\R^2 ~:~ m>0, y \geq 0\}, 
	\qquad \mathcal D_{\tilde L_{h,1}} = \{(0,0)\} \cup \{(m,z) \in\R^2 ~:~ m>0, z \leq 0\}.
\end{align*}

As in the periodic case (see~\eqref {eq:lastMaxLtilde})), Fenchel-Moreau Theorem yields
\begin{align*}
	\forall (m,y,z) \in \mathcal D_{\tilde L_h}, \qquad
	\tilde L_h(x,m,y,z)
	= & \max_{(b,c) \in \R^2} \big\{ -(y, z)  \cdot (b,c)  + mH_h(x,m,b,c) \big\}, \\
	\forall (m,y) \in \mathcal D_{\tilde L_{h,0}}, \qquad 
	\tilde L_{h,0}(m,y)
	= & \max_{b \in \R} \big\{ -yb + mH_{h,0}(m,b) \big\}, \\
	\forall (m,z) \in \mathcal D_{\tilde L_{h,1}}, \qquad 
	\tilde L_{h,1}(m,z)
	= & \max_{c \in \R} \big\{ -zc + mH_{h,1}(m,c) \big\}.
\end{align*}
So we can also express $\tilde L_h$ as follows: $\forall x \in \Omega_h$,
\begin{equation*}
	\tilde L_h(x,m,y,z)
	=
	\begin{cases}
		(\beta-1)\beta^{-\beta^*} \frac{\left( y^2 + z^2 \right)^{\beta^*/2}}{m^{(\beta^*-1)(1-\alpha)}} + m \ell(x,m) &\hbox{ if $m>0, y \geq 0, z \leq 0$} \\
		0 &\hbox{ if $m=y=z=0$} \\
		+\infty &\hbox{ otherwise.}
	\end{cases}
\end{equation*}
At the boundaries:
\begin{align*}
	&\tilde L_{h,0}(m,y)
	=
	\begin{cases}
		(\beta-1)\beta^{-\beta^*} \frac{|y|^{\beta^*}}{m^{(\beta^*-1)(1-\alpha)}} + m \ell(0,m) &\hbox{ if $m>0, y \geq 0$} \\
		0 &\hbox{ if $m=y=0$} \\
		+\infty &\hbox{ otherwise,}
	\end{cases}
	\\
	&\tilde L_{h,1}(m,z)
	=
	\begin{cases}
		(\beta-1)\beta^{-\beta^*} \frac{|z|^{\beta^*}}{m^{(\beta^*-1)(1-\alpha)}} + m \ell(1,m) &\hbox{ if $m>0, z \leq 0$} \\
		0 &\hbox{ if $m=z=0$} \\
		+\infty &\hbox{ otherwise.}
	\end{cases}
\end{align*}

\subsection{Augmented Lagrangian}
Let us go back to the primal formulation of $\mathcal A_h$. We decouple  $\mathcal F_h$ and $\mathcal G_h$ by introducing a different set of arguments for $\mathcal G_h$
 then  add the constraint that the latter arguments coincide with $\Lambda_h(\phi)$. The problem becomes:
\begin{align*}
	\mathcal A_h 
	= - &\inf_{\phi \in \R^N} \inf_{q \in \R^{3N'}}\big\{ \mathcal F_h(\phi) + \mathcal G_h (q) \big\} 
	\hbox{ subject to $q = \Lambda_h(\phi)$.}
\end{align*}
The Lagrangian corresponding to this constrained optimization problem is
\begin{equation}\label{eq:discreteL-ST}
	\mathcal L_h (\phi, q, \sigma) = \mathcal F_h(\phi) + \mathcal G_h(q) - \langle \sigma , (\Lambda_h(\phi) - q) \rangle_{\ell^2(\R^{3N'})}
\end{equation}
and the augmented Lagrangian is defined for $r>0$ as
\begin{align*}
	\mathcal L_h^r (\phi, q, \sigma) = \mathcal L_h (\phi, q, \sigma) + \frac r 2 \left|\left| \Lambda_h(\phi) - q \right|\right|_{\ell^2(\R^{3N'})}^2.
\end{align*}

\subsection{Alternating Direction Method of Multipliers for $\mathcal L_h^r$}
Since the ADMM is similar to the periodic setting,  we only put the stress on  the main differences.
We note $q = (a,b,c)\in\RR^{3N'}$ and $\sigma = (m,y,z)\in\RR^{3N'}$. Starting from an initial candidate solution $(\phi^0, q^0, \sigma^0)$, we find for $k\geq 0$
\begin{align}
	&\phi^{k+1} \in \argmin_{\phi \in \R^N} \Big\{
		\mathcal F_h(\phi) - \langle \sigma^k , \Lambda_h(\phi) \rangle_{\ell^2(\R^{3N'})}  + \frac r 2 \left|\left| \Lambda_h(\phi) - q^k \right|\right|_{\ell^2(\R^{3N'})}^2
	 \Big\}, \label{eq:update-phi-ST}\\
	&q^{k+1} \in \argmin_{q \in \R^{3N'}} \Big\{
		 \mathcal G_h(q) + \langle \sigma^k , q \rangle_{\ell^2(\R^{3N'})} + \frac r 2 \left|\left| \Lambda_h(\phi^{k+1}) - q \right|\right|_{\ell^2(\R^{3N'})}^2
	 \Big\}, \label{eq:update-q-ST}\\
	 &\sigma^{k+1} = \sigma^k - r \Big(\Lambda_h(\phi^{k+1}) - q^{k+1} \Big). \label{eq:update-sigma-ST}
\end{align}
Below, we give  details on the  algorithm. 

\paragraph{Step 1 : update of $\phi$ :}
To shorten the notations, let us drop the superscript $k$ and note in this step $q = (a,b,c) = (a^k,b^k,c^k) = q^k$ and $\sigma = (m,y,z) = (m^k,y^k,z^k) = \sigma^k$.
We are looking for the $\phi \in \R^N$ that satisfies:
\begin{enumerate}
	\item $\forall \, i = 0, \dots, N_h, \quad \phi_{i}^{N_T} = \tilde u_{i}^{T}$ (otherwise $\mathcal F_h(\phi) = +\infty$),
	\item and, for any $\psi \in \R^N$,
	$0 = D \mathcal F_h(\phi) \psi - \sigma \cdot D\Lambda_h(\phi)\psi + r \left(\Lambda_h(\phi) - q\right) \cdot  D\Lambda_h(\phi)\psi.$
\end{enumerate}
 If $\phi$ satisfies the first condition then the second condition can be written as follows:
\begin{align*}
	0
	=
	& -h \sum_{i=0}^{N_h} m_{i}^0 \psi_{i}^0
	+ h \Delta t \sum_{i=0}^{N_h} \left(r \frac{{\widetilde u}_{i}^{T} - \phi_{i}^{N_T-1}}{\Delta t} - ra_{i}^{N_T}  - m_{i}^{N_T}\right) \frac{ - \psi_{i}^{N_T-1}}{\Delta t} \\
	&+ h \Delta t \sum_{n=1}^{N_T-1} \sum_{i=0}^{N_h} \left(r \frac{\phi_{i}^{n} - \phi_{i}^{n-1}}{\Delta t} - ra_{i}^{n}  - m_{i}^{n}\right) \frac{\psi_{i}^n - \psi_{i}^{n-1}}{\Delta t} \\
	& +h \Delta t \sum_{n=0}^{N_T-1} \sum_{i=0}^{N_h-1} \left(2 r \frac{\phi_{i+1}^{n} - \phi_{i}^{n}}{h} - rb_{i}^{n+1}  - y_{i}^{n+1} -r c_{i+1}^{n+1} - z_{i+1}^{n+1} \right) \frac{\psi_{i+1}^{n} - \psi_{i}^{n}}{h}.
\end{align*}
After a discrete integration by parts, we deduce that $\phi$ must satisfy the following set of equations:
for all $n$, inside the domain the equation is the same as the periodic case. 
Moreover, for $n=0$:
for $i = 0$
	\begin{align*}
		& 
		- r \frac{\phi_0^{1} - \phi_0^{0} }{(\Delta t)^2}  + 2 r \frac{\phi_0^{0} - \phi_{1}^{0}}{h^2}
		=  
		- r \left( \frac{a_0^{1}}{\Delta t} + \frac{b_{0}^{1}}{h} + \frac{c_{1}^{1}}{h}\right)
		-\left( \frac{m_0^{1}- {\widetilde m}_0^0}{\Delta t}
		  + \frac{y_{0}^{1}}{h} + \frac{z_{1}^{1}}{h} \right),
	\end{align*}
	and for $i = N_h$
	\begin{align*}
		& - r \frac{\phi_{N_h}^{1} - \phi_{N_h}^{0} }{(\Delta t)^2} + 2 r \frac{\phi_{N_h}^{0} - \phi_{N_h-1}^{0}}{h^2}
		=
		-r\left(\frac{a_{N_h}^{1}}{\Delta t} - \frac{b_{N_h-1}^{1}}{h} - \frac{c_{N_h}^{1}}{h}\right)
		- \left( \frac{ m_{N_h}^{1} - {\widetilde m}_{N_h}^0}{\Delta t} - \frac{y_{N_h-1}^{1}}{h} - \frac{z_{N_h}^{1}}{h} \right).
	\end{align*}
For $n \in \{1,\dots,N_T-2\}$:
for $i = 0$
	\begin{align*}
		&r \frac{2\phi_0^{n} - \phi_0^{n-1} - \phi_0^{n+1} }{(\Delta t)^2} + 2 r \frac{\phi_0^{0} - \phi_{1}^{0}}{h^2} 
		=
		-r \left( \frac{a_0^{n+1} - a_0^{n}}{\Delta t} + \frac{b_{0}^{1}}{h} + \frac{c_{1}^{1}}{h}\right)
		-\left( \frac{m_0^{n+1} - m_0^{n}}{\Delta t} + \frac{y_{0}^{1}}{h} + \frac{z_{1}^{1}}{h} \right),
	\end{align*}
and for $i = N_h$
	\begin{align*}
		&r \frac{2\phi_{N_h}^{n} - \phi_{N_h}^{n-1} - \phi_{N_h}^{n+1} }{(\Delta t)^2}+ 2 r \frac{\phi_{N_h}^{0} - \phi_{N_h-1}^{0}}{h^2}
		=
		-r \left( \frac{a_{N_h}^{n+1} - a_{N_h}^{n}}{\Delta t} - \frac{b_{N_h-1}^{1}}{h} - \frac{c_{N_h}^{1}}{h}\right) 
		- \left( \frac{m_{N_h}^{n+1} - m_{N_h}^{n}}{\Delta t} - \frac{y_{N_h-1}^{1}}{h} - \frac{z_{N_h}^{1}}{h} \right).
	\end{align*}
For $n=N_T-1$:
	for $i = 0$
	\begin{align*}
		&r \frac{2\phi_0^{N_T-1} - {\widetilde u}_0^{N_T} - \phi_0^{N_T-2} }{(\Delta t)^2} + 2 r \frac{\phi_0^{N_T-1} - \phi_{1}^{N_T-1}}{h^2}
		\\
		=\,&
		-r \left(\frac{a_i^{N_T} - a_i^{N_T-1}}{\Delta t} + \frac{b_{0}^{N_T}}{h} + \frac{c_{1}^{N_T}}{h}\right) - \left( \frac{m_i^{N_T} - m_i^{N_T-1}}{\Delta t} + \frac{y_{0}^{N_T}}{h} + \frac{z_{1}^{N_T}}{h} \right),
	\end{align*}
and for $i = N_h$
	\begin{align*}
		&r \frac{2\phi_{N_h}^{N_T-1} - {\widetilde u}_{N_h}^{N_T} - \phi_{N_h}^{N_T-2} }{(\Delta t)^2} + 2 r \frac{\phi_{N_h}^{N_T-1} - \phi_{N_h-1}^{N_T-1}}{h^2}
		\\
		=\,&
		-r \left( \frac{a_i^{N_T} - a_i^{N_T-1}}{\Delta t} - \frac{b_{N_h-1}^{N_T}}{h} - \frac{c_{N_h}^{N_T}}{h} \right) - \left( \frac{m_i^{N_T} - m_i^{N_T-1}}{\Delta t} - \frac{y_{N_h-1}^{N_T}}{h} - \frac{z_{N_h}^{N_T}}{h} \right).
	\end{align*}

\paragraph{Step 2 : update of $q = (a,b,c)$ :}
This step is similar to the periodic case (see~\S~\ref{subsec:algoADMM}): 
we obtain one optimization problem at each point of the domain (including he boundaries). The optimization problems on the boundaries differ slightly from the optimization problems inside the domain, but they are dealt with using the same techniques.

\paragraph{Step 3 : update of $\sigma = (m,y,z)$ :}
The last step is similar to the periodic case: $	\sigma^{k+1} = \sigma^k - r \Big(\Lambda_h(\phi^{k+1}) - q^{k+1} \Big)$.

\section{Numerical results}\label{sec:numerics}
The methods discussed above have been implemented for both periodic and state constraint boundary conditions, 
and tested on several examples that will be reported below. In particular, we will discuss the  convergence  of the iterative method in \S~\ref{subsec:numCV} and compare the  results obtained for different sets of parameters in \S~\ref{subsec:influenceParams}.

\subsection{Description of the test cases}
In what follows,  $\Omega=(0,1)^2$ and $T=1$ except when explicitly mentioned. The Lagrangian will always be of the form ~\eqref{eq:6}.

\paragraph{Test case 1: evacuation of a square subdomain. }
The first test case is similar to the one  discussed in~\cite{BenamouCarlier2015ALG2}, 
except that we deal with a mean field type control problem instead of a mean field game and 
that the model includes congestion. \\
We take  $\alpha = 0.5, \beta = 2, \ell(x,m) = m$, and we impose periodic boundary conditions.\\
The  agents are  uniformly distributed in a square subdomain of side $1/2$ at the center of the domain and the terminal cost
 is an incitation for  the agents to leave the central subdomain.
 More precisely, the initial density and the terminal cost are given by $m_0 = u_T = \indic_{[1/4,3/4] \times [1/4,3/4]}$.
     These data are displayed in Figure~\ref{fig:squareevacuation-Data}. 
\begin{figure}
\centering
\begin{subfigure}{.4\textwidth}
  \centering
  \includegraphics[width=\linewidth]{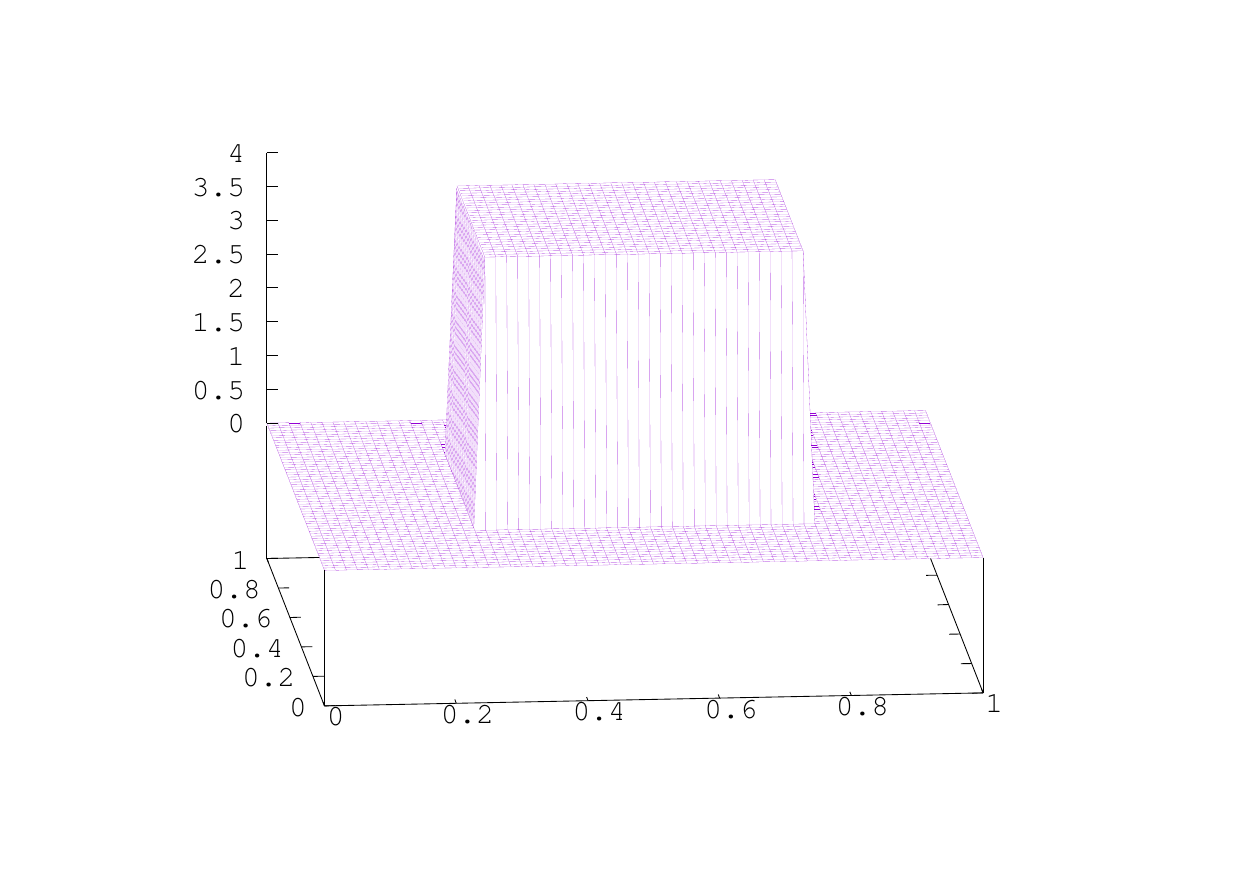}
  \caption{Initial distribution $m = m_0$.}
\end{subfigure}%
\begin{subfigure}{.4\textwidth}
  \centering
  \includegraphics[width=\linewidth]{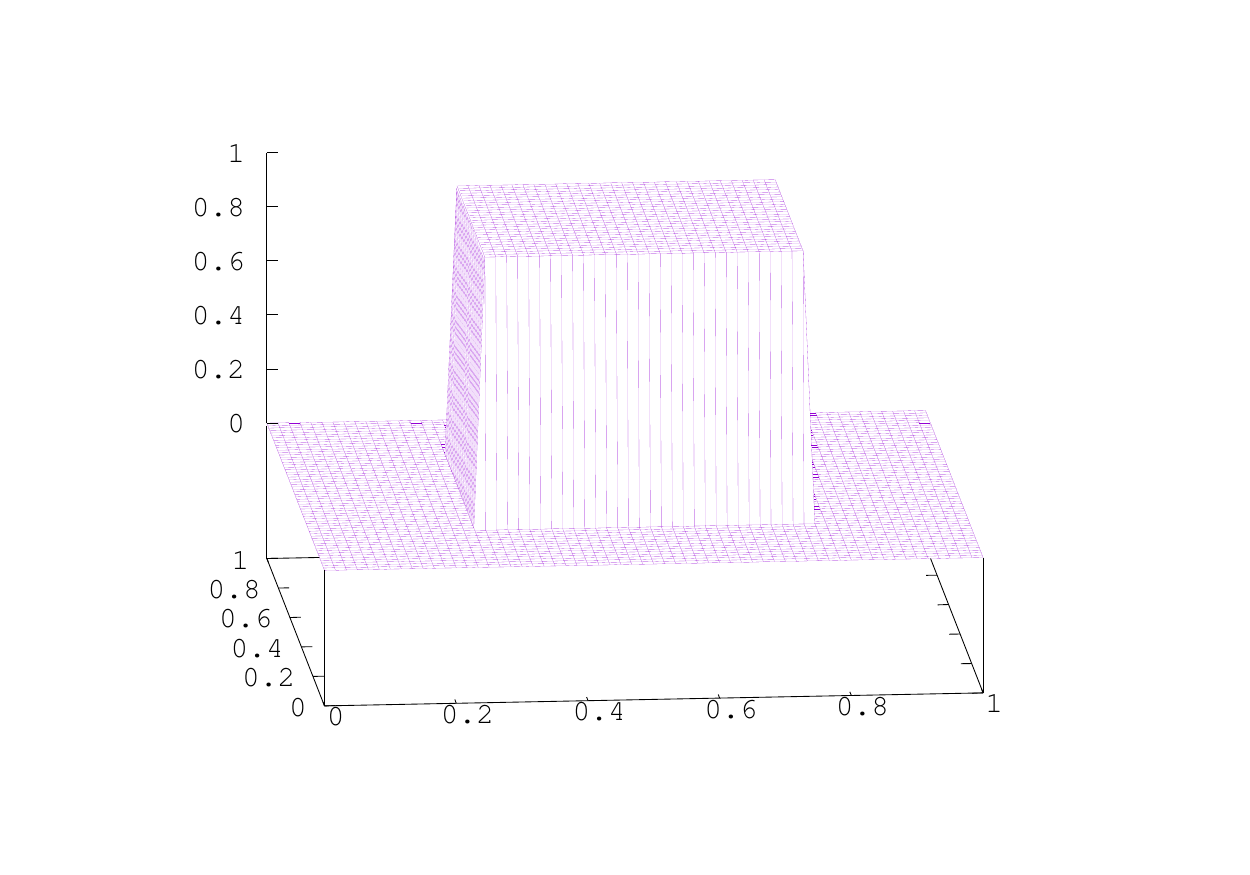}
  \caption{Terminal cost $\phi = u_T$.}
\end{subfigure}
\caption{The data for Test case 1.}
\label{fig:squareevacuation-Data}
\end{figure}

\paragraph{Test case 2: from one corner to the opposite one. }
 Here, the initial density and the terminal cost are given by
$
	m_0 = \indic_{[0,0.2] \times [0,0.2]},$ and  $u_T = 1 - \indic_{[0.8, 1] \times [0.8, 1]}.
$
The data are displayed in Figure~\ref{fig:c2c-Data}. The agents are initially uniformly distributed in a square subdomain located at the bottom-left corner of $\Omega$.
The terminal cost makes the agents move to the top-right corner.  
In this test case, we will compare the effects of  the two  boundary conditions discussed above, see Figure~\ref{fig:c2c-CompareSC}.\\
 We will also present a case in which there is a square obstacle near the center of the domain, see Figure~\ref{fig:c2c-SC-obstacle}.
\begin{figure}
\centering
\begin{subfigure}{.4\textwidth}
  \centering
  \includegraphics[width=\linewidth]{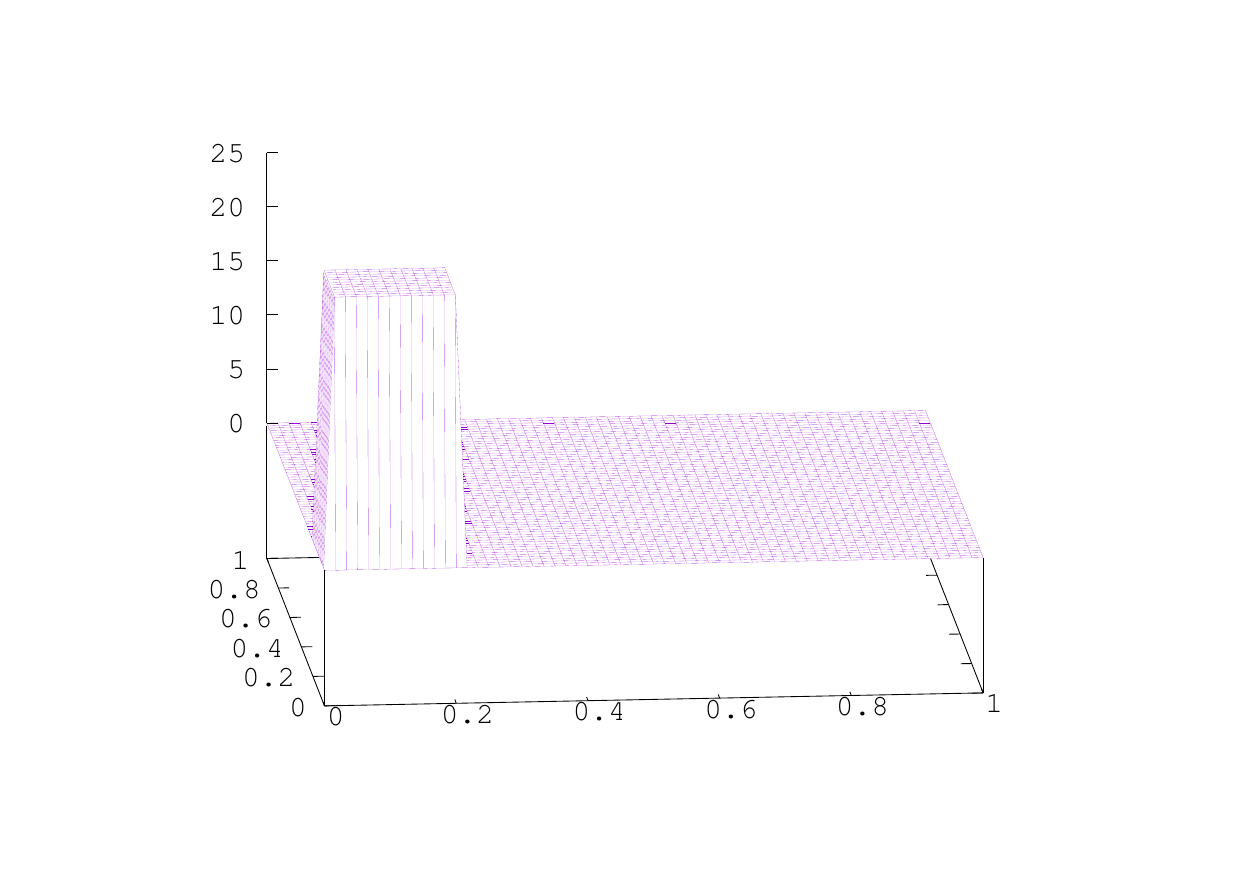}
  \caption{Initial distribution $m = m_0$.}
\end{subfigure}
\begin{subfigure}{.4\textwidth}
  \centering
  \includegraphics[width=\linewidth]{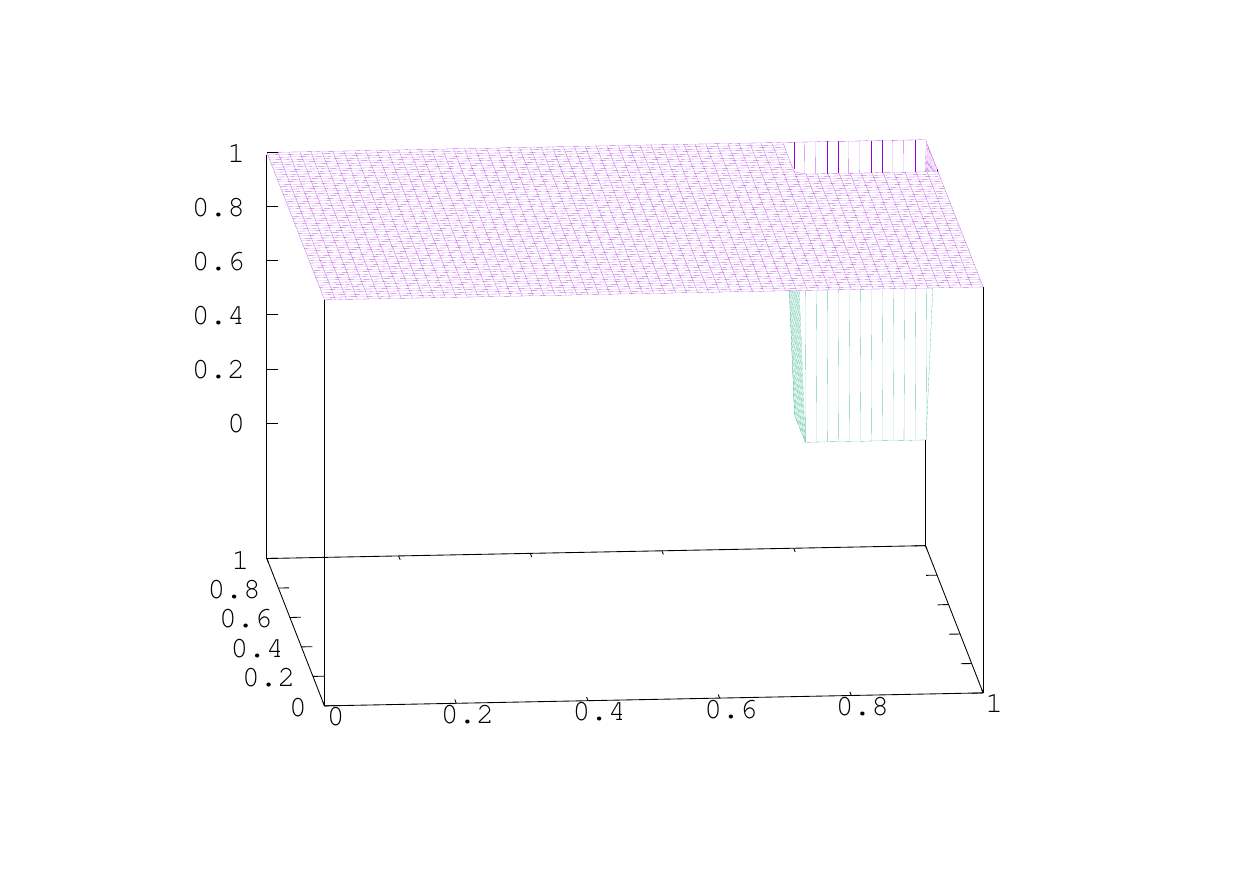}
  \caption{Terminal cost $\phi = u_T$.}
\end{subfigure}
\caption{The data for Test case 2.}
\label{fig:c2c-Data}
\end{figure}

\paragraph{Test case 3: small hump vs peaky hump. } In this test case, the mass is initially distributed in two disconnected regions. 
The initial density of agents is the sum of two nonnegative functions $\psi_1$ and $ \psi_2$ with disjoint supports, both of total mass $1/2$.
The graph of $\psi_1$ is a sinusoidal hump (small amplitude and large support). The graph of $\psi_2$ is a very peaky 
exponential hump (alternatively, we could have taken an approximation of a Dirac mass, but the graphical representation would have been
 more difficult). The terminal cost is an incitation for the agents to move towards the center of the domain. More precisely,
take $(x_0,y_0) = (1/4,3/4)$. The initial distribution is given by $m_0(x) = \psi_1(x)+\psi_2(x)$ with  
\begin{align*}
        \psi_1(x) &= \indic_{[1/2,1] \times [0,1/2]}(x) \frac{\psi^1(x)}{2 \int_{[1/2,1] \times [0,1/2]} \psi^1(y) d y}, \quad  \psi_2(x)=\indic_{[0,1/2] \times [1/2,1]}(x) \frac{\psi^2(x)}{2 \int_{[0,1/2] \times [1/2,1]} \psi^2(y) d y},\\
	\psi^1(x) &= \left( -\sin(2 x_1 \pi) \sin(2 x_2 \pi) - 0.5\right)^+, \quad 	\psi^2(x) = \exp\left(-400\left[(x_1-x_0)^2+(y_1-y_0)^2\right]\right),
\end{align*}
and the terminal cost is given by
$
	u_T(x) = -\exp\left(-20\left[x_1^2 + x_2^2\right]\right).
$
The data are displayed  in Figure~\ref{fig:diracbosse-Data}. We expect that if $\alpha$ is not too small, then
the whole population moves toward the center of the domain, but that the part of the population
 which is initially very dense takes more time to reach the center.
Moreover, the shape of the peaky hump should be modified during its migration to the center.
\\
We will use this example in order to illustrate the impact of the parameter $\alpha$ and of the function $\ell$ on the dynamics of the population, see Figures~\ref{fig:diracbosse-CompareAlpha-ell001} and~\ref{fig:diracbosse-ell-sc}. 

\begin{figure}
\centering
\begin{subfigure}{.4\textwidth}
  \centering
  \includegraphics[width=\linewidth]{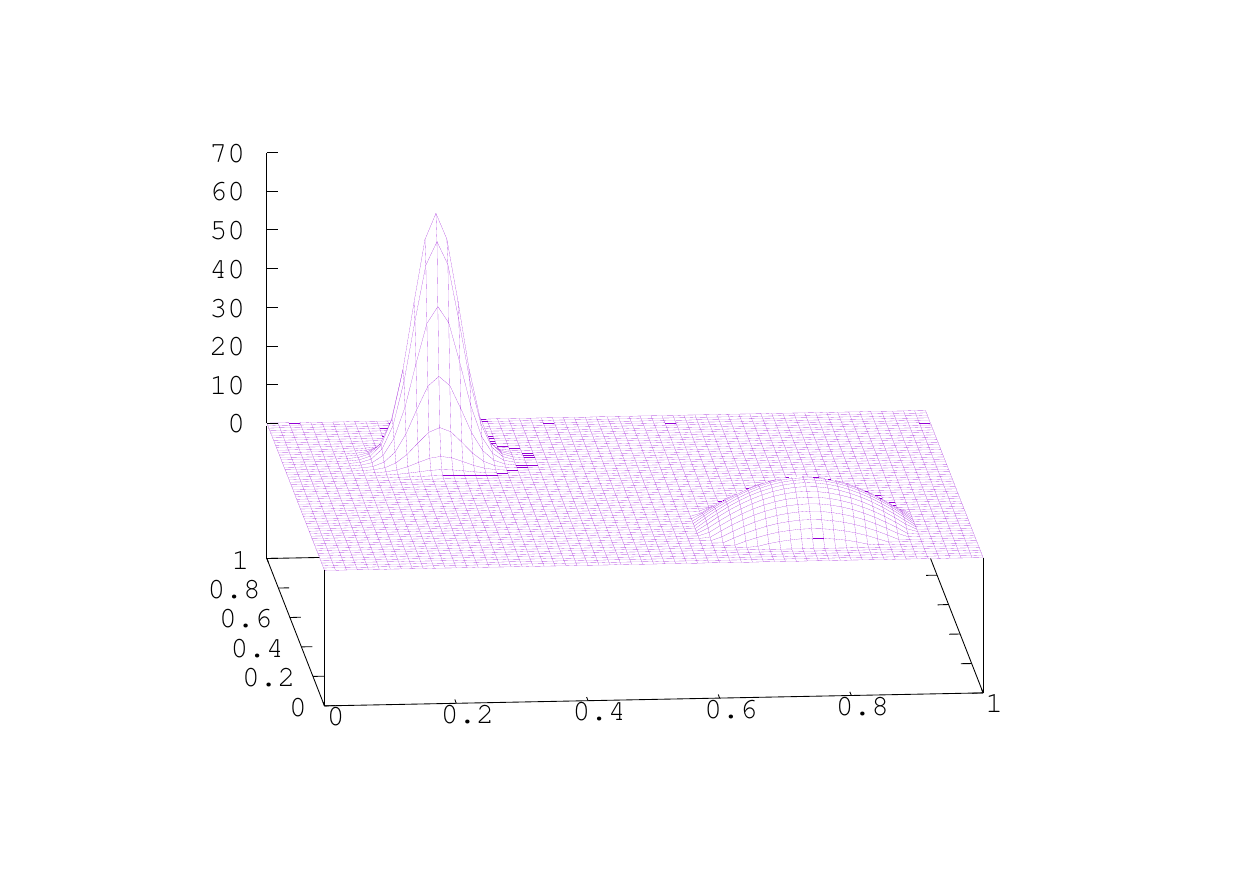}
  \caption{Initial distribution $m = m_0$.}
\end{subfigure}
\begin{subfigure}{.4\textwidth}
  \centering
  \includegraphics[width=\linewidth]{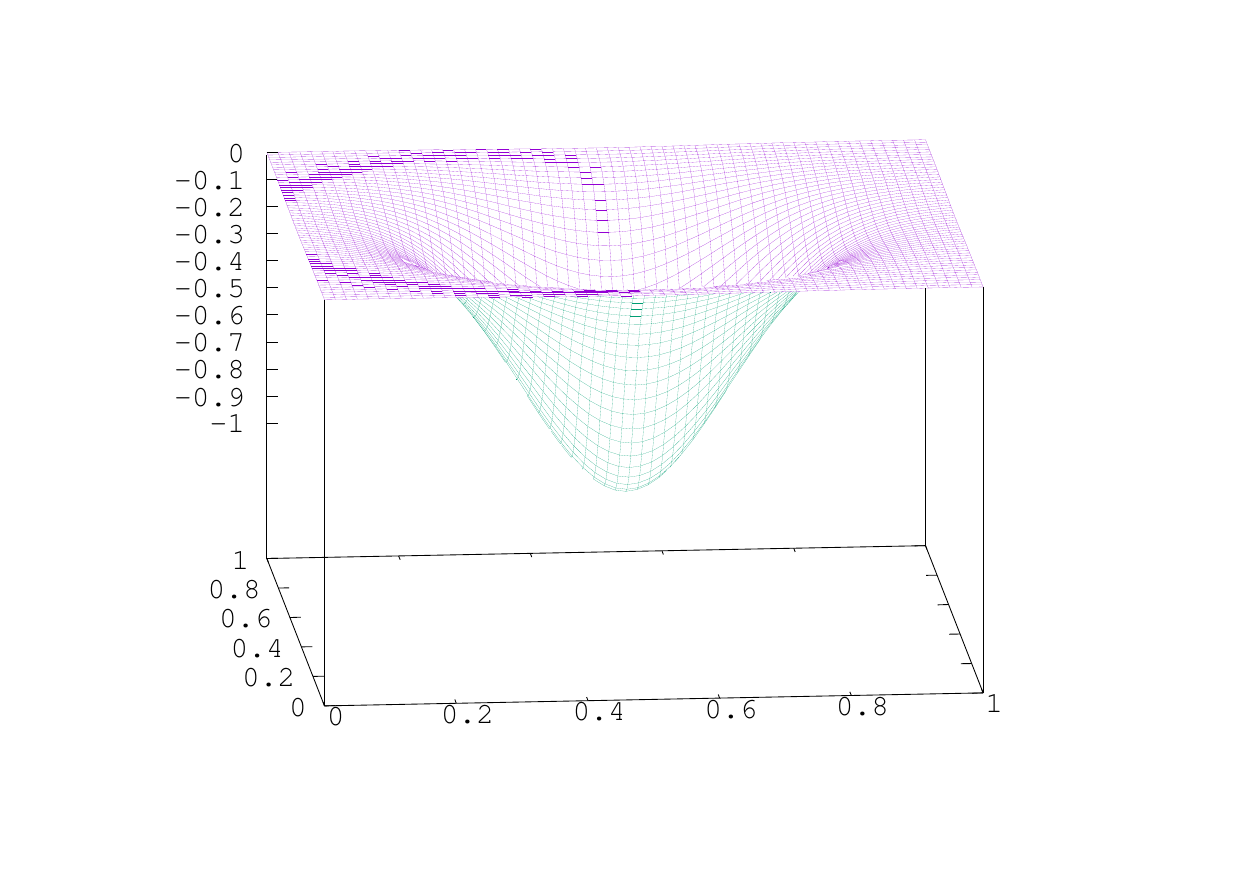}
  \caption{Terminal cost $\phi = u_T$.}
\end{subfigure}
\caption{Data for Test case 3.}
\label{fig:diracbosse-Data}
\end{figure}

\subsection{Convergence of the ADMM}\label{subsec:numCV}
\begin{remark}
	As pointed out in~\cite{FortinGlowinski2000augmented} (\S~5.3 of Chapter~3), the convergence of ALG2 is sensitive to the augmentation parameter $r$. In Figure~\ref{fig:squareevacuation-Convergence-ell1-r} we have plotted the convergence history of the HJB residual for two values of $r$. One can see that the convergence is faster with $r=0.1$ than with $r=10$ (at least in Test case 1).
	However, the choice $r=1$ generally gives good results. We take the latter value in our numerical experiments.
\end{remark}
For Test case 1 described above, the  convergence histories of ADMM  are displayed 
in Figure~\ref{fig:squareevacuation-Convergence-ell1}. We use the criteria described in~\S~\ref{subsec:convergenceCrit}, plotted in log-log scale.
 The convergence curves   are similar to those  obtained in~\cite{BenamouCarlier2015ALG2}.
In Figures~\ref{fig:cv-16-HJB} and~\ref{fig:cv-30-HJB}, we plot the $\ell^2$ norm (weighted by $m$) 
of  the discrete Bellman equation residuals for two grids, with respectively 
 $16 \times 16\times 16$ (left) and $30 \times 30\times 30$ nodes (right):
In Figures~\ref{fig:cv-16-consensus} and~\ref{fig:cv-30-consensus}, we see  the convergence of $||\Lambda(\phi^k) - q^k||_{\ell^2}$ for the first two coordinates (similar convergence holds for the other coordinates).
Figures~\ref{fig:cv-16-mphi} and~\ref{fig:cv-30-mphi} present the convergence of $||\phi^{k+1} - \phi^k||_{\ell^2}$ and $||m^{k+1} - m^k||_{\ell^2}$. 
\begin{figure}
\centering
\begin{subfigure}{.5\textwidth}
  \centering
  \includegraphics[width=\linewidth]{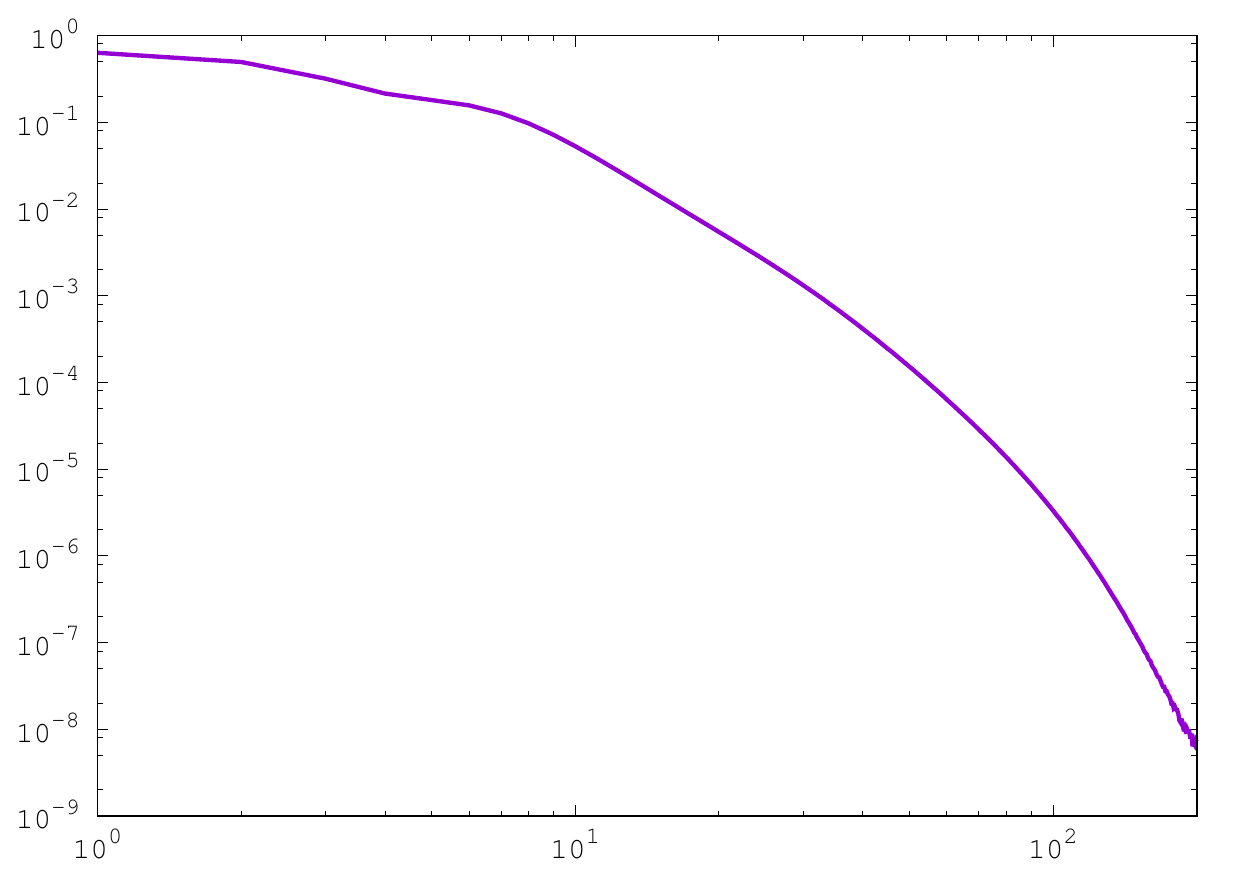}
  \caption{Discrete Bellman equation residual}
  \label{fig:cv-16-HJB}
\end{subfigure}%
\begin{subfigure}{.5\textwidth}
  \centering
  \includegraphics[width=\linewidth]{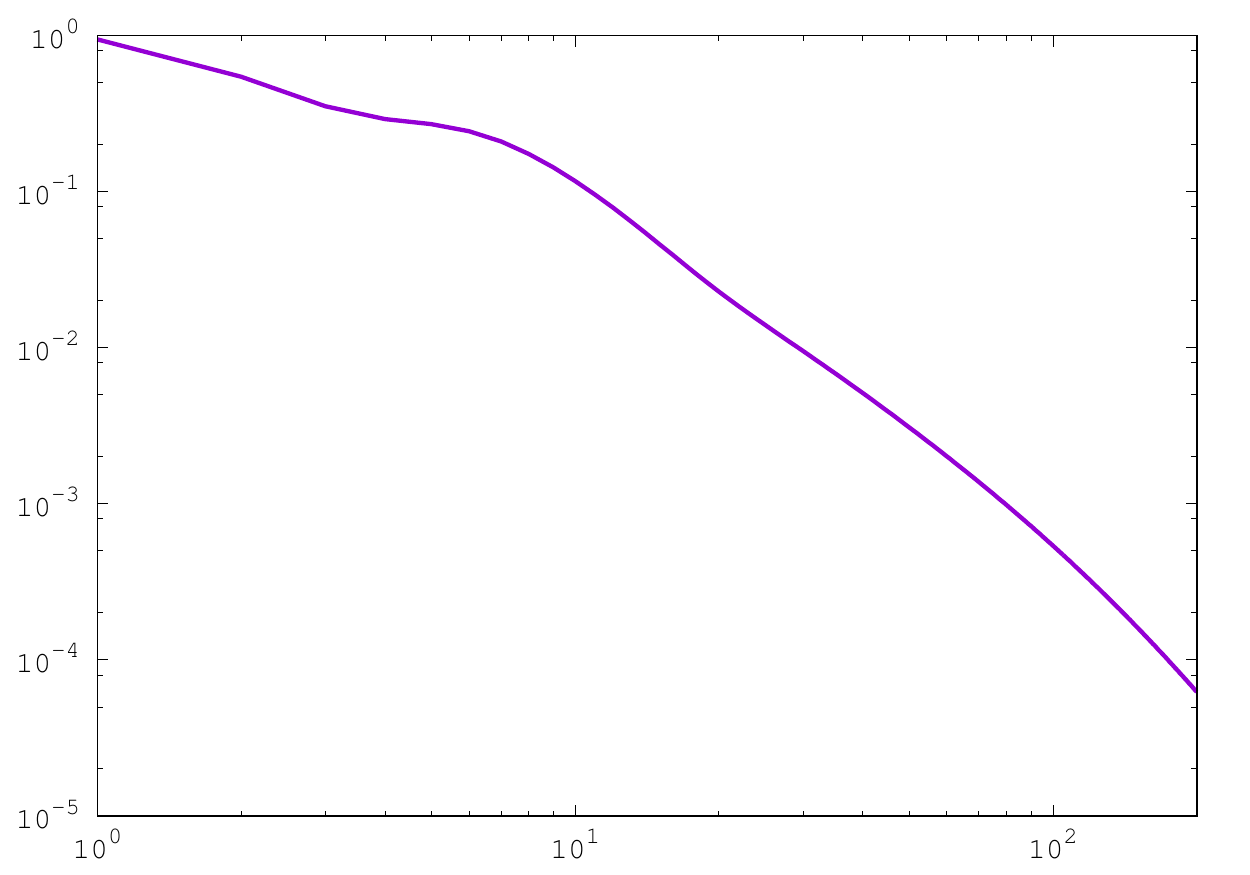}
  \caption{Discrete Bellman equation residual}
  \label{fig:cv-30-HJB}
\end{subfigure}
\begin{subfigure}{.5\textwidth}
  \centering
  \includegraphics[width=\linewidth]{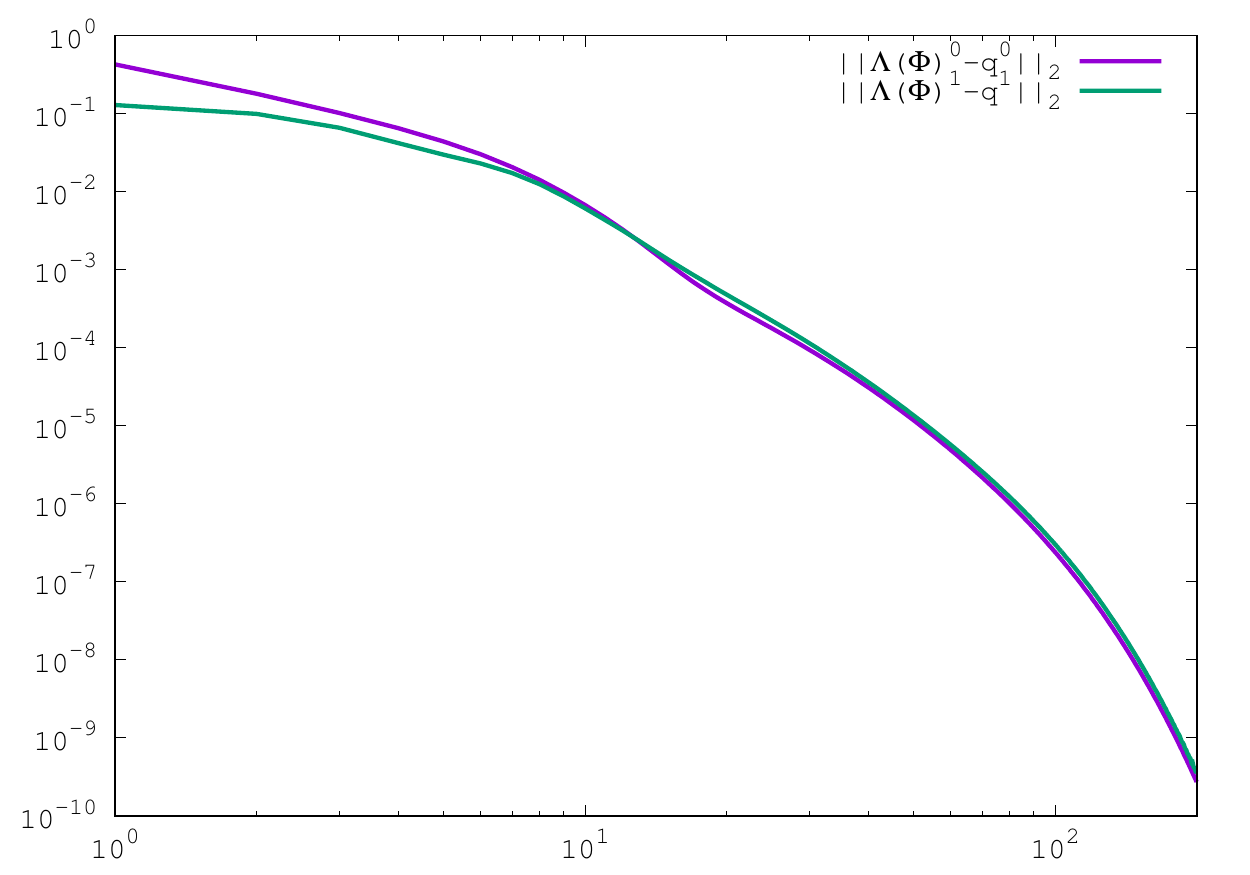}
  \caption{Error between $\Lambda(\phi)$ and $q$}
  \label{fig:cv-16-consensus}
\end{subfigure}%
\begin{subfigure}{.5\textwidth}
  \centering
  \includegraphics[width=\linewidth]{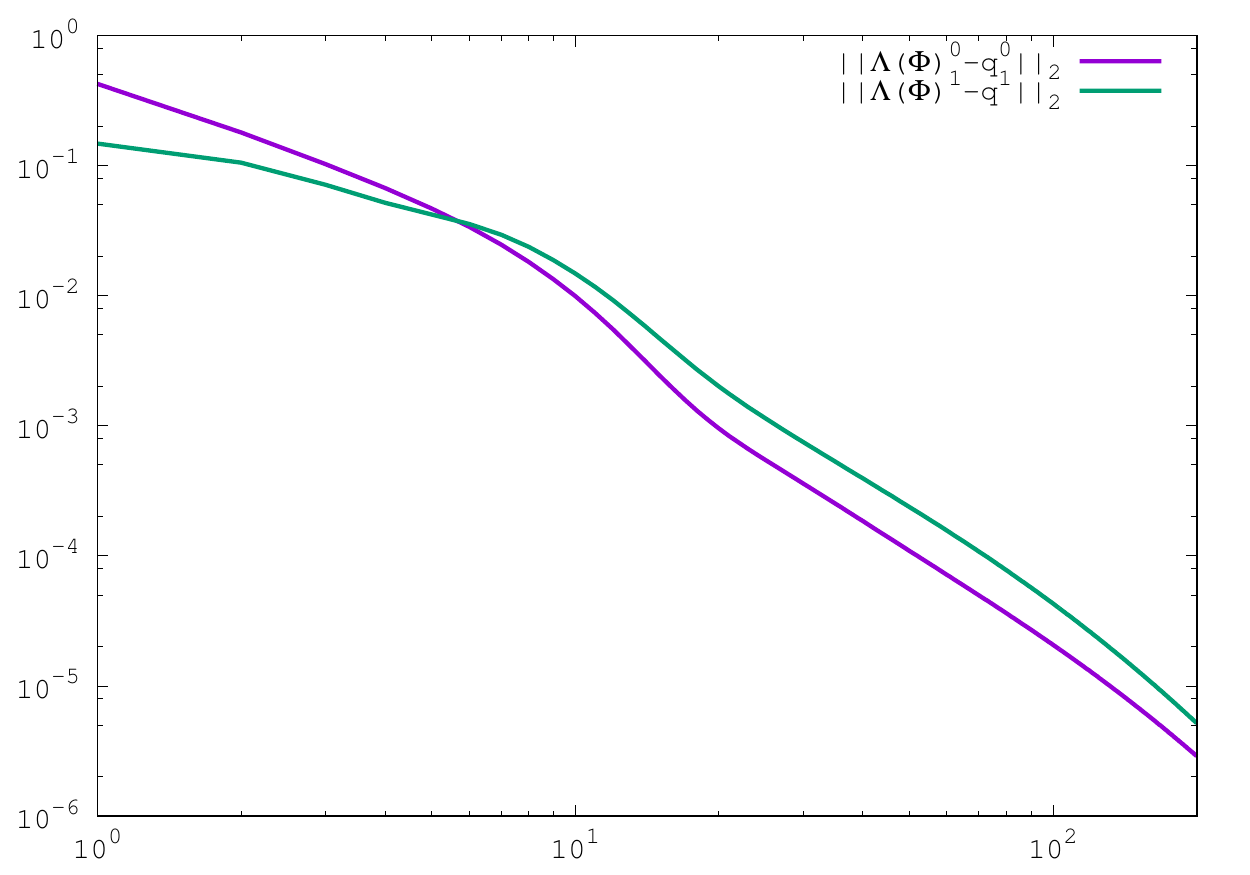}
  \caption{Error between $\Lambda(\phi)$ and $q$}
  \label{fig:cv-30-consensus}
\end{subfigure}
\begin{subfigure}{.5\textwidth}
  \centering
  \includegraphics[width=\linewidth]{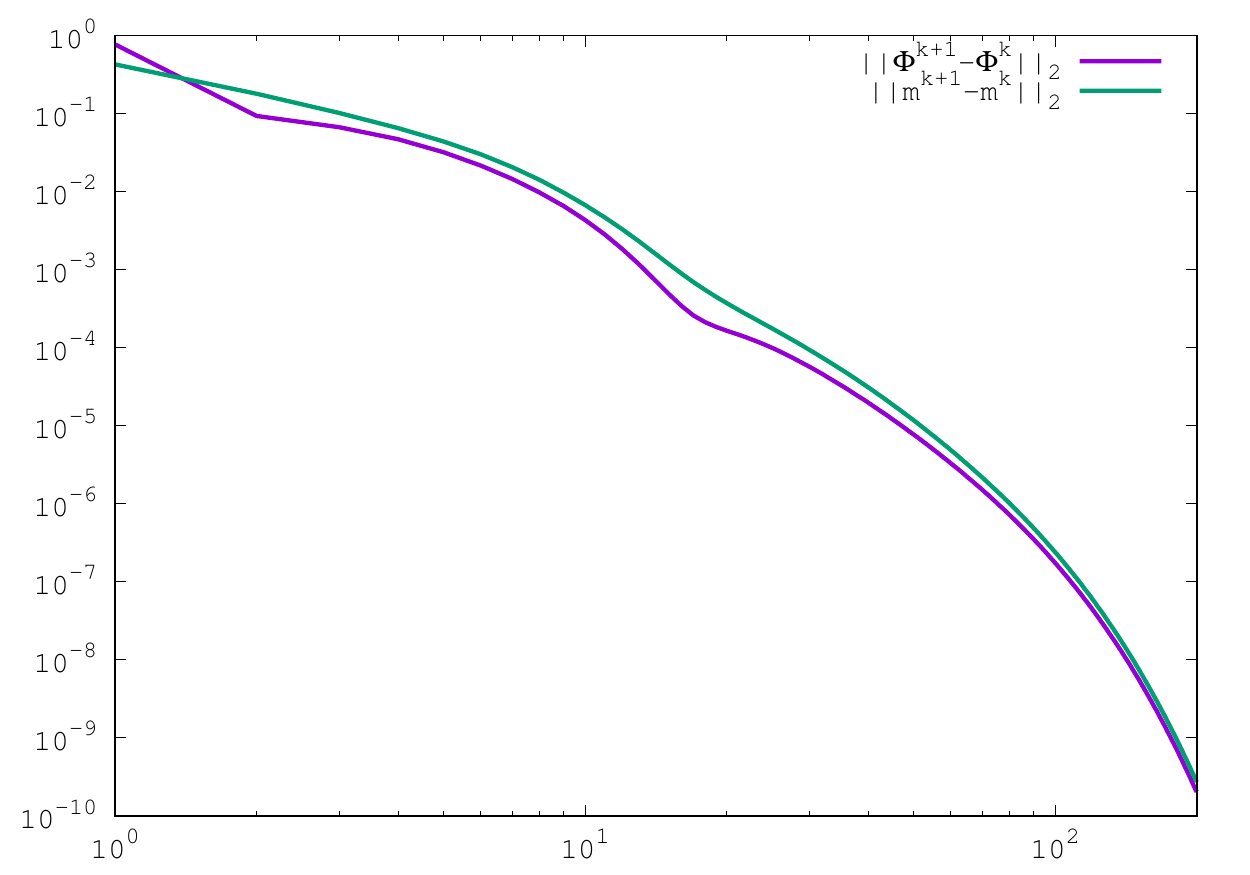}
  \caption{Error between two successive iterates of $m$ and $\phi$}
  \label{fig:cv-16-mphi}
\end{subfigure}%
\begin{subfigure}{.5\textwidth}
  \centering
  \includegraphics[width=\linewidth]{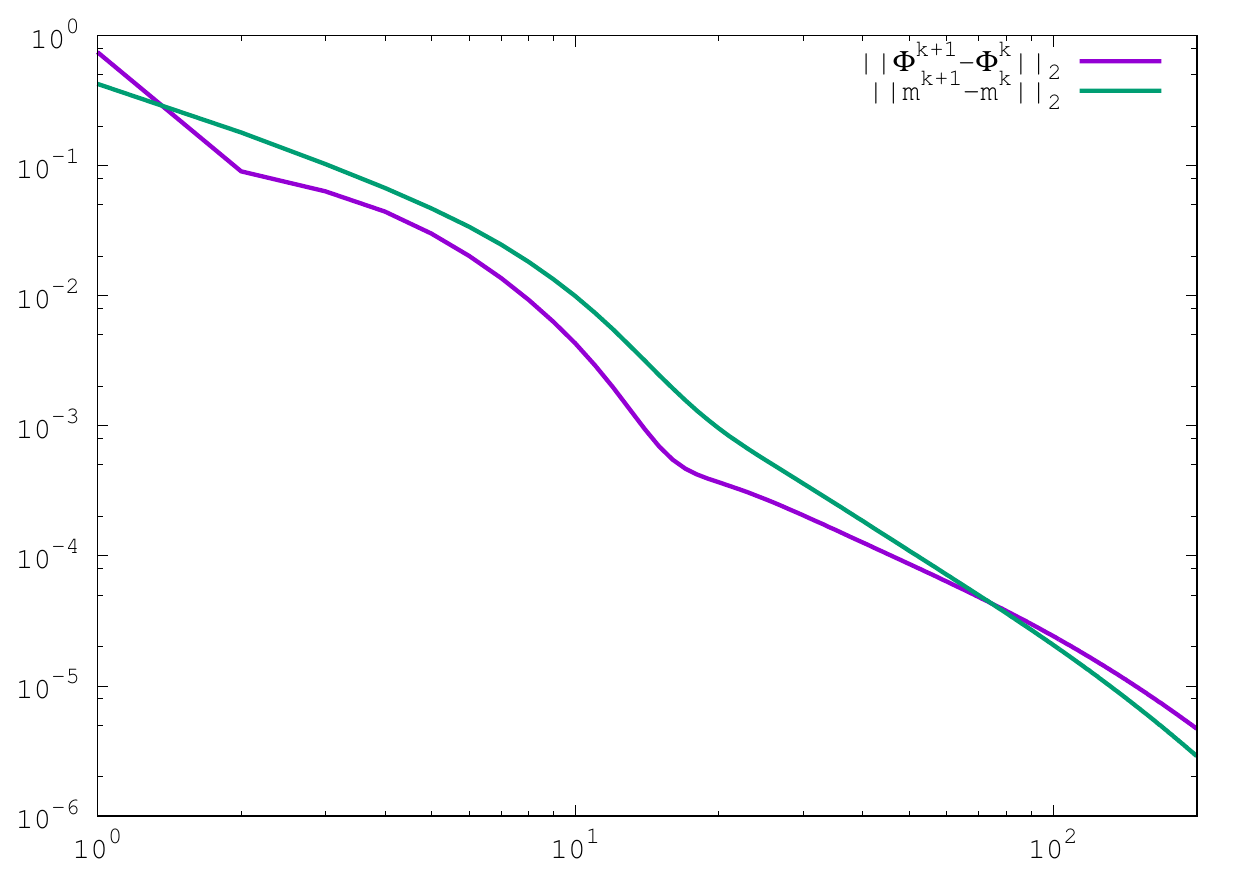}
  \caption{Error between two successive iterates of $m$ and $\phi$}
  \label{fig:cv-30-mphi}
\end{subfigure}
\caption{Convergence of the ADMM for Test case 1 (with periodic boundary conditions, $\alpha=0.5, \beta=2,$ and $\ell(x,m) = m$). Left: grid of $16 \times 16 \times 16$ nodes, right: grid of $32 \times 32 \times 32$ nodes}.
\label{fig:squareevacuation-Convergence-ell1}
\end{figure}

\begin{figure}
\centering
\begin{subfigure}{.5\textwidth}
  \centering
  \includegraphics[width=\linewidth]{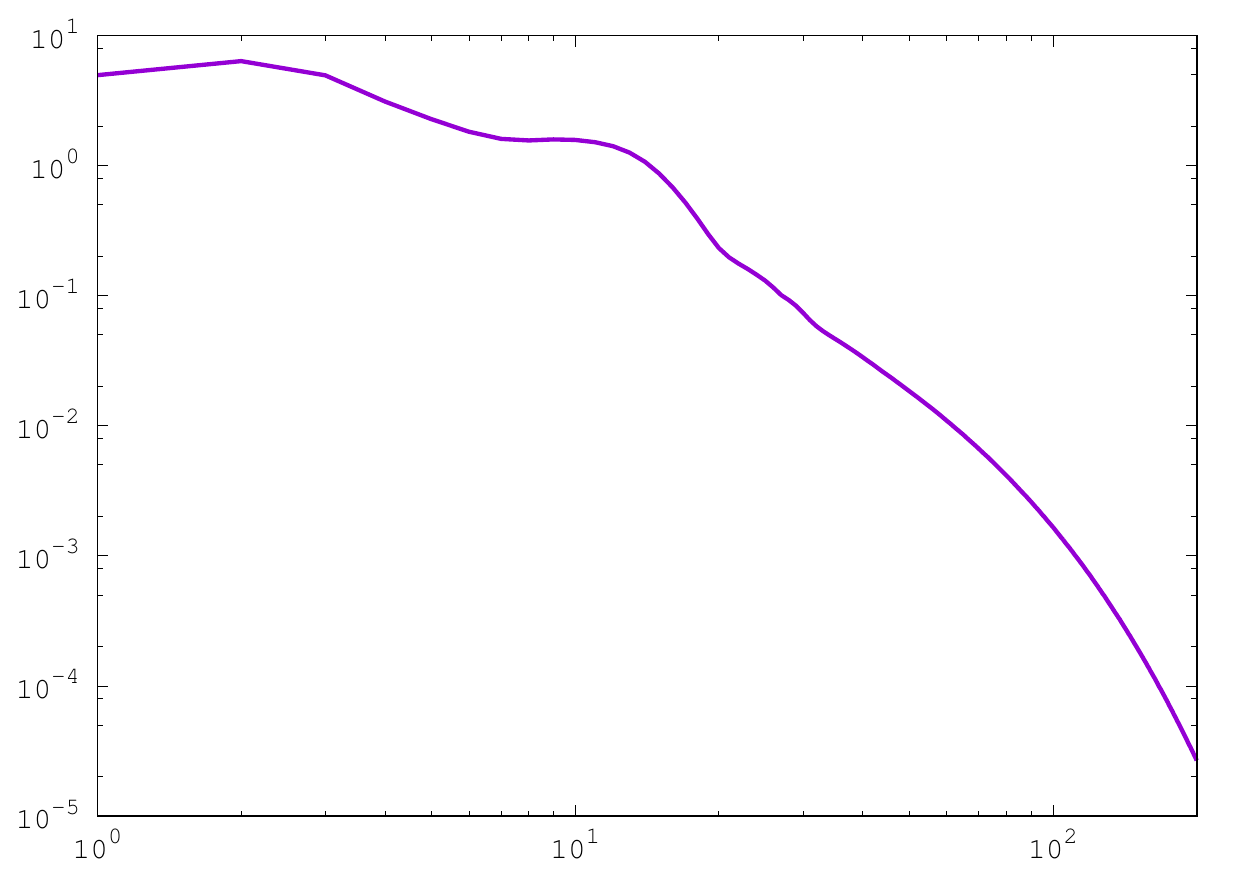}
\end{subfigure}
\begin{subfigure}{.5\textwidth}
  \centering
  \includegraphics[width=\linewidth]{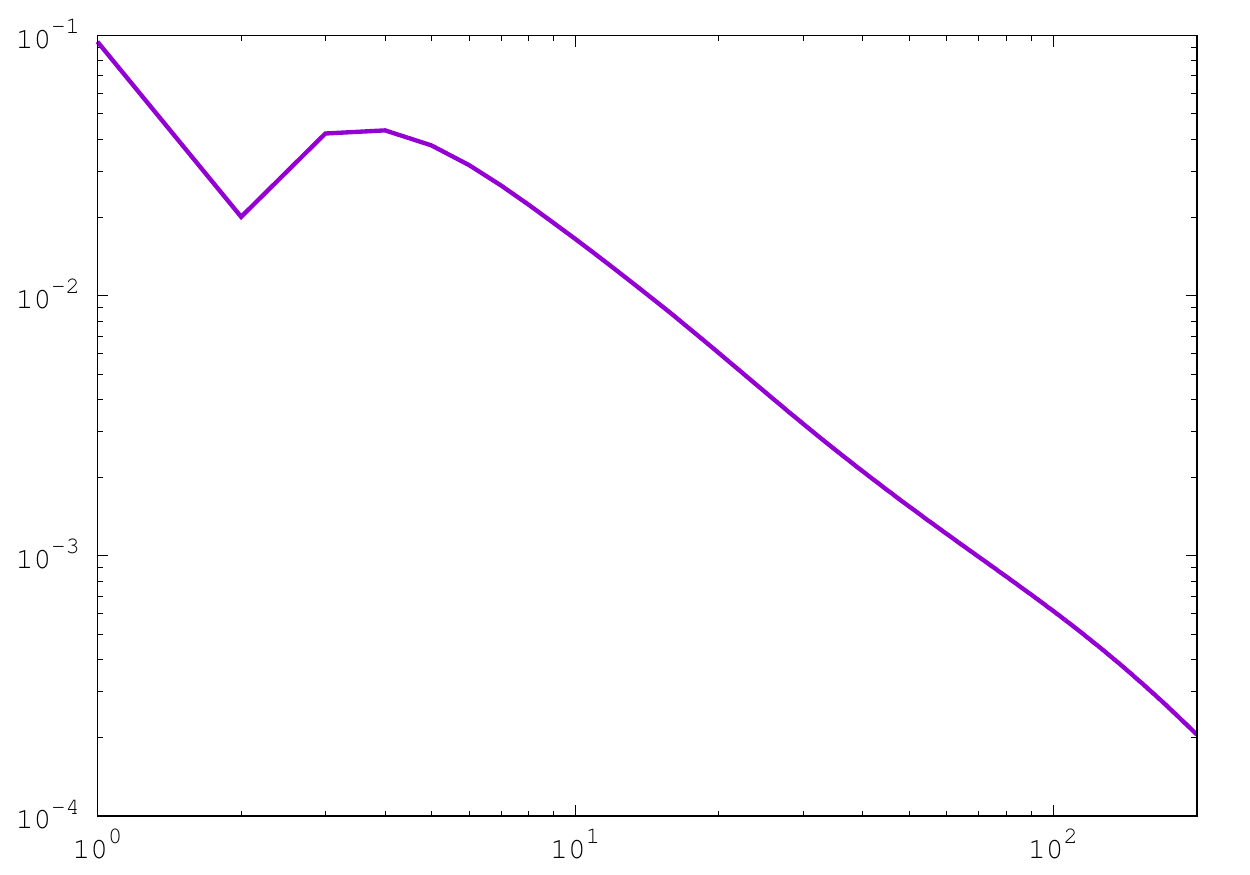}
\end{subfigure}
\caption{HJB residual for Test case 1 (with periodic boundary conditions, $\alpha=0.5, \beta=2, \ell(x,m) = m$, and a grid of $32 \times 32 \times 32$ nodes). Left: $r=0.1$, right: $r=10$.}
\label{fig:squareevacuation-Convergence-ell1-r}
\end{figure}

\subsection{State constraints}\label{subsec:numerics-sc}

\paragraph{Test case 2: from a corner to the opposite one.}

In Figure~\ref{fig:c2c-CompareSC}, we display the evolution of the distribution in Test case 2, with parameters $\alpha=0.01$, $\beta = 2$, and $\ell(x,m) = 0.001m$.

The left (resp. right) part of Figure \ref{fig:c2c-CompareSC} contains  the results obtained for periodic boundary conditions (resp. state constraint boundary condition). 
In both cases,  the population   moves from bottom left corner to the top right corner as expected due to the final cost.
With periodic conditions,  the initial location of the population and the target are close to each other in the torus;
By contrast, with state constraint boundary conditions, the population forms has to cross the unit square along its diagonal. 
Finally, we add a square obstacle at the center of the domain, i.e. the agents cannot penetrate the square
$[0.4, 0.6] \times [0.4, 0.6]$. The dynamics of the population is displayed on Figure~\ref{fig:c2c-SC-obstacle}:  we see that the population splits into two groups to circumvent the obstacle.
 \\
Finally, at time $t=T=1$ (last rows of Figure~\ref{fig:c2c-CompareSC} and \ref{fig:c2c-SC-obstacle}), 
 the mass is concentrated in the top right corner, but the distribution differ in the periodic and the state constrained cases. 
Indeed,  in the periodic case, the agents can travel in any direction and since they stop as soon as they have reached   the square $[0.8, 1] \times [0.8, 1]$, the density is higher near the corner of coordinates $(1,1)$. In the state constrained case, the agents must travel from one corner to the opposite one, and the density is higher near the   point  $(0.8,0.8)$

\begin{figure}
\centering
\begin{subfigure}{.45\textwidth}
  \centering
  \includegraphics[width=\linewidth]{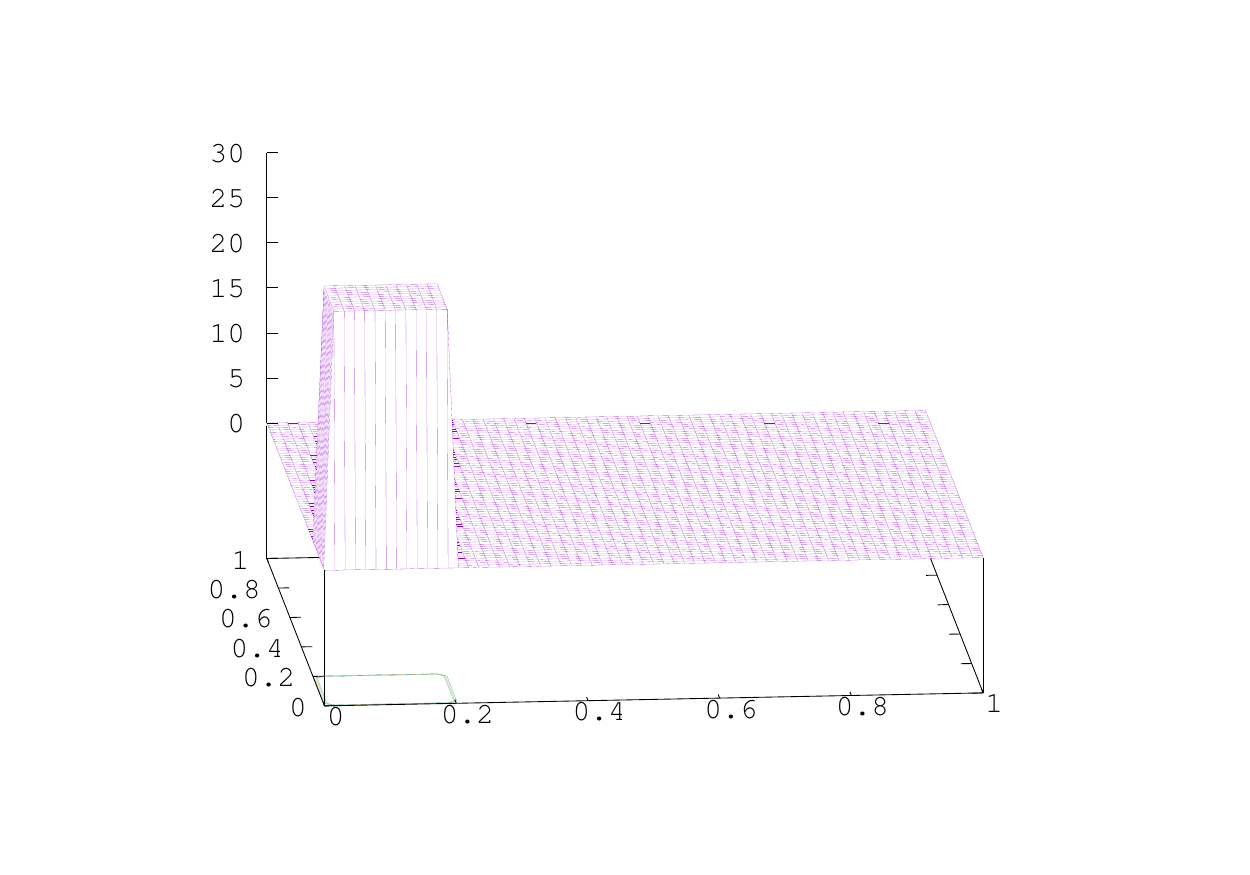}
\end{subfigure}%
\vspace{-3em}
\begin{subfigure}{.45\textwidth}
  \centering
  \includegraphics[width=\linewidth]{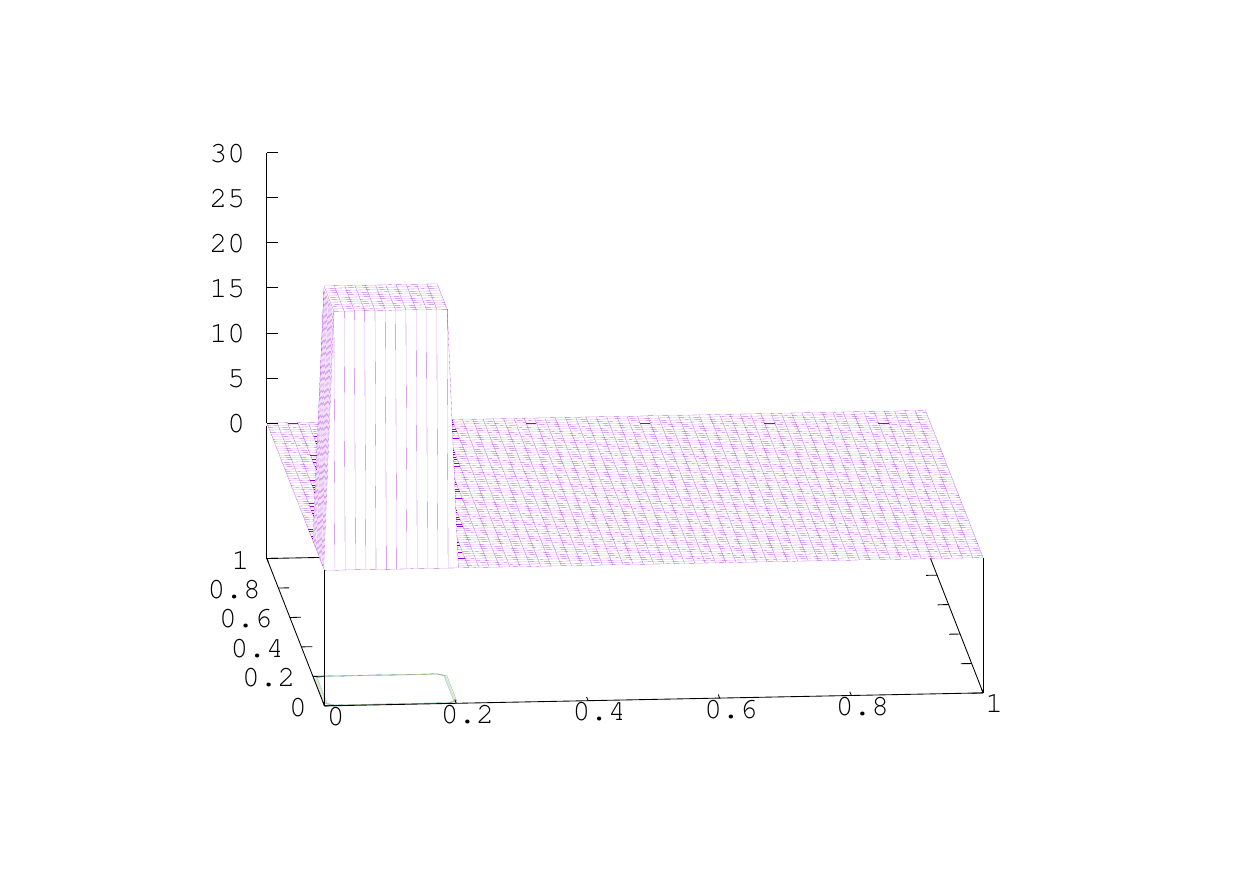}
\end{subfigure}
\vspace{-3em}
\begin{subfigure}{.45\textwidth}
  \centering
  \includegraphics[width=\linewidth]{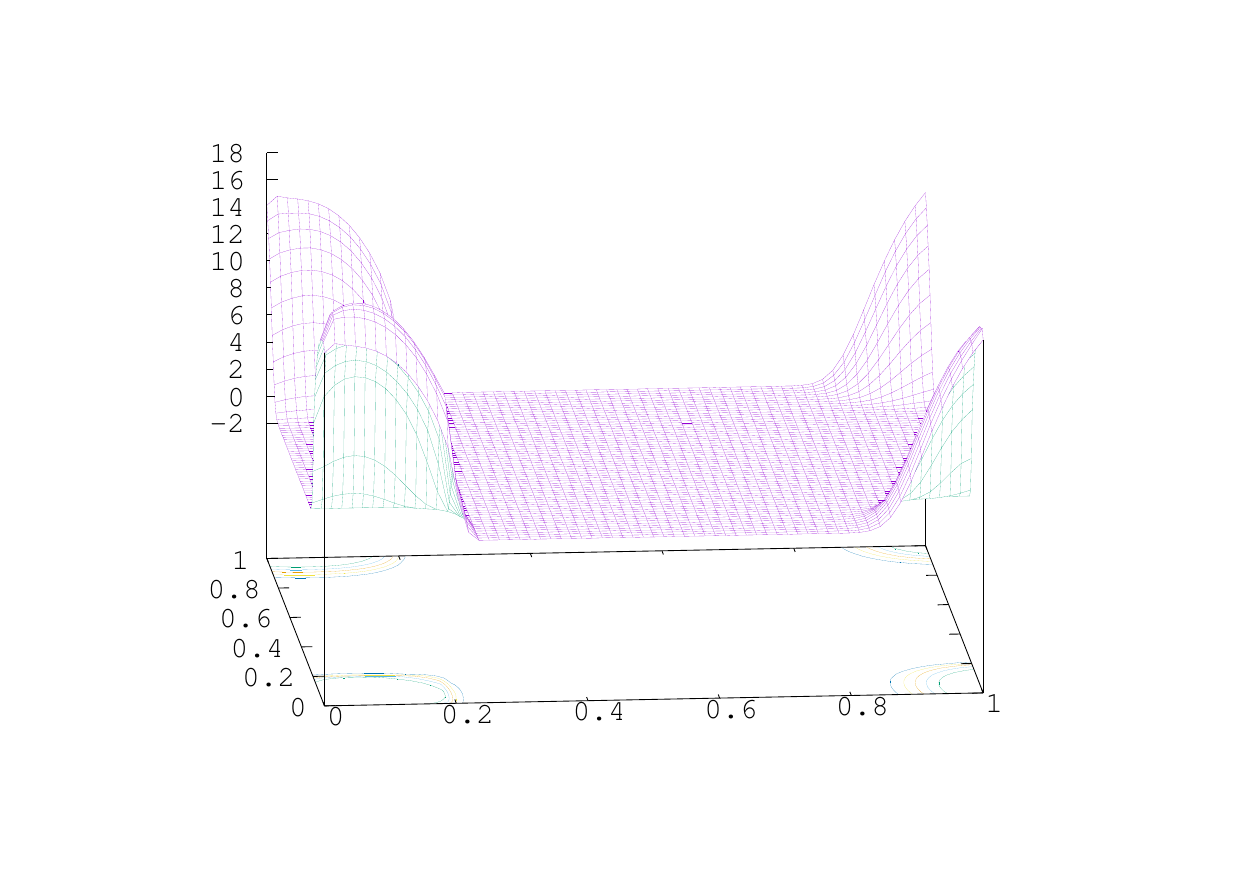}
\end{subfigure}%
\begin{subfigure}{.45\textwidth}
  \centering
  \includegraphics[width=\linewidth]{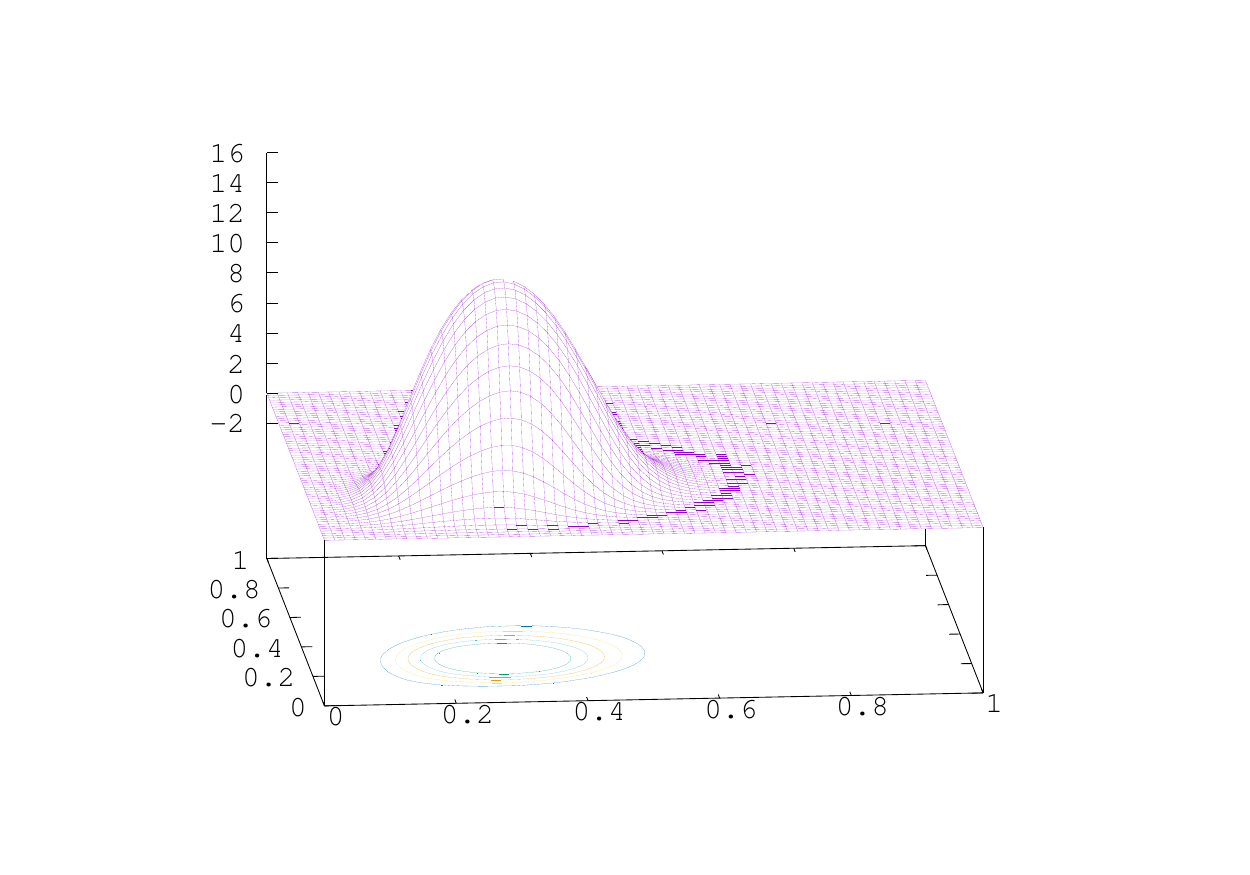}
\end{subfigure}
\vspace{-3em}
\begin{subfigure}{.45\textwidth}
  \centering
  \includegraphics[width=\linewidth]{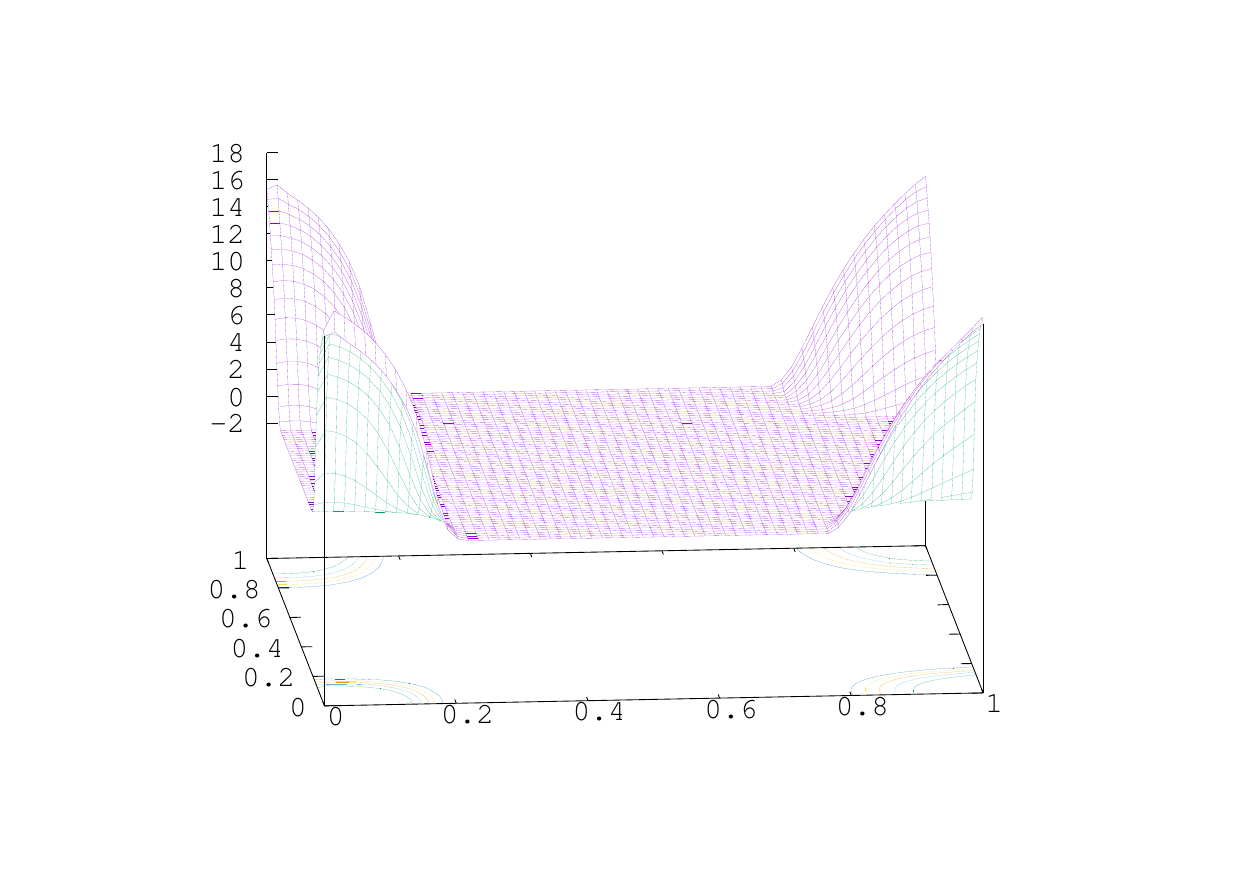}
\end{subfigure}%
\begin{subfigure}{.45\textwidth}
  \centering
  \includegraphics[width=\linewidth]{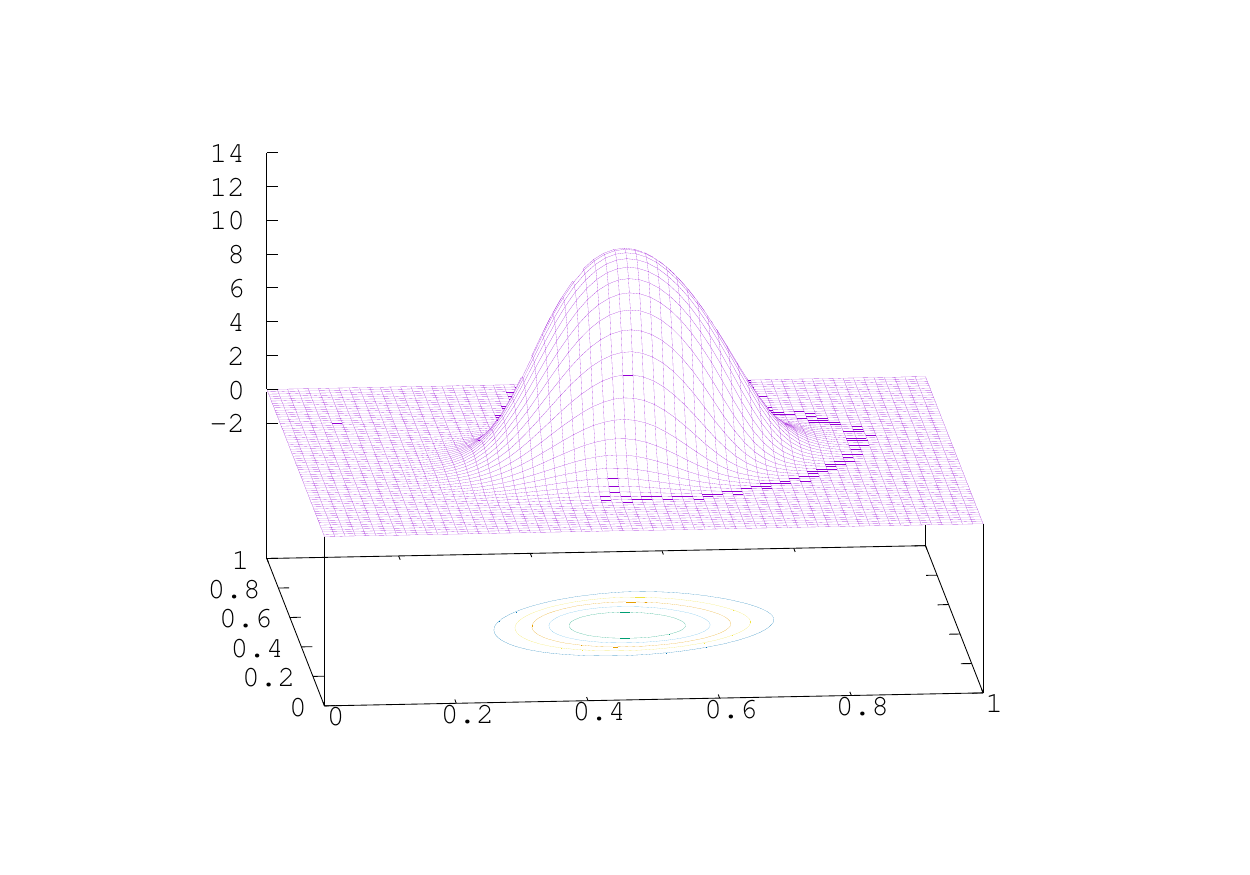}
\end{subfigure}
\vspace{-3em}
\begin{subfigure}{.45\textwidth}
  \centering
  \includegraphics[width=\linewidth]{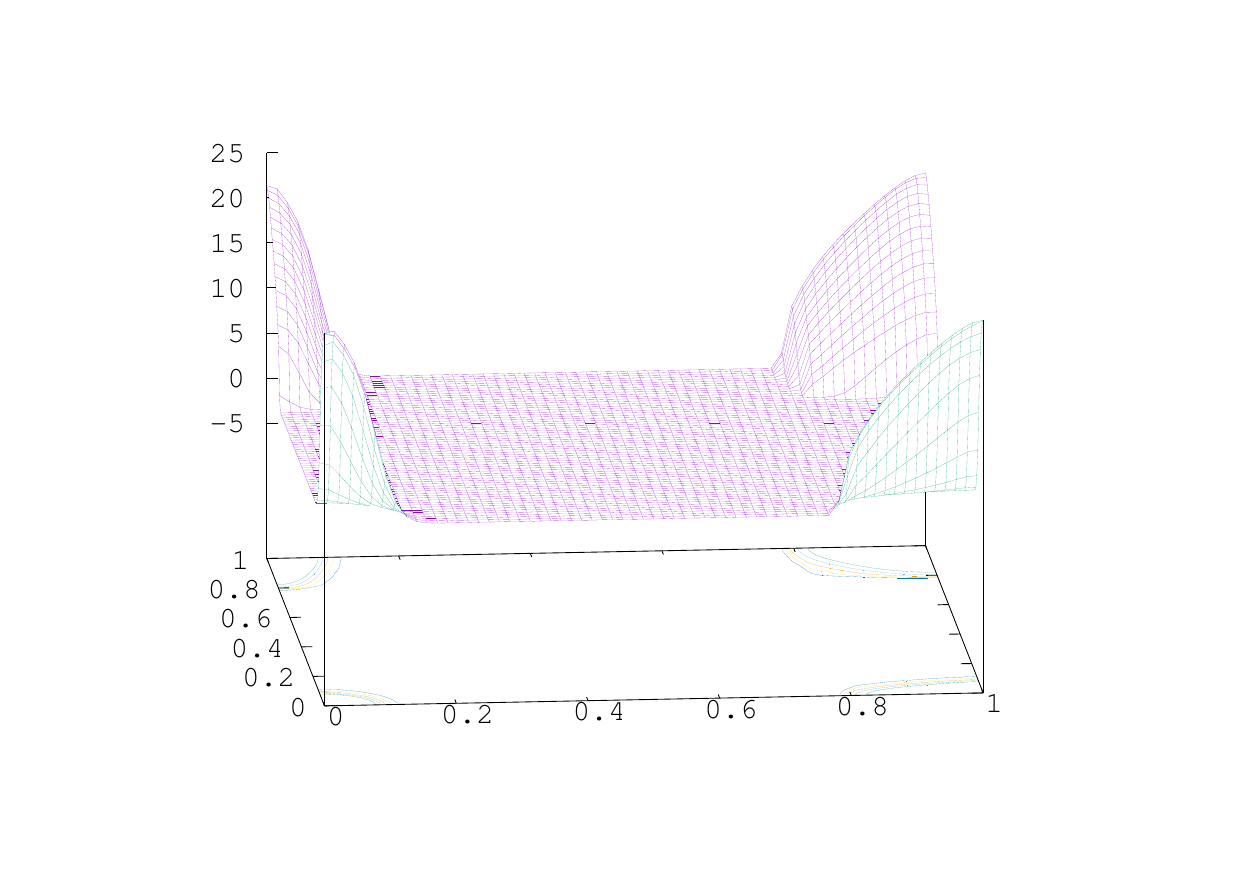}
\end{subfigure}%
\begin{subfigure}{.45\textwidth}
  \centering
  \includegraphics[width=\linewidth]{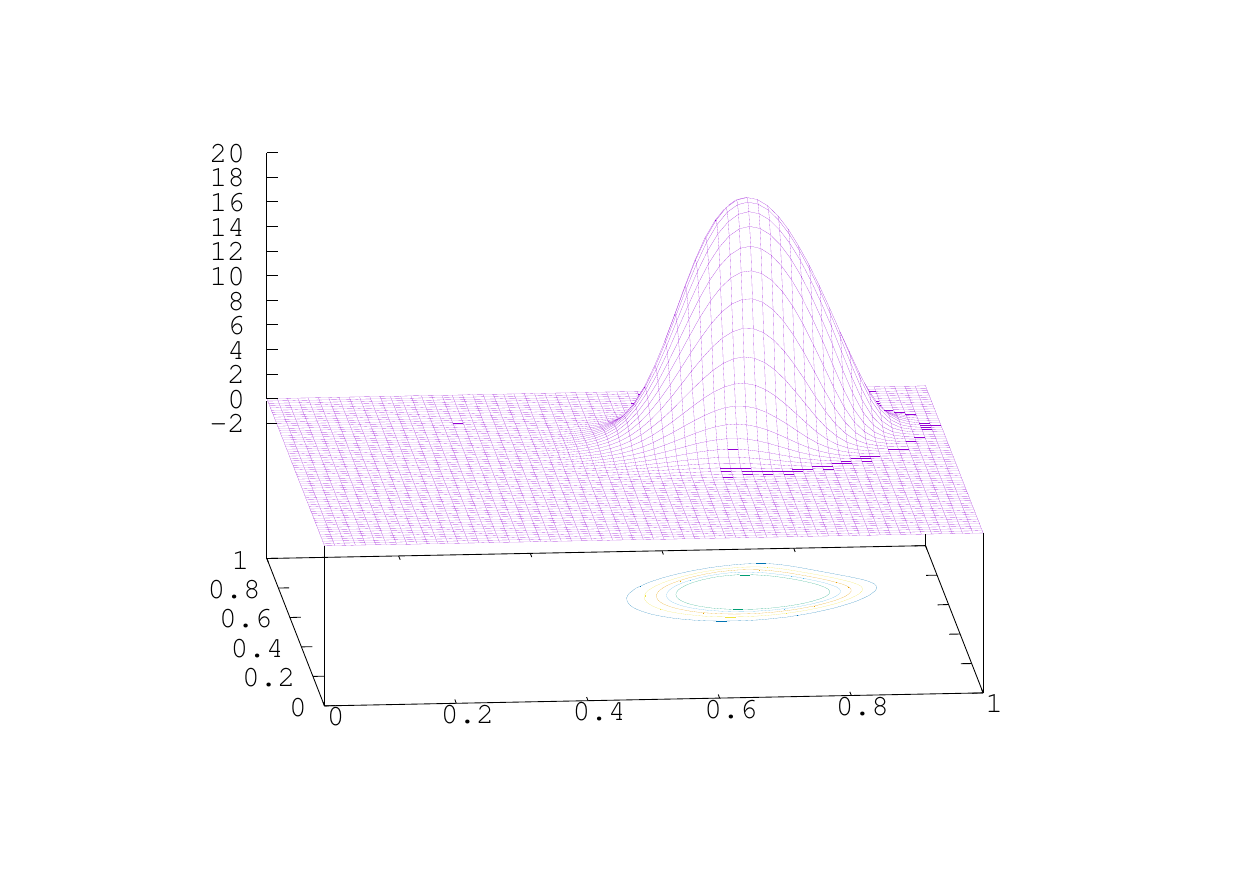}
\end{subfigure}
%
\begin{subfigure}{.45\textwidth}
  \centering
  \includegraphics[width=\linewidth]{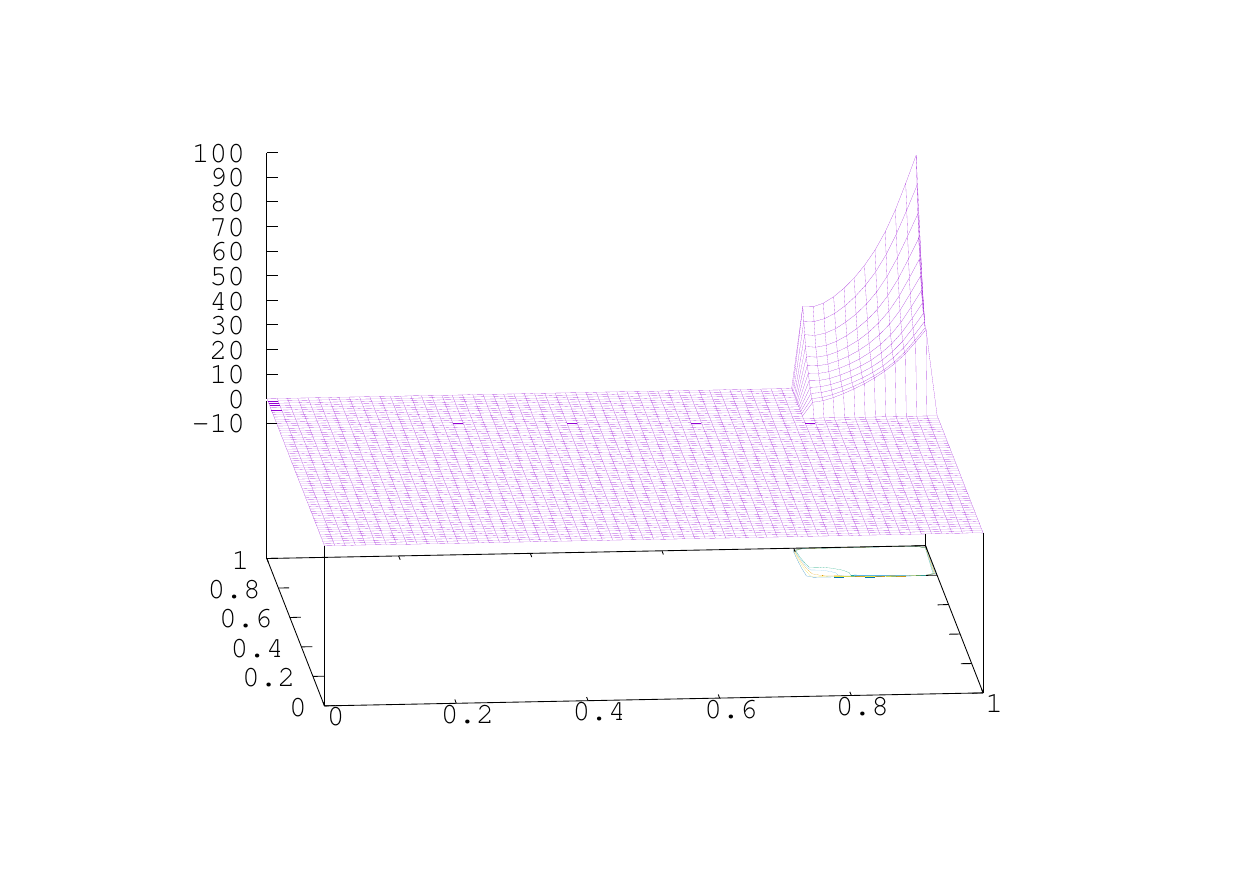}
\end{subfigure}%
\vspace{-3em}
\begin{subfigure}{.45\textwidth}
  \centering
  \includegraphics[width=\linewidth]{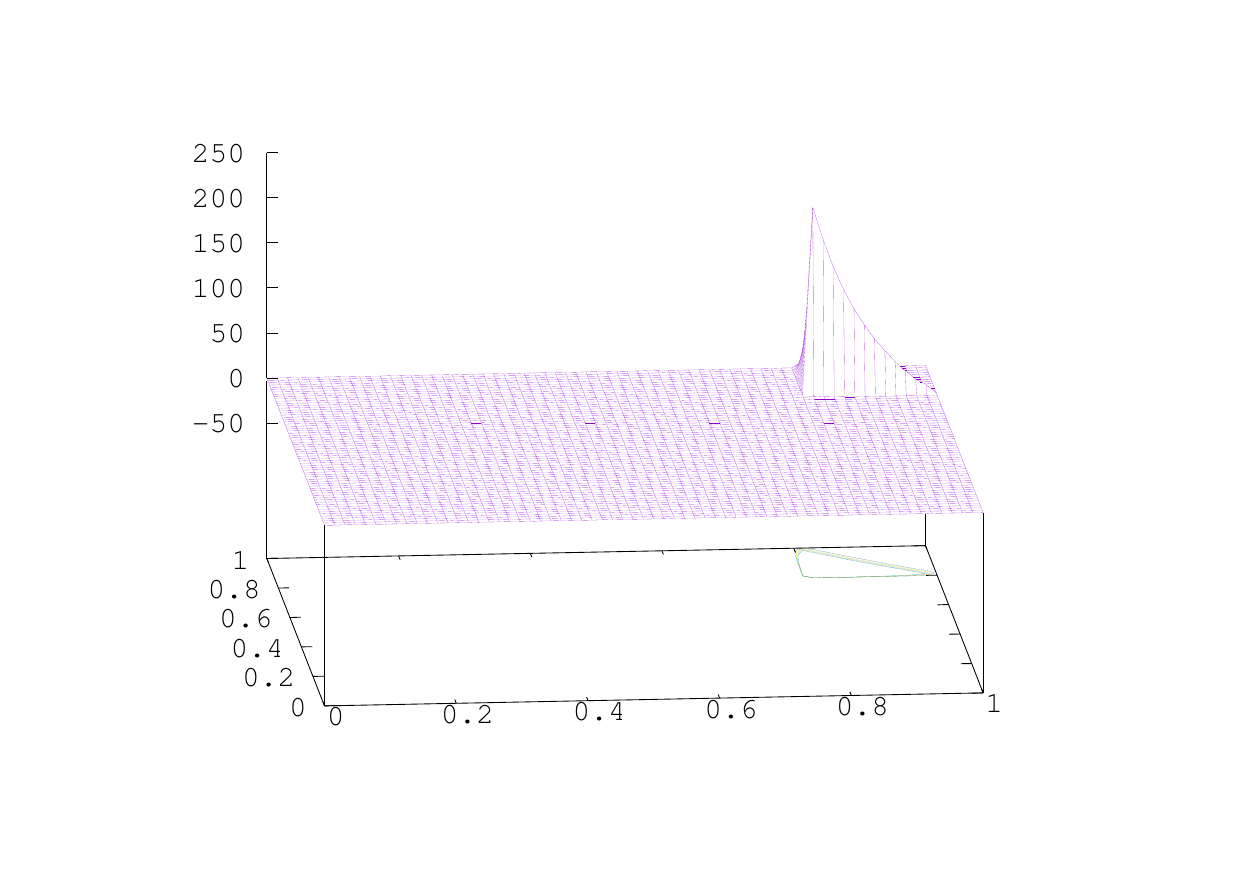}
\end{subfigure}
\caption{Test case 2. Left: periodic case, right: state constrained to   $[0,1]^2$ (with $\alpha = 0.01, \beta = 2, \ell(x,m) = 0.001m$). The rows correspond to $t=0, 1/4, 1/2, 3/4,$ and $T=1$.
 The contours lines correspond to levels  $2,4,6,8,$ and $10$.}
\label{fig:c2c-CompareSC}
\end{figure}

\begin{figure}
\centering
\begin{subfigure}{.45\textwidth}
  \centering
  \includegraphics[width=\linewidth]{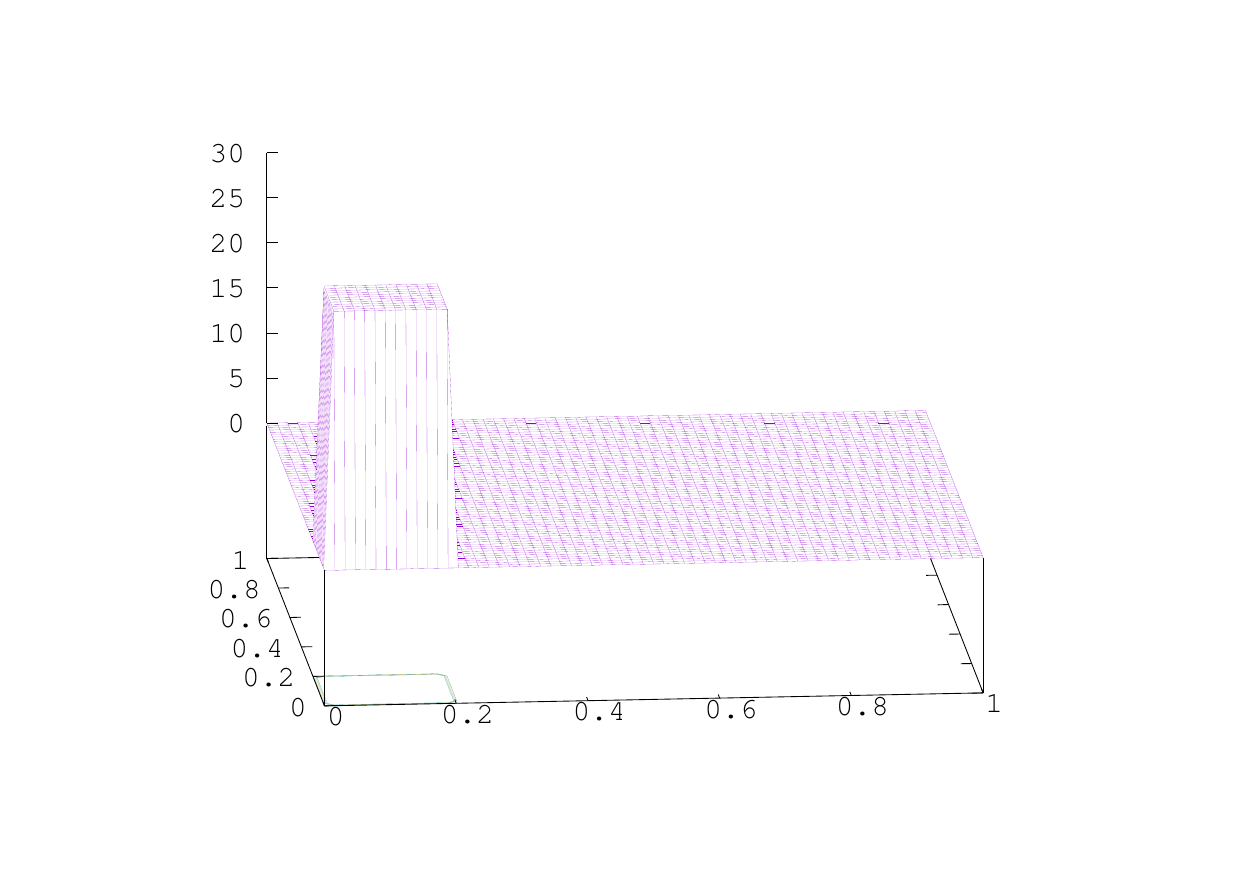}
\end{subfigure}
\vspace{-3em}%
\begin{subfigure}{.45\textwidth}
  \centering
  \includegraphics[width=.75\linewidth]{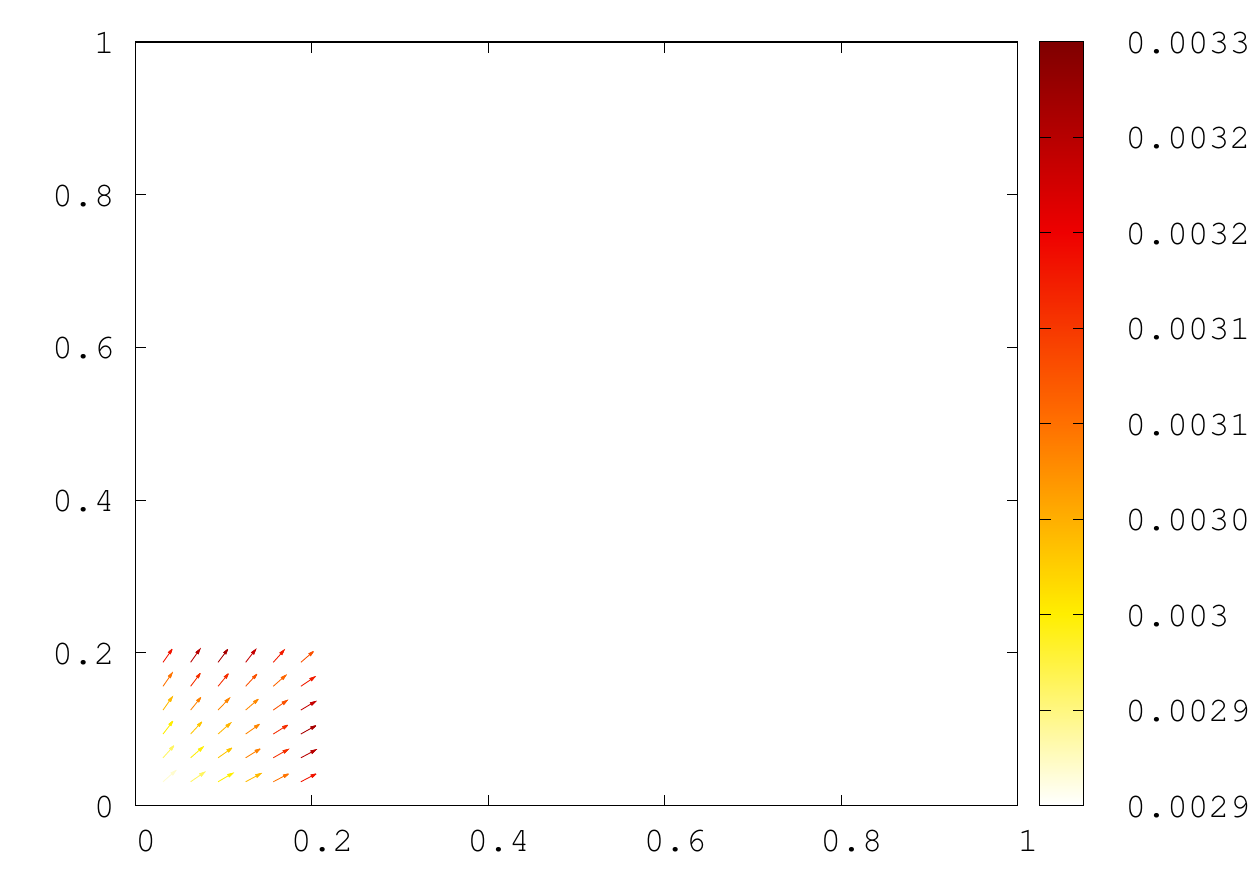}
\end{subfigure}
\vspace{-3em}
\begin{subfigure}{.45\textwidth}
  \centering
  \includegraphics[width=\linewidth]{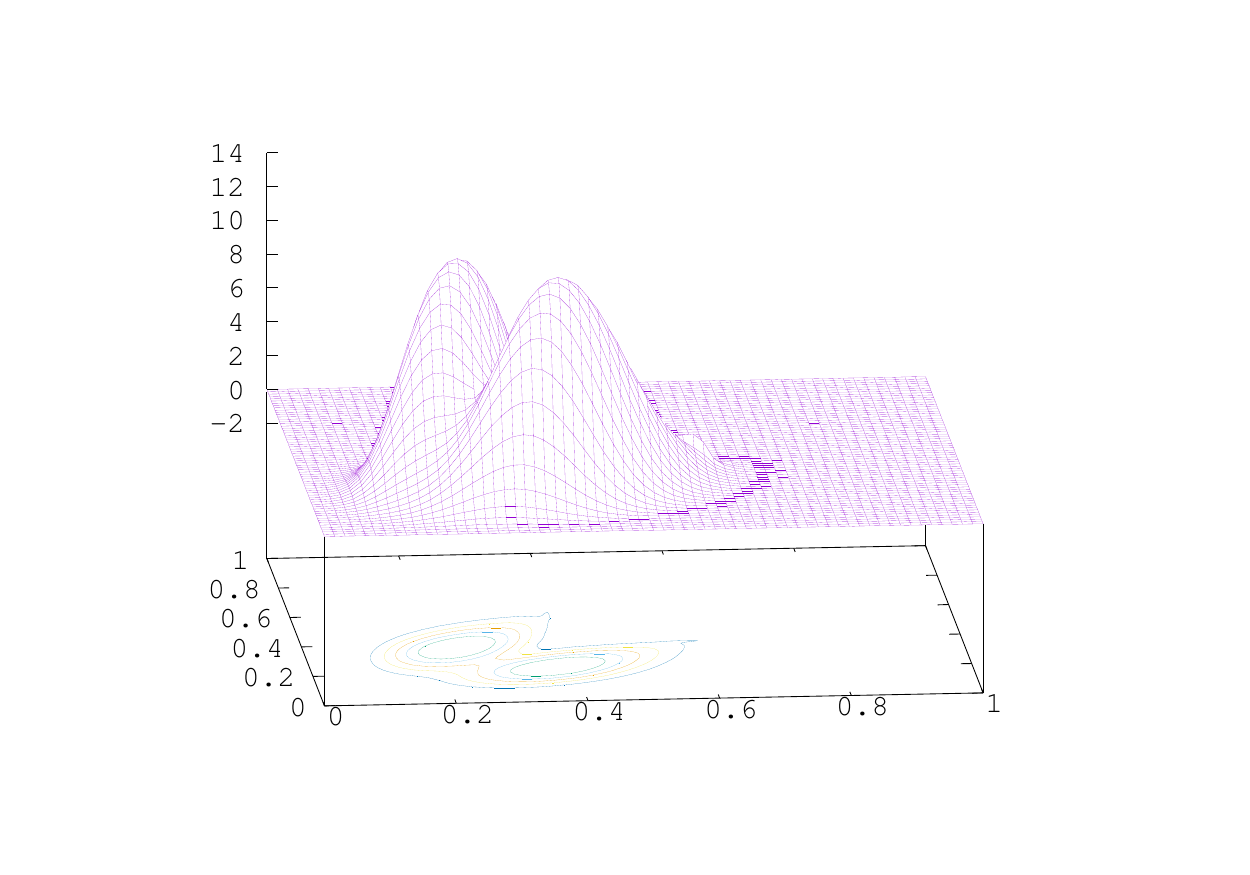}
\end{subfigure}
\begin{subfigure}{.45\textwidth}
  \centering
  \includegraphics[width=.75\linewidth]{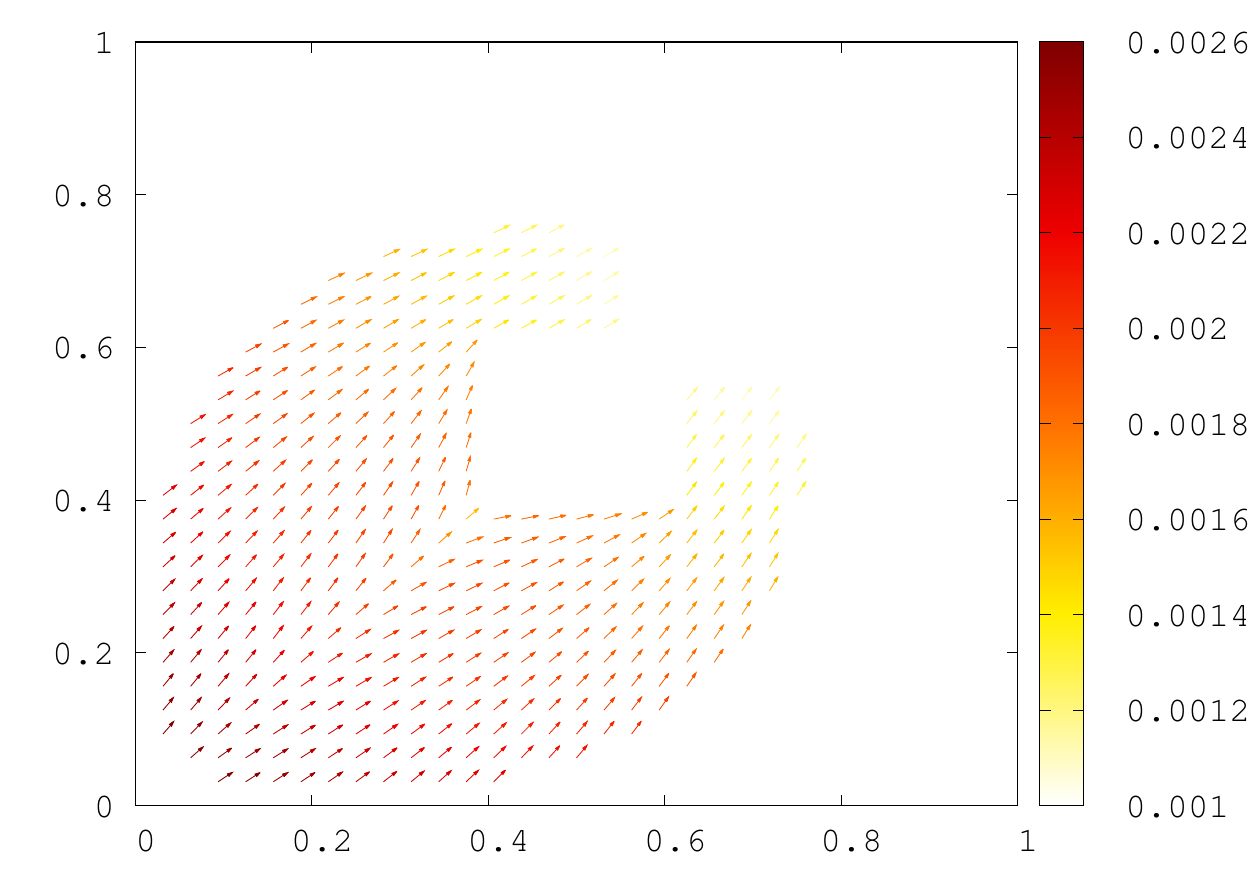}
\end{subfigure}
\vspace{-3em}
\begin{subfigure}{.45\textwidth}
  \centering
  \includegraphics[width=\linewidth]{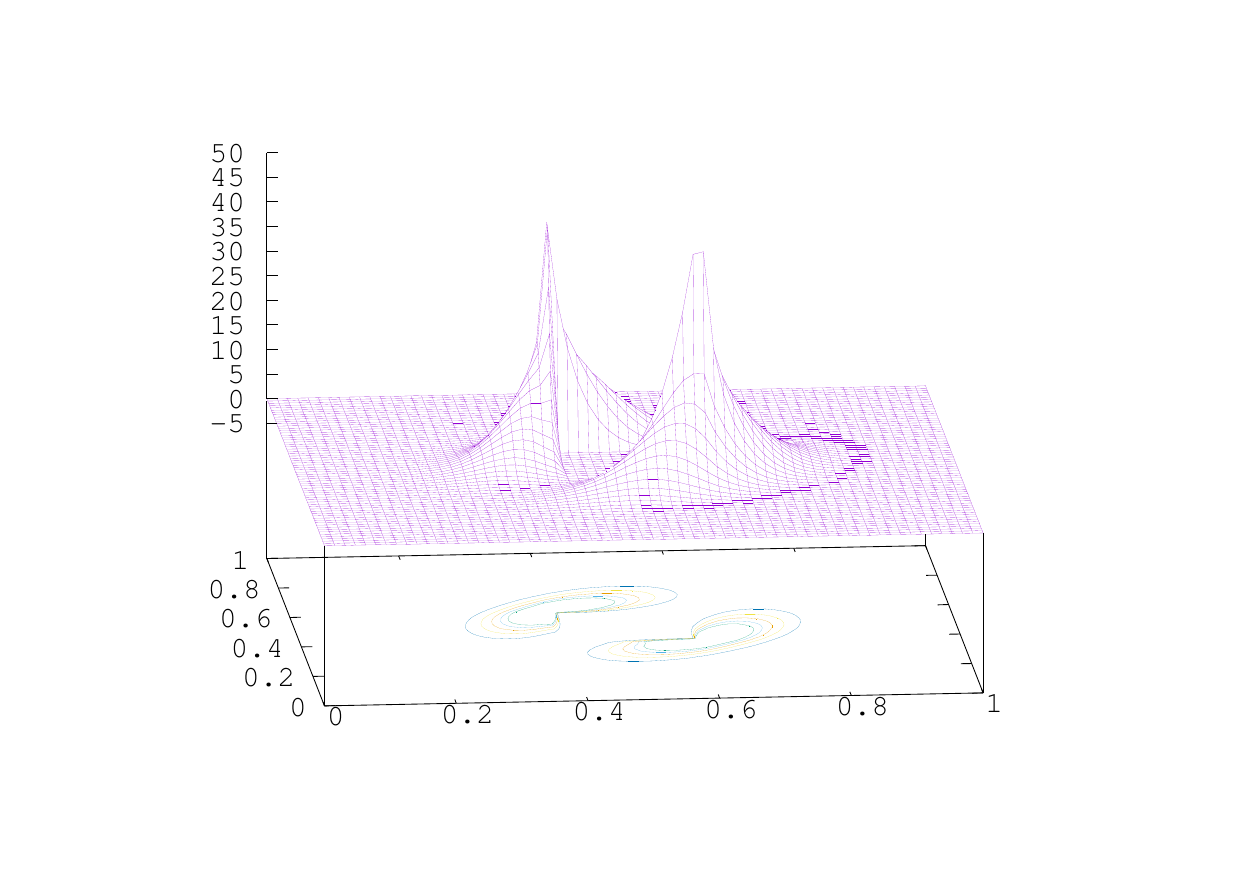}
\end{subfigure}
\begin{subfigure}{.45\textwidth}
  \centering
  \includegraphics[width=.75\linewidth]{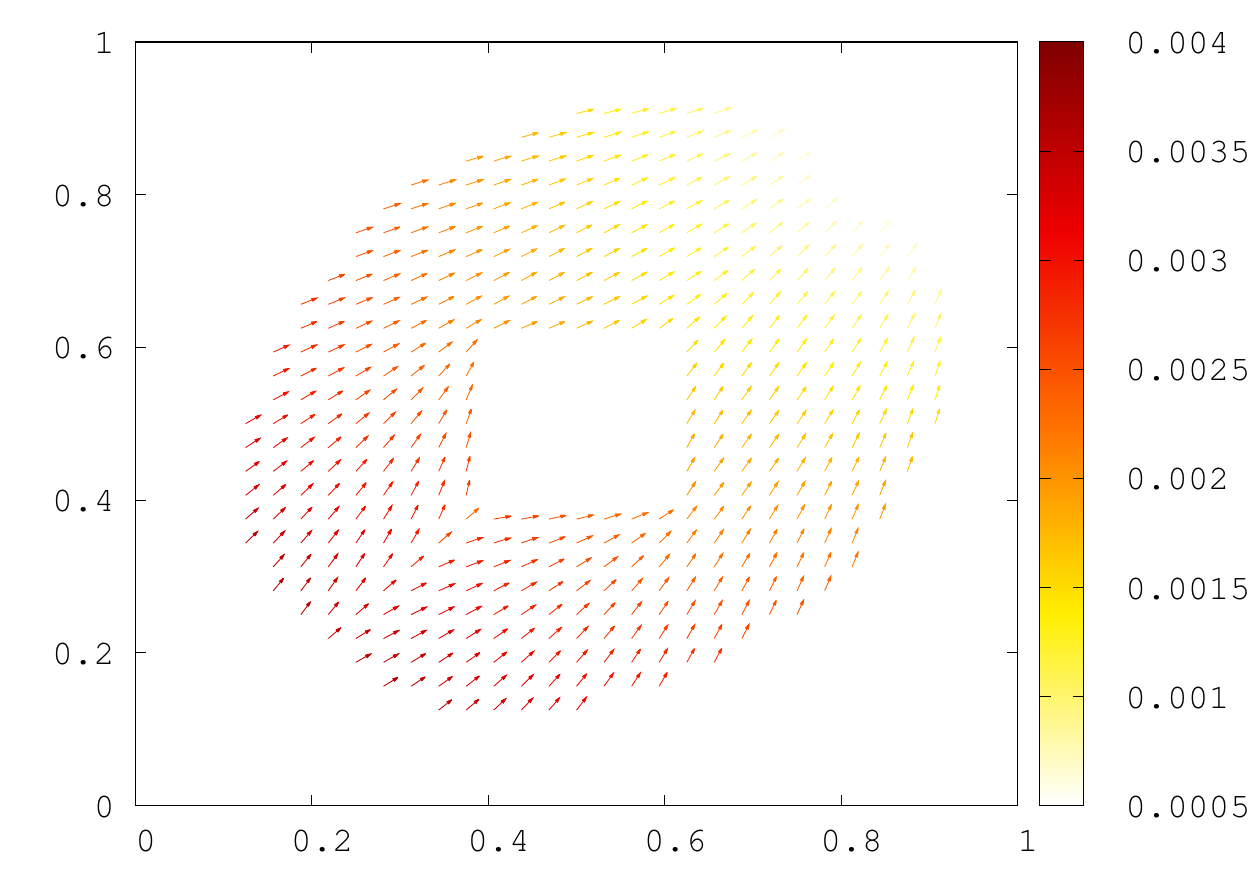}
\end{subfigure}
\vspace{-3em}
\begin{subfigure}{.45\textwidth}
  \centering
  \includegraphics[width=\linewidth]{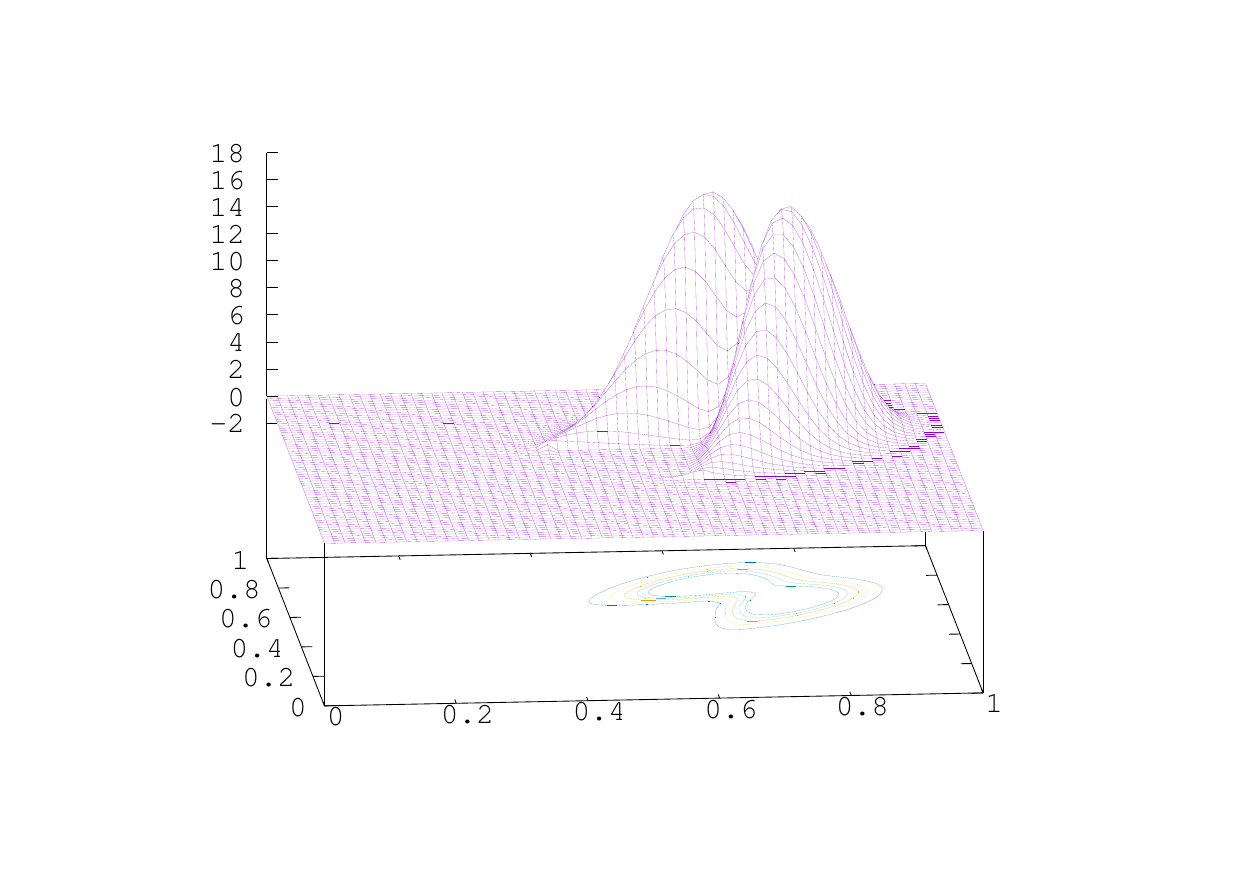}
\end{subfigure}
\begin{subfigure}{.45\textwidth}
  \centering
  \includegraphics[width=.75\linewidth]{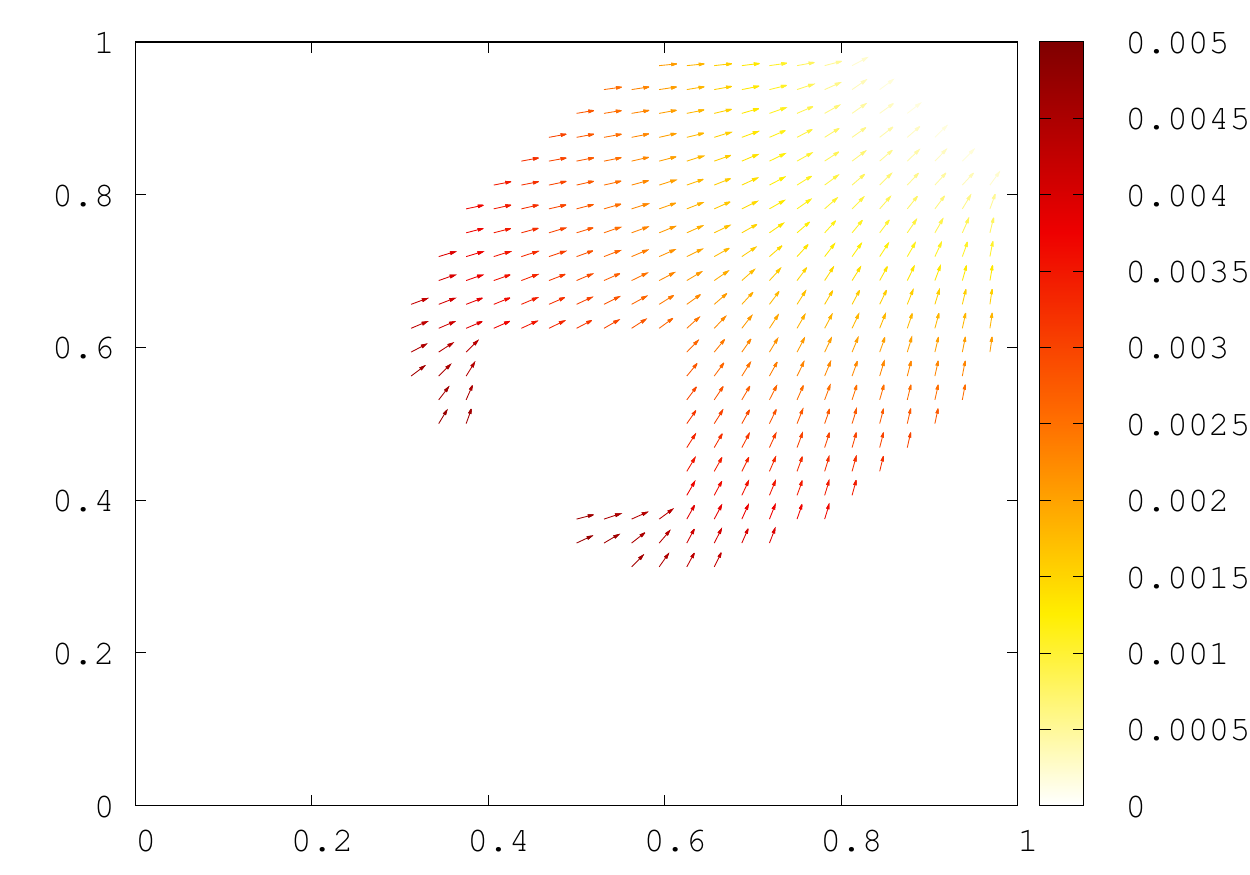}
\end{subfigure}
\begin{subfigure}{.45\textwidth}
  \centering
  \includegraphics[width=\linewidth]{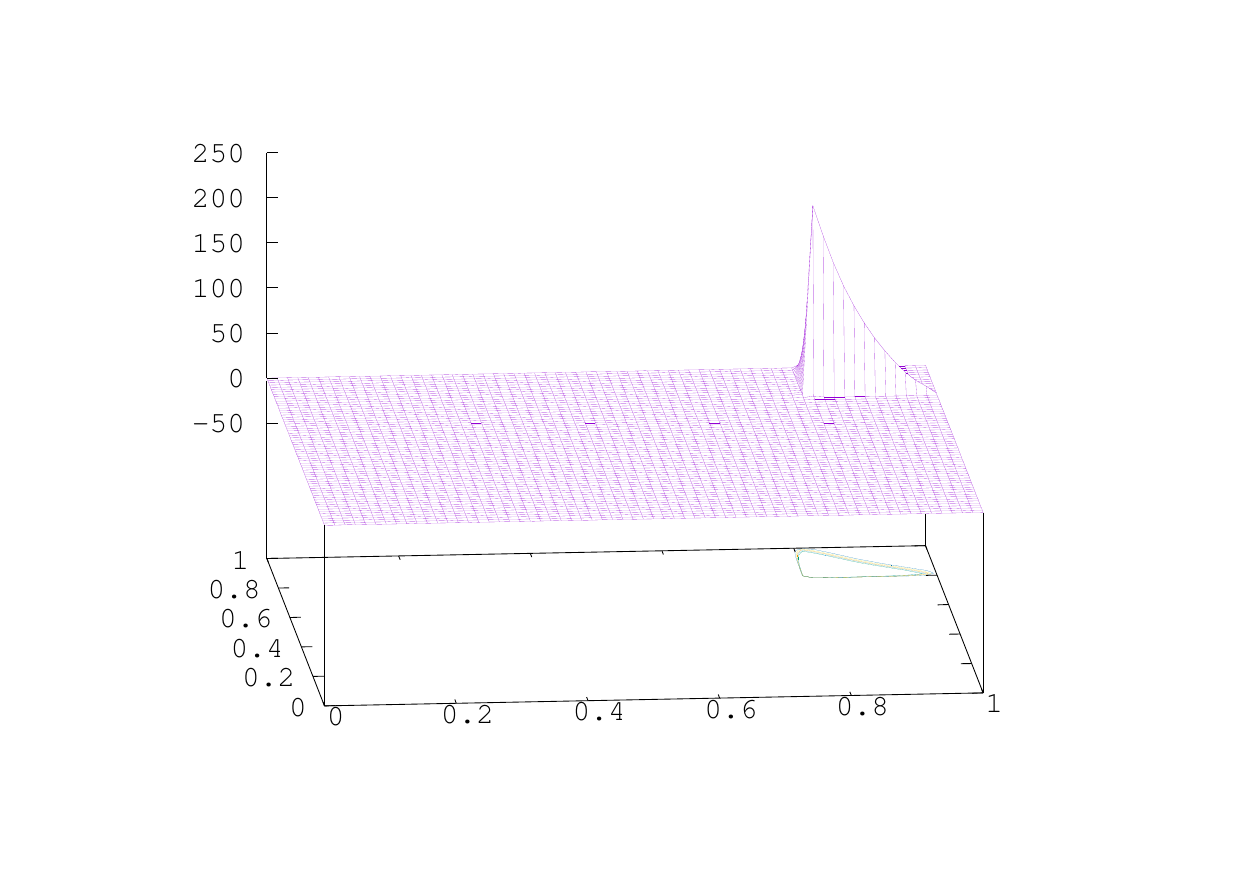}
\end{subfigure}
\begin{subfigure}{.45\textwidth}
  \centering
  \includegraphics[width=.75\linewidth]{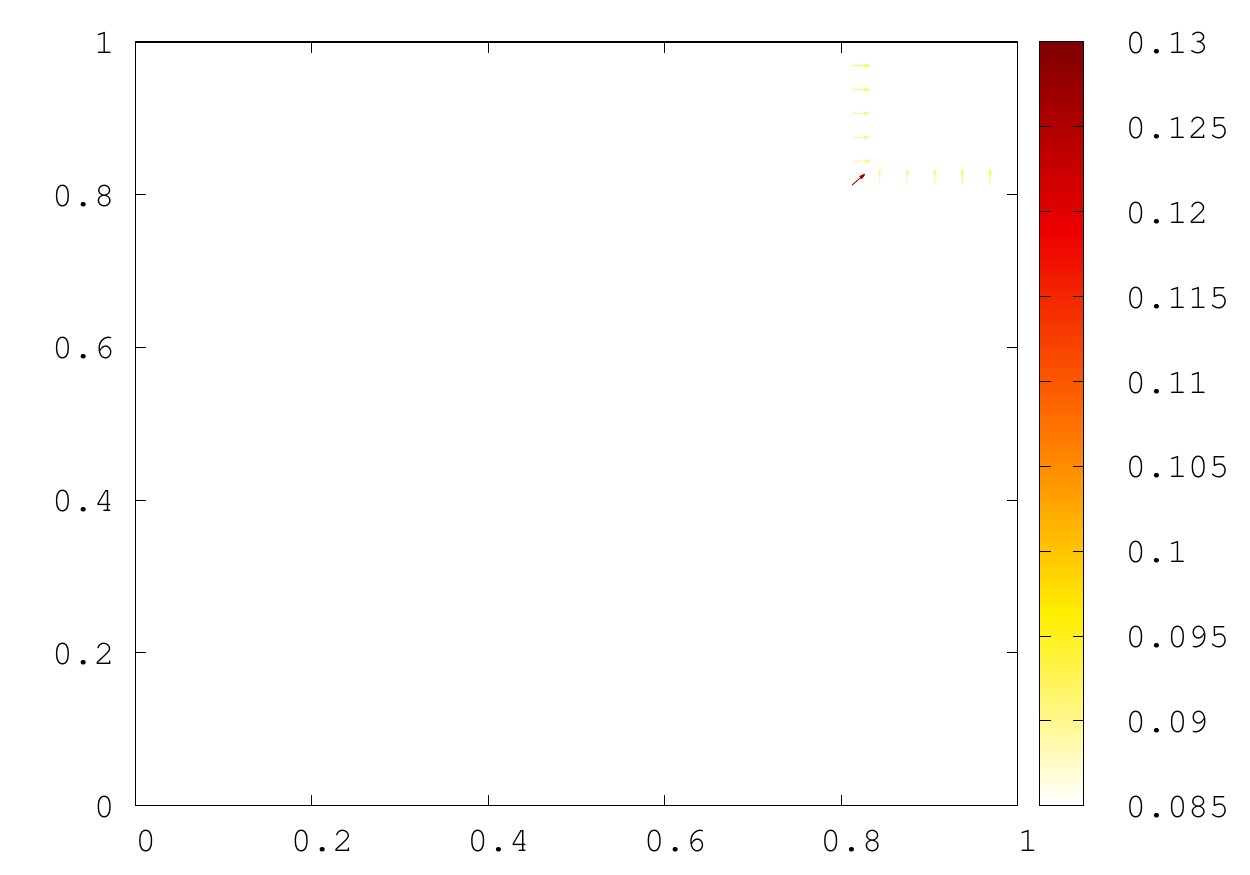}
\end{subfigure}
\caption{Test case 2 with state constrained to   $[0,1]^2 \setminus (0.4,0.6)^2$ ($\alpha = 0.01, \beta = 2, \ell(x,m) = 0.001m$). The rows correspond to $t=0, 1/4, 1/2, 3/4,$ and $T=1$.
Left column: density ( The contours lines correspond to levels  $2,4,6,8,$ and $10$), right column: velocity field.}
\label{fig:c2c-SC-obstacle}
\end{figure}

\subsection{Influence of the parameters}\label{subsec:influenceParams}

\paragraph{Impact of $\alpha$.}

Figure~\ref{fig:diracbosse-CompareAlpha-ell001} shows the evolution of the distribution in Test case 3 
with state constraints,  $\beta = 2, \ell(m) = 0.01m$ and for the exponents $\alpha = 0.3$ and $\alpha = 0.7$. 
For $\alpha=0.3$, the distribution moves faster to the border of the target. Moreover the peak vanishes quickly in this case,
 in comparison with the case $\alpha = 0.7$.  This can be explained by the fact that a smaller value of $\alpha$   makes motion less expensive in congested zones.
\begin{figure}
\centering
\begin{subfigure}{.45\textwidth}
  \centering
  \includegraphics[width=\linewidth]{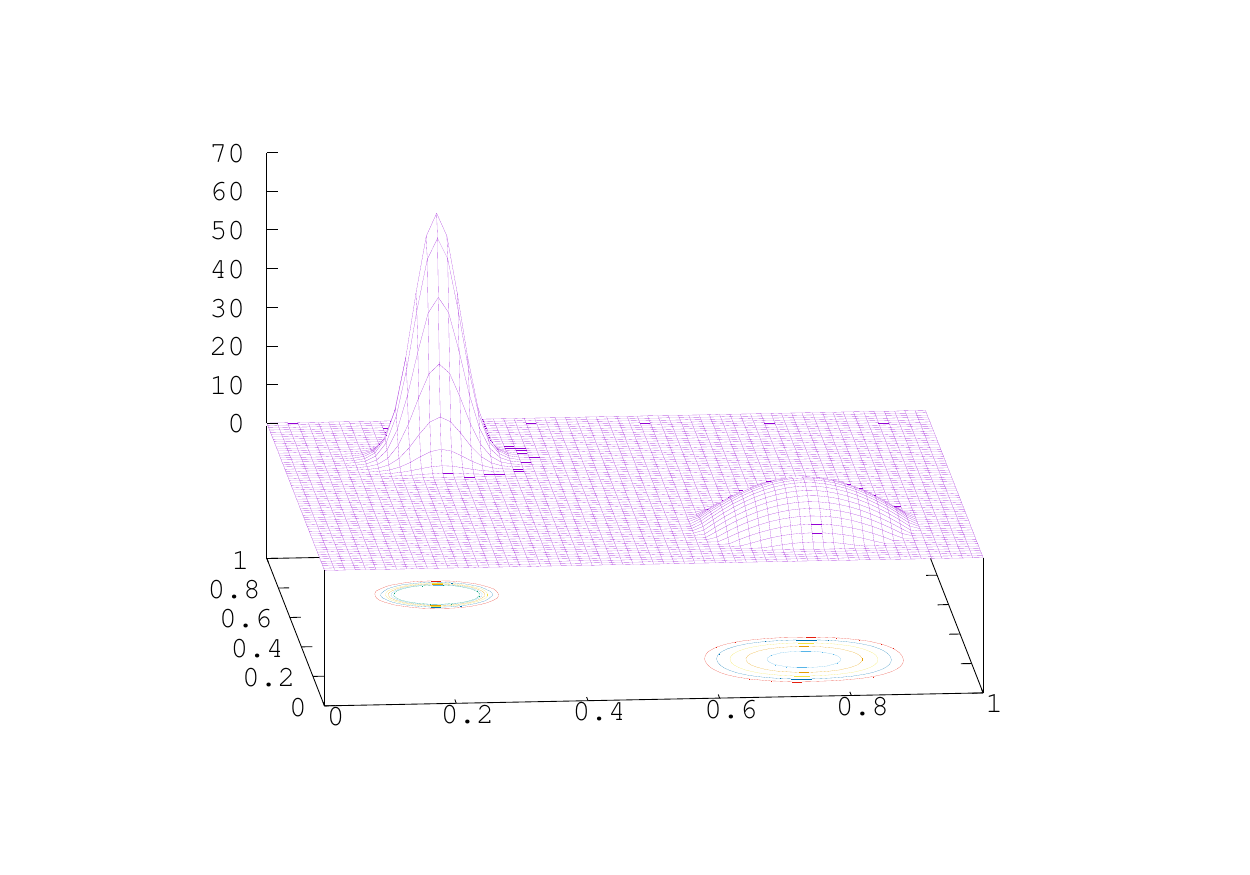}
\end{subfigure}%
\vspace{-3em}
\begin{subfigure}{.45\textwidth}
  \centering
  \includegraphics[width=\linewidth]{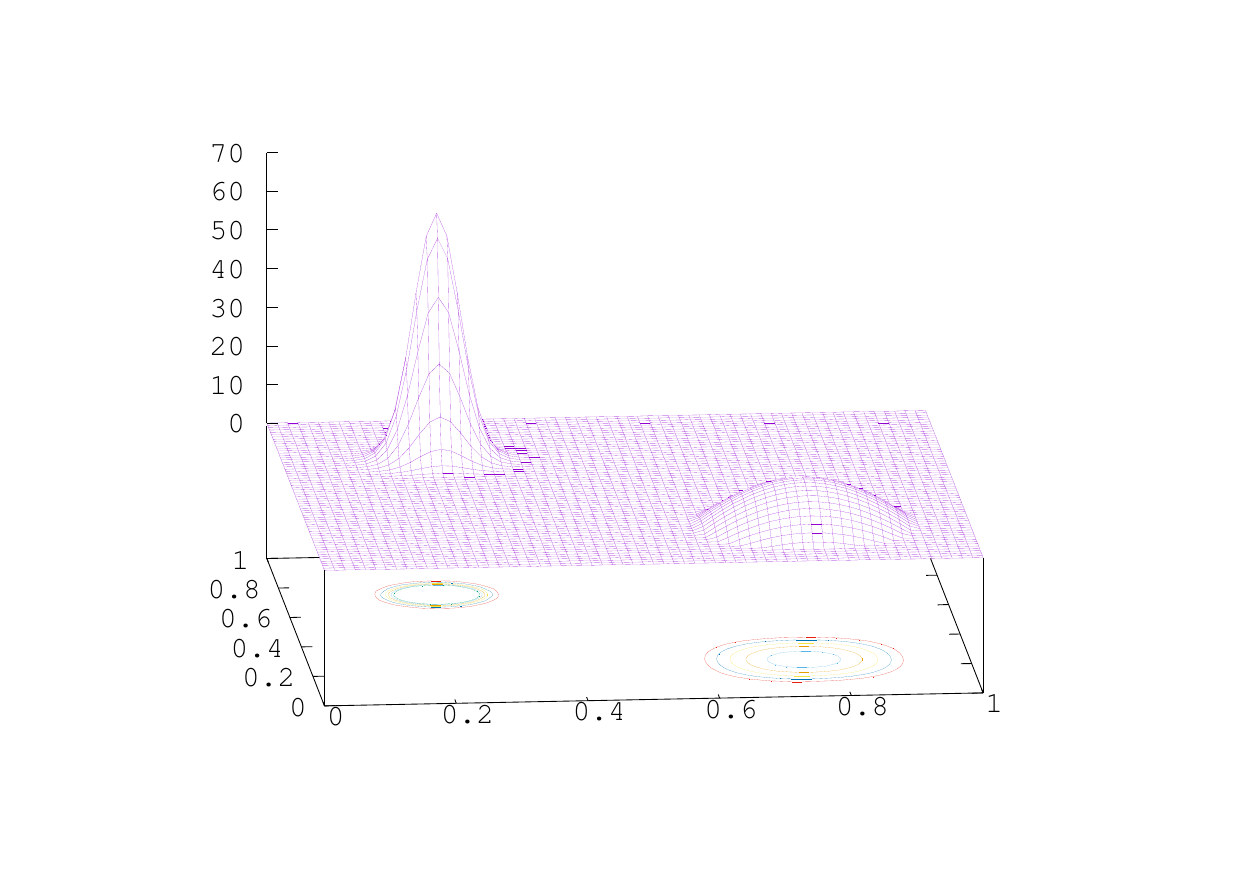}
\end{subfigure}
\vspace{-3em}
\begin{subfigure}{.45\textwidth}
  \centering
  \includegraphics[width=\linewidth]{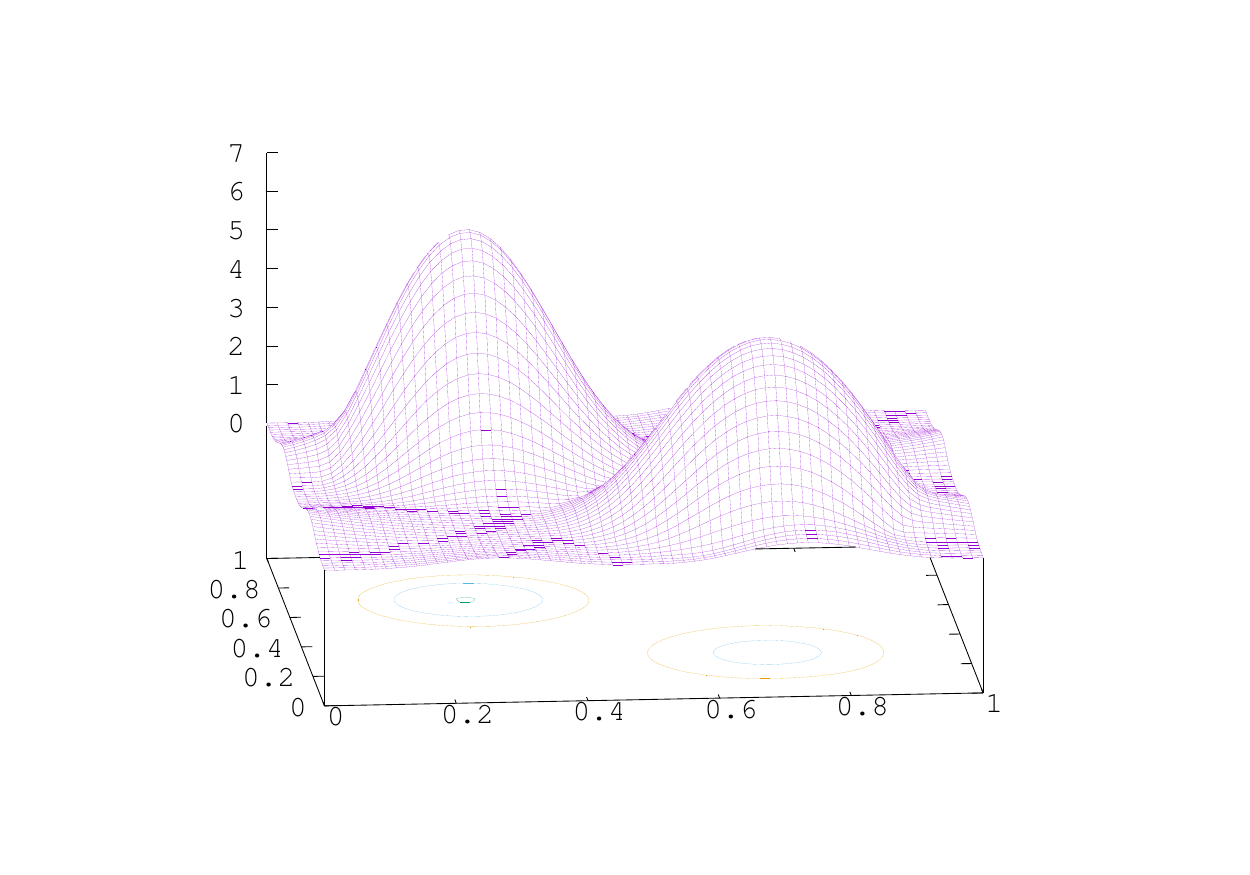}
\end{subfigure}%
\begin{subfigure}{.45\textwidth}
  \centering
  \includegraphics[width=\linewidth]{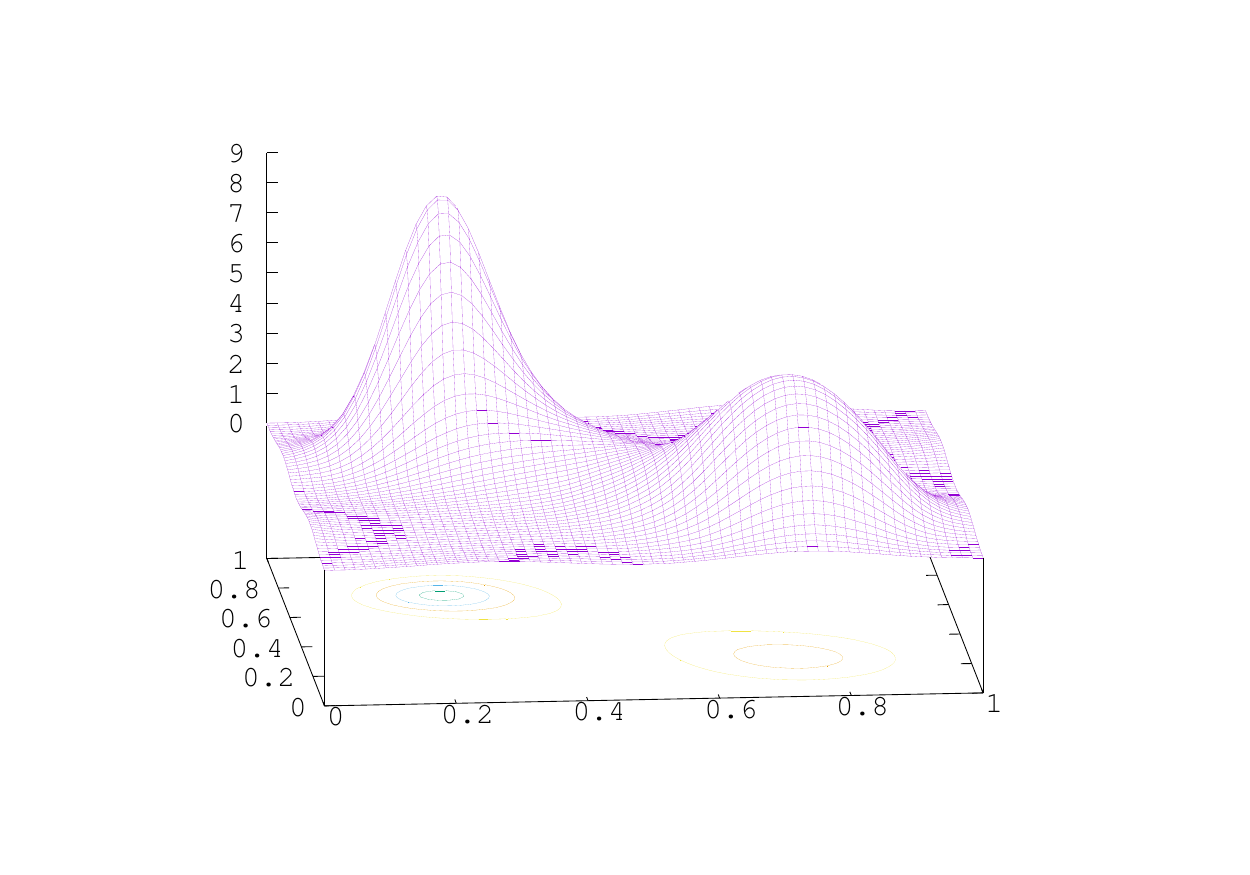}
\end{subfigure}
\vspace{-3em}
\begin{subfigure}{.45\textwidth}
  \centering
  \includegraphics[width=\linewidth]{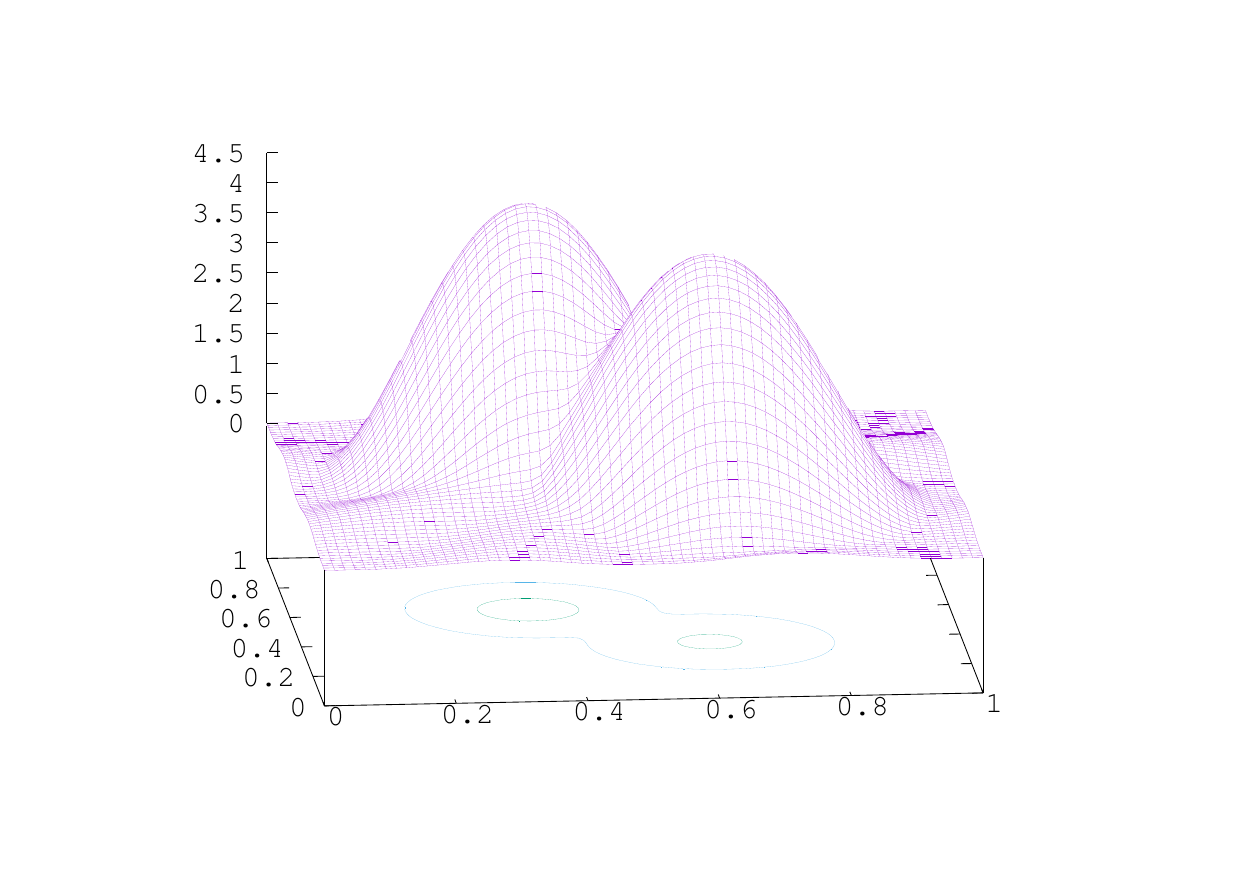}
\end{subfigure}%
\begin{subfigure}{.45\textwidth}
  \centering
  \includegraphics[width=\linewidth]{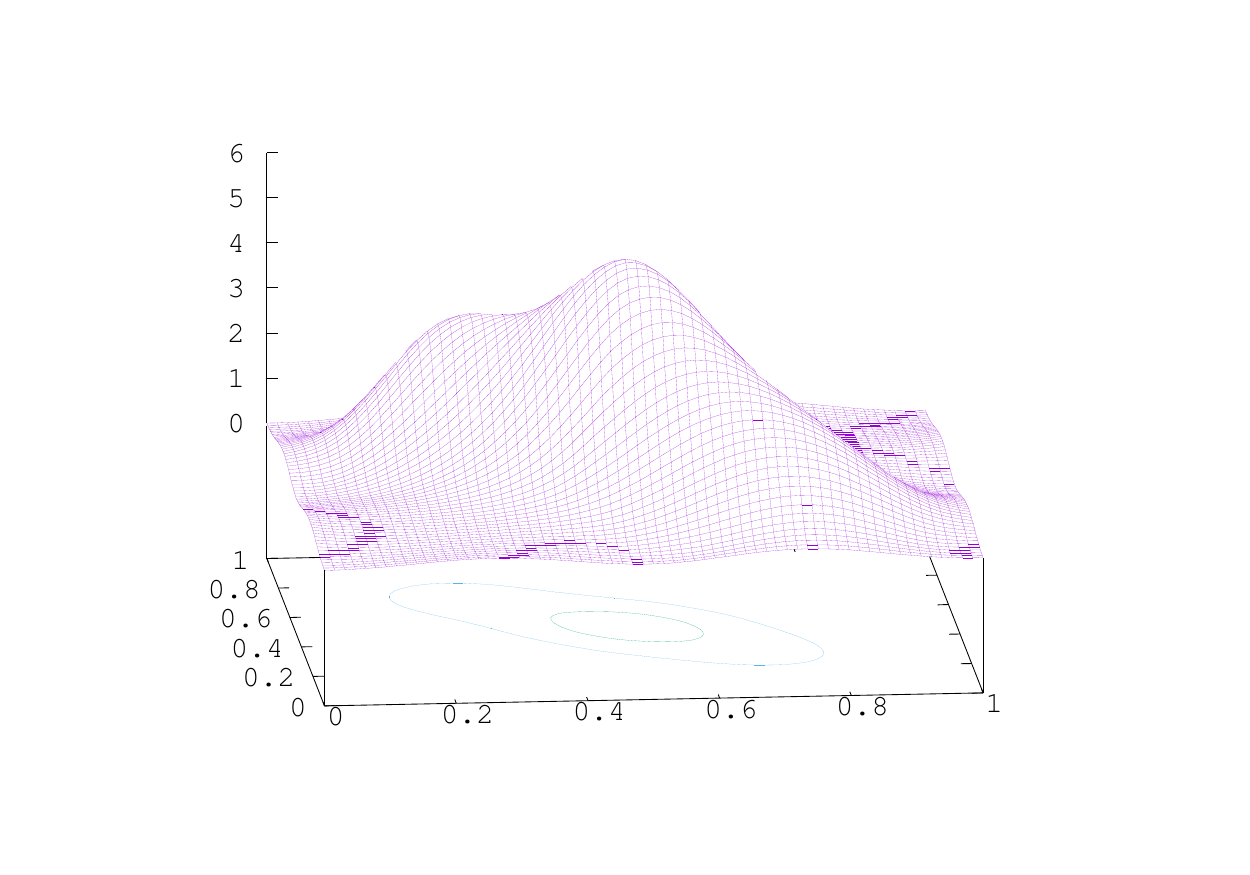}
\end{subfigure}
\vspace{-3em}
\begin{subfigure}{.45\textwidth}
  \centering
  \includegraphics[width=\linewidth]{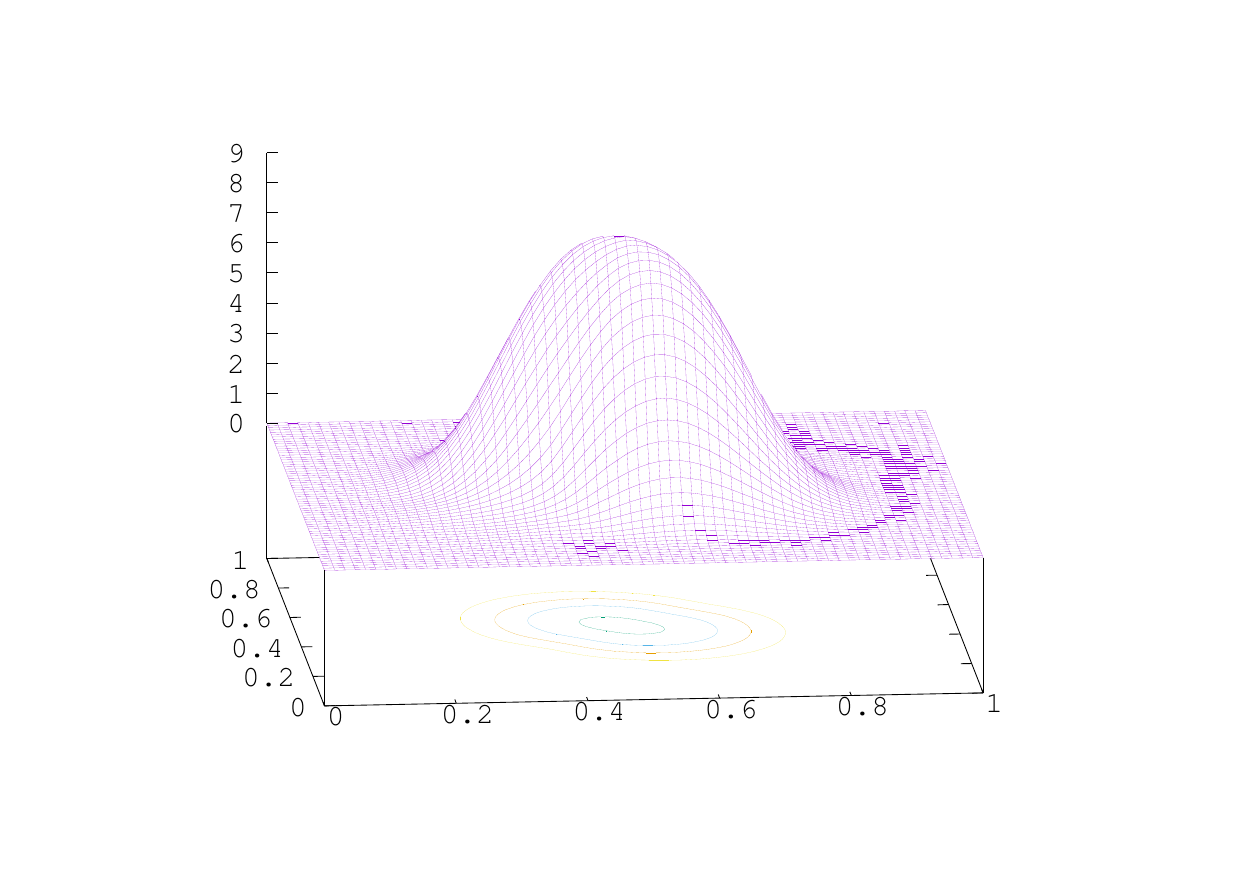}
\end{subfigure}%
\begin{subfigure}{.45\textwidth}
  \centering
  \includegraphics[width=\linewidth]{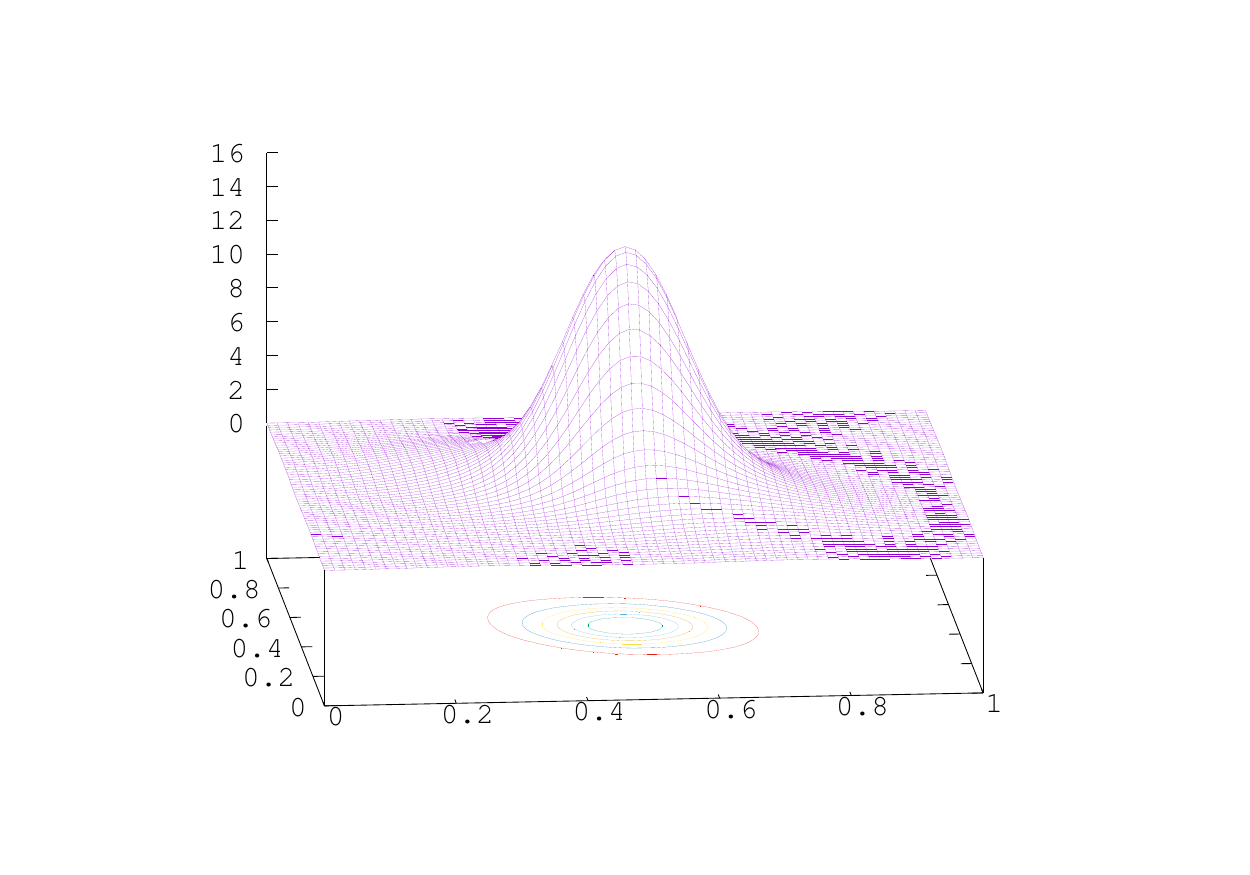}
\end{subfigure}
\begin{subfigure}{.45\textwidth}
  \centering
  \includegraphics[width=\linewidth]{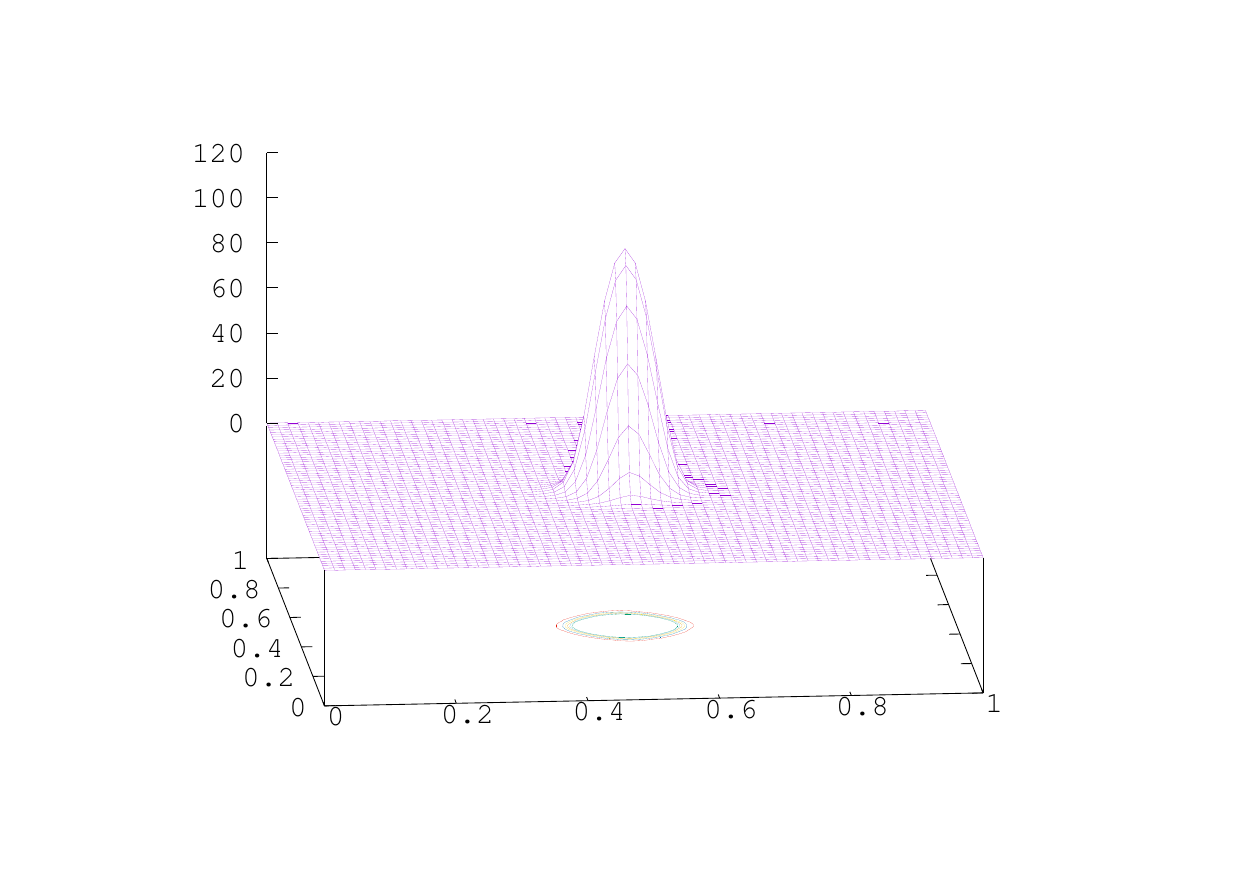}
\end{subfigure}%
\begin{subfigure}{.45\textwidth}
  \centering
  \includegraphics[width=\linewidth]{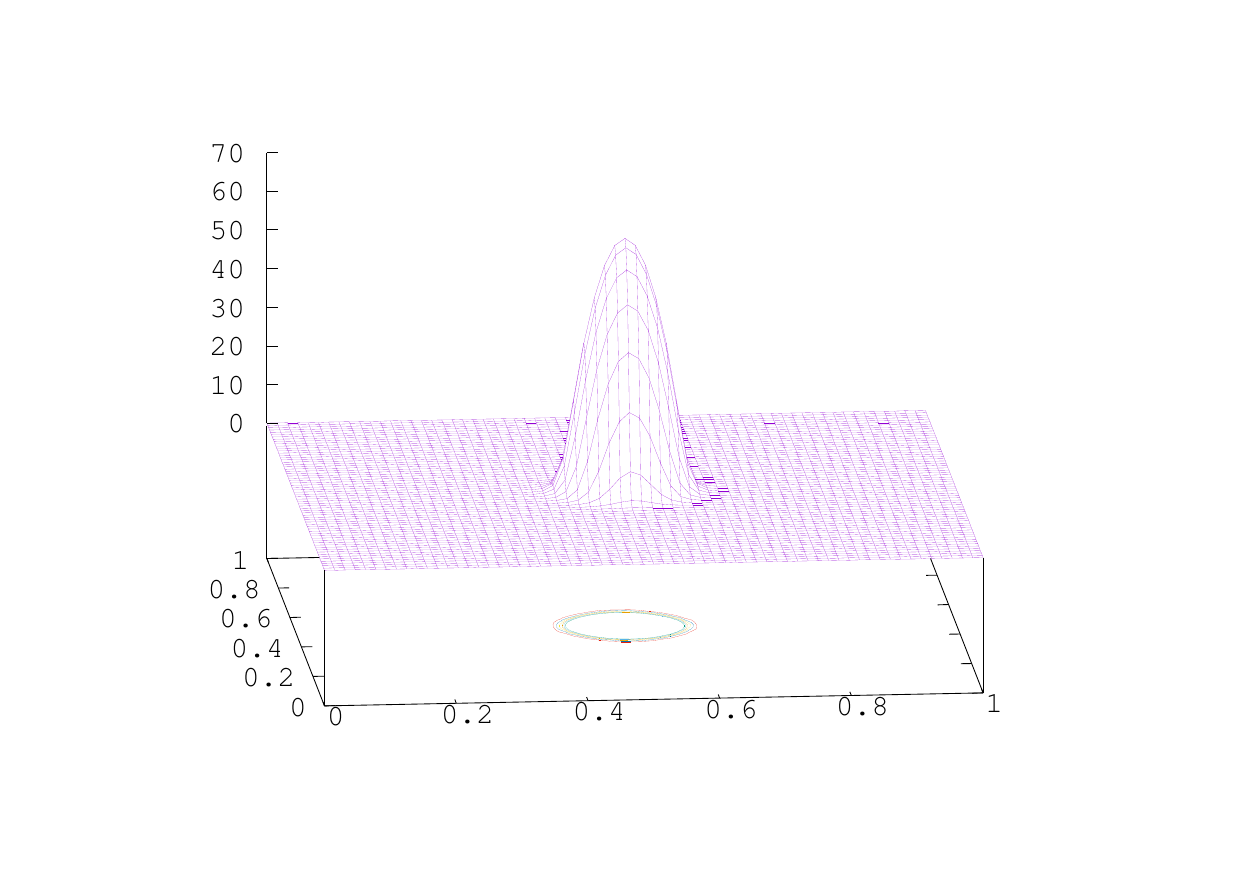}
\end{subfigure}
\caption{Influence of $\alpha$  in Test case 3. Left: $\alpha  = 0.3$, right: $\alpha =  0.7$ (and $\beta = 2, \ell(m) = 0.01m$). The rows correspond to $t=0, 1/4, 1/2, 3/4,$ and $T=1$.  The contours lines correspond to levels  $2,4, 6, 8, 10$ and $12$.}
\label{fig:diracbosse-CompareAlpha-ell001}
\end{figure}

\paragraph{Impact of $\ell$.}

We first investigate  the influence of $\ell$ in Test case 3 with  state constraints, see Figure~\ref{fig:diracbosse-ell-sc}.
We compare the evolution of the distribution with  $\ell(x,m) = \lambda m$ for $\lambda = 0.01$ and $\lambda = 0.05$. 
The larger $\lambda$ is, the less tolerant the agent are to high densities.
 We see that the distribution evolves to humps localized near the center of the domain, but that the hump is more peaky when $\ell(x,m)=0.01m$.  Note also  that especially when $\ell=0.01m$ and due to congestion effects, the part of the population initially very concentrated near $x=(0.2,0.8)$ takes more time to reach the center of the domain (and that the shape of the hump varies in time due to congestion effects).
\\ Note also that  there are regions where the density remains $0$.  These empty regions are well dealt with by the present numerical method.  \\
 The influence of $\ell$ can also be seen in Test case 2 with an obstacle, see Figure~\ref{c2c-ell-sc}. In this case, similar observations can be made.

\begin{figure}
\centering
\begin{subfigure}{.4\textwidth}
  \centering
  \includegraphics[width=\linewidth]{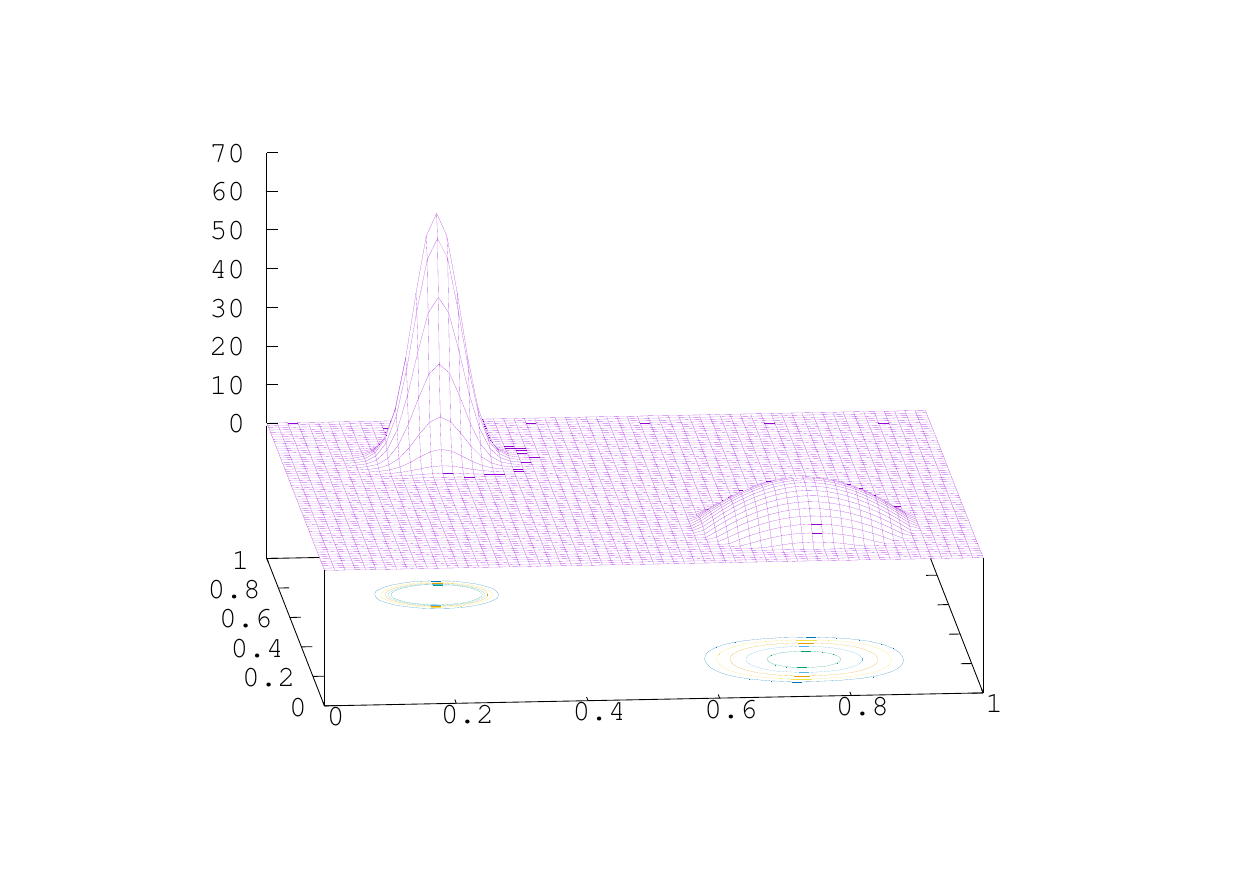}
\end{subfigure}%
\vspace{-3em}%
\begin{subfigure}{.4\textwidth}
  \centering
  \includegraphics[width=\linewidth]{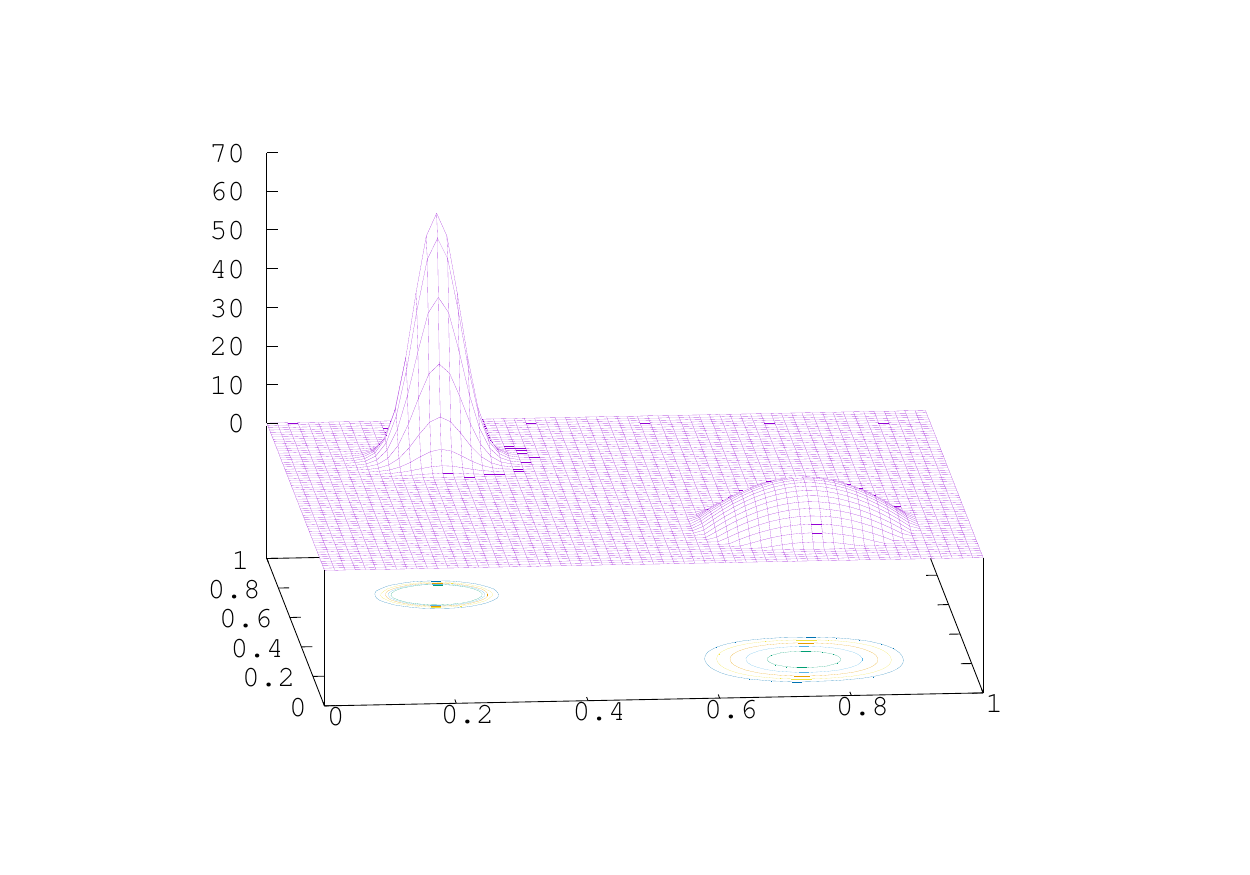}
\end{subfigure}
\begin{subfigure}{.4\textwidth}
  \centering
  \includegraphics[width=\linewidth]{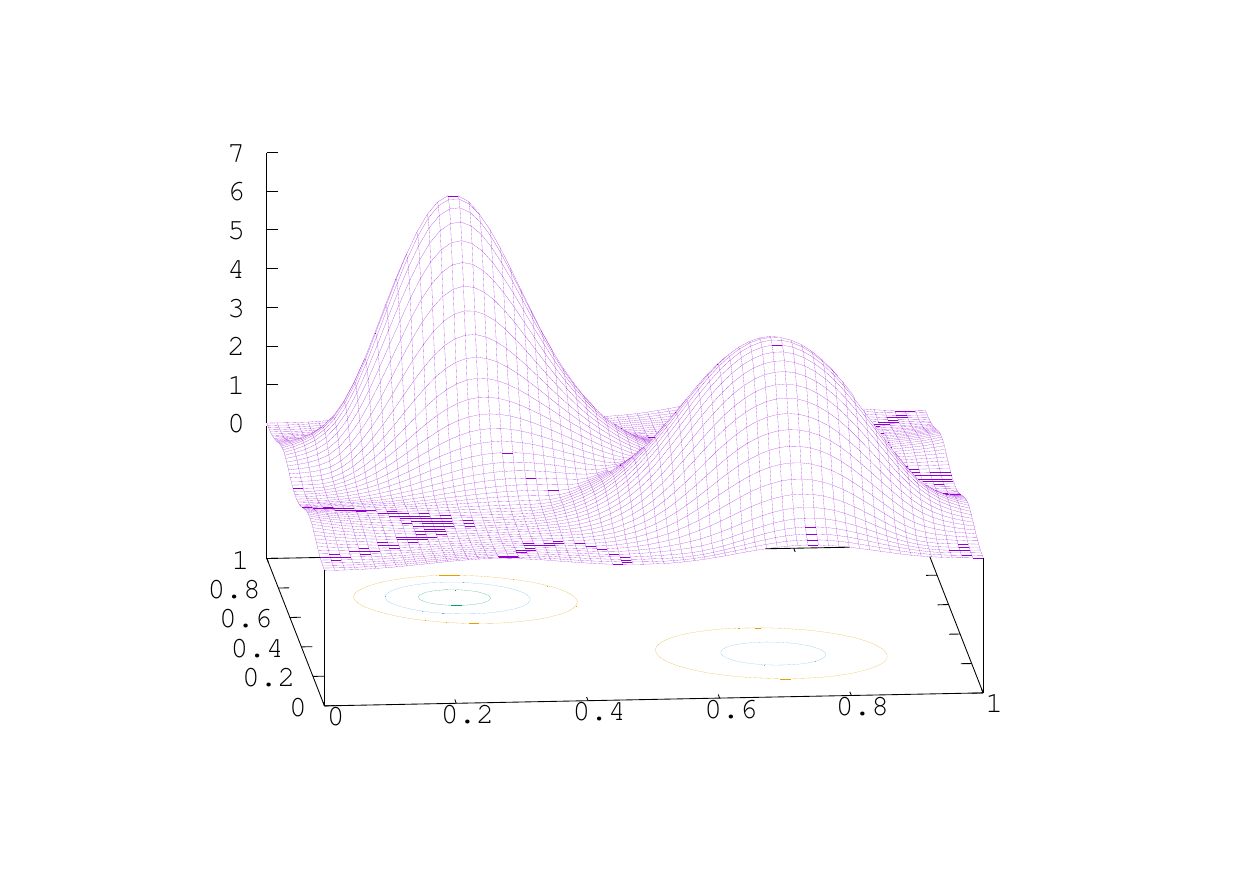}
\end{subfigure}%
\begin{subfigure}{.4\textwidth}
  \centering
  \includegraphics[width=\linewidth]{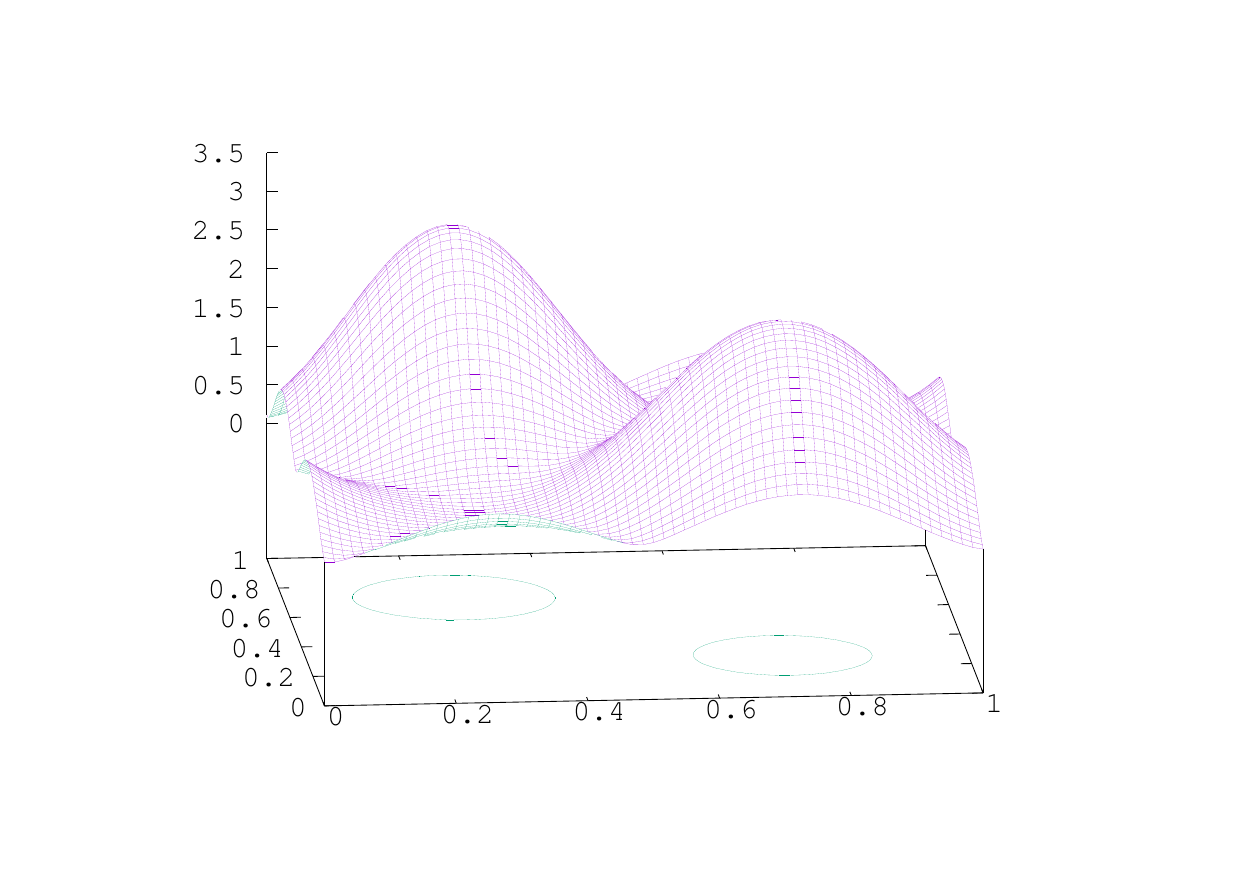}
\end{subfigure}
\begin{subfigure}{.4\textwidth}
  \centering
  \includegraphics[width=\linewidth]{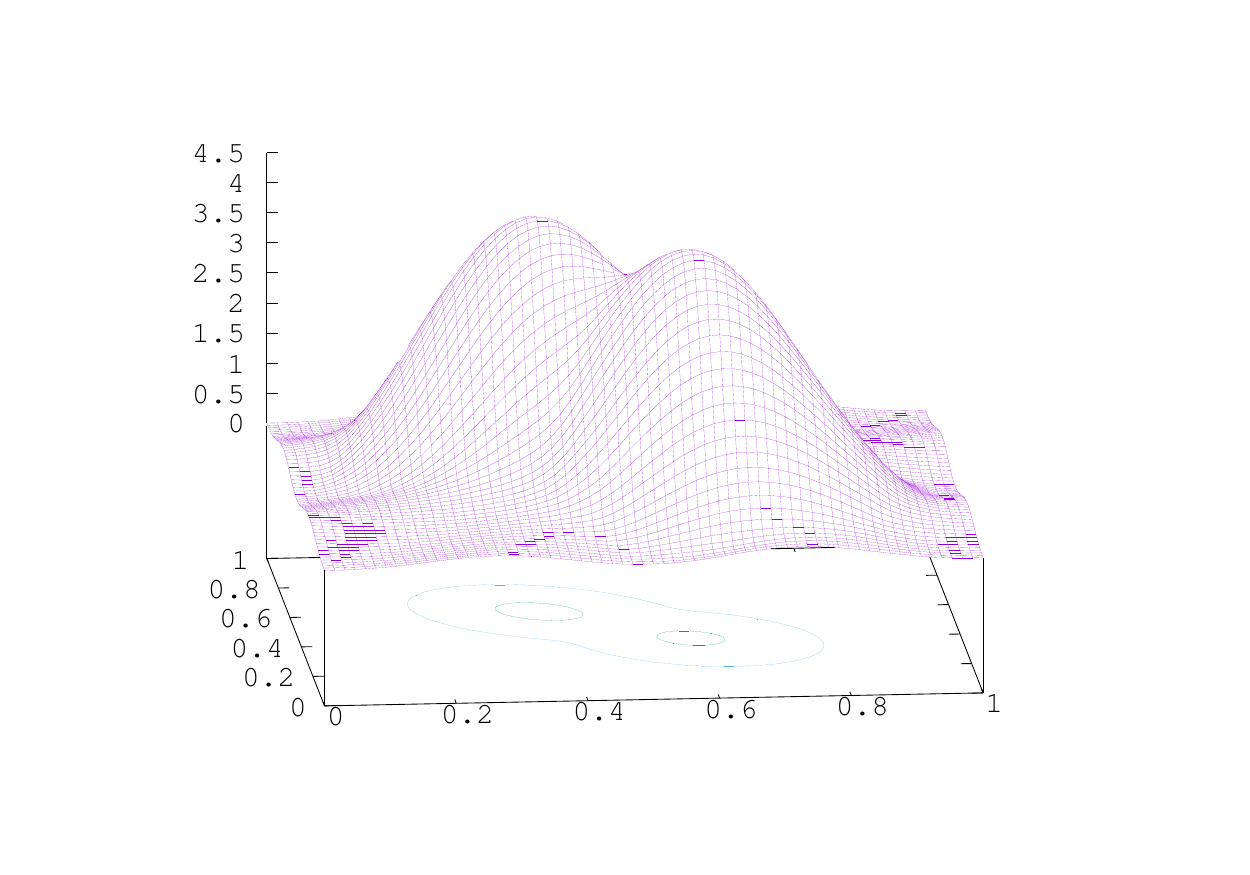}
\end{subfigure}%
\begin{subfigure}{.4\textwidth}
  \centering
  \includegraphics[width=\linewidth]{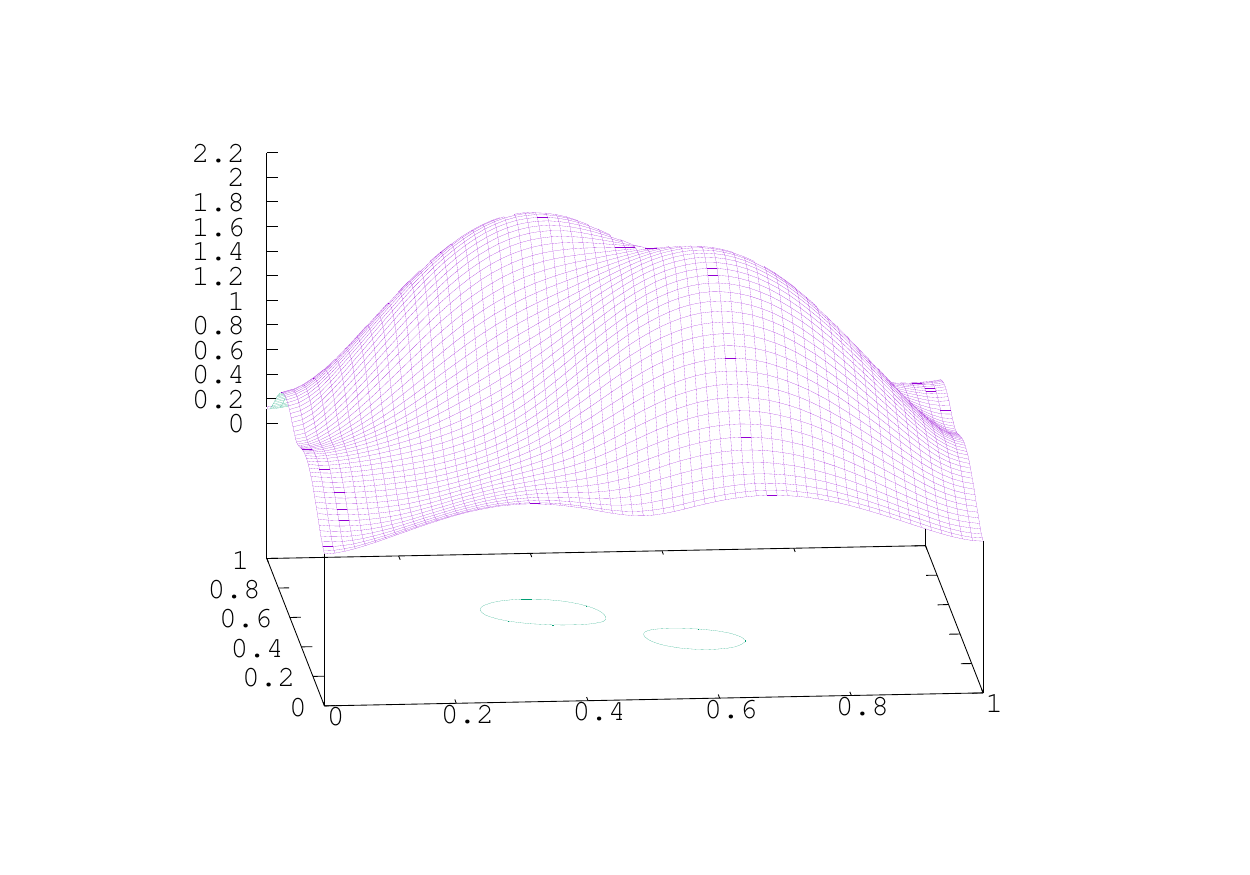}
\end{subfigure}
\begin{subfigure}{.4\textwidth}
  \centering
  \includegraphics[width=\linewidth]{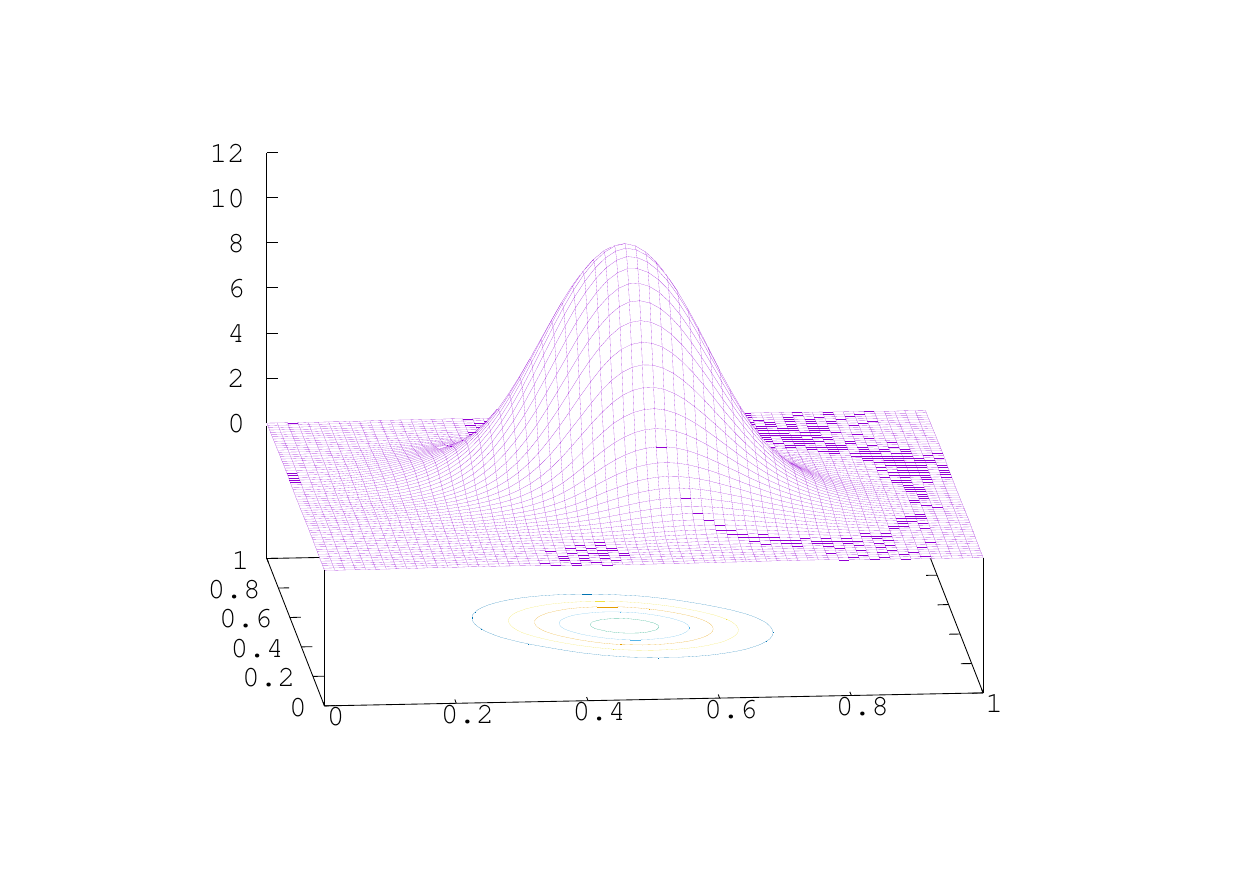}
\end{subfigure}%
\begin{subfigure}{.4\textwidth}
  \centering
  \includegraphics[width=\linewidth]{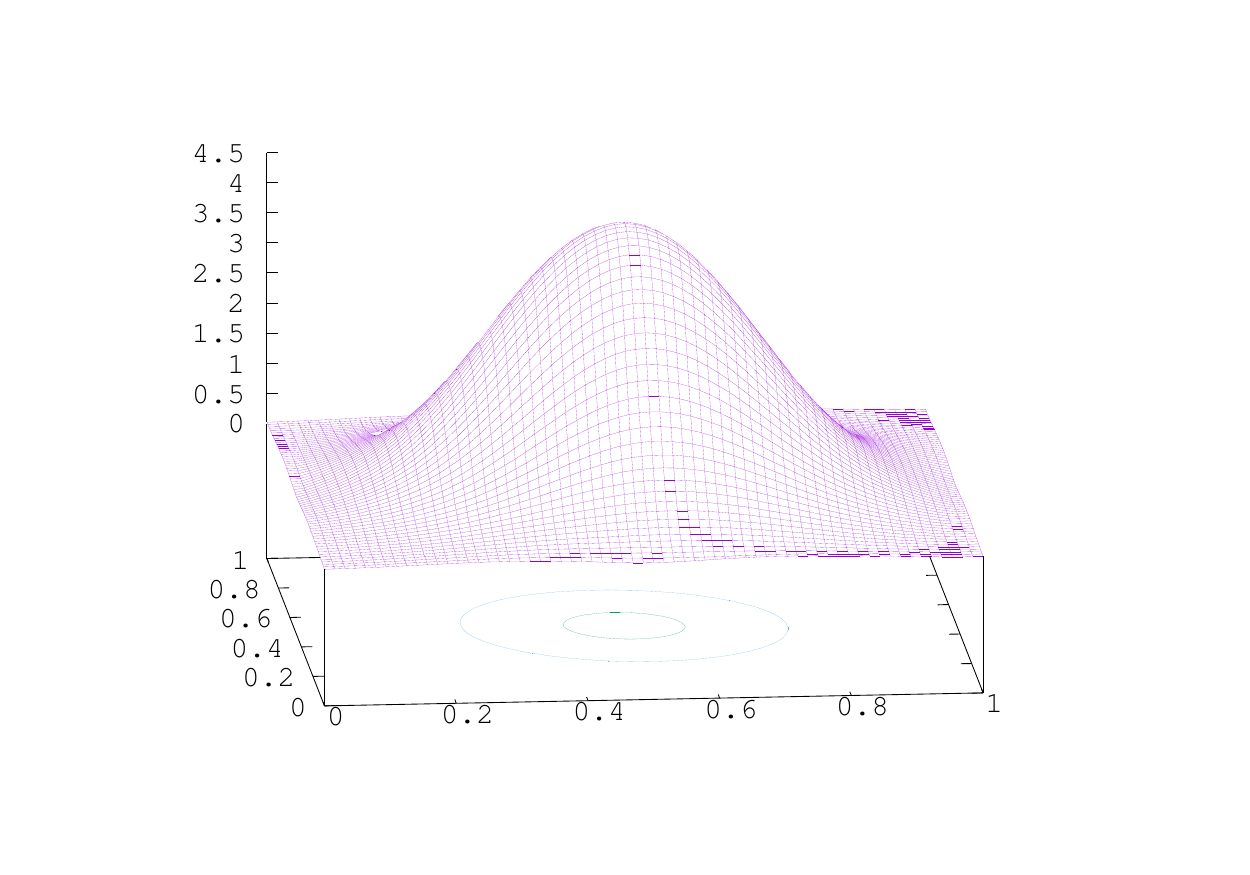}
\end{subfigure}
\begin{subfigure}{.4\textwidth}
  \centering
  \includegraphics[width=\linewidth]{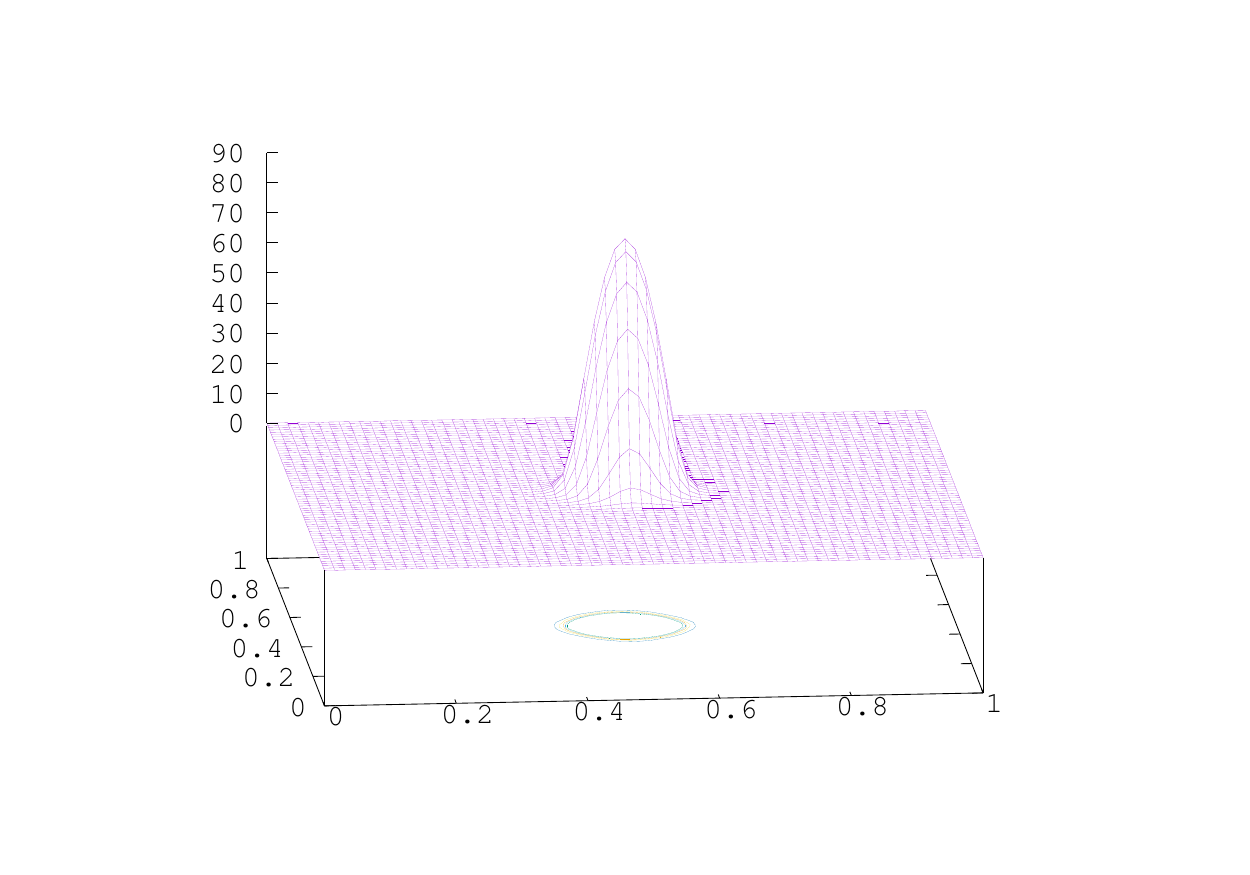}
\end{subfigure}%
\begin{subfigure}{.4\textwidth}
  \centering
  \includegraphics[width=\linewidth]{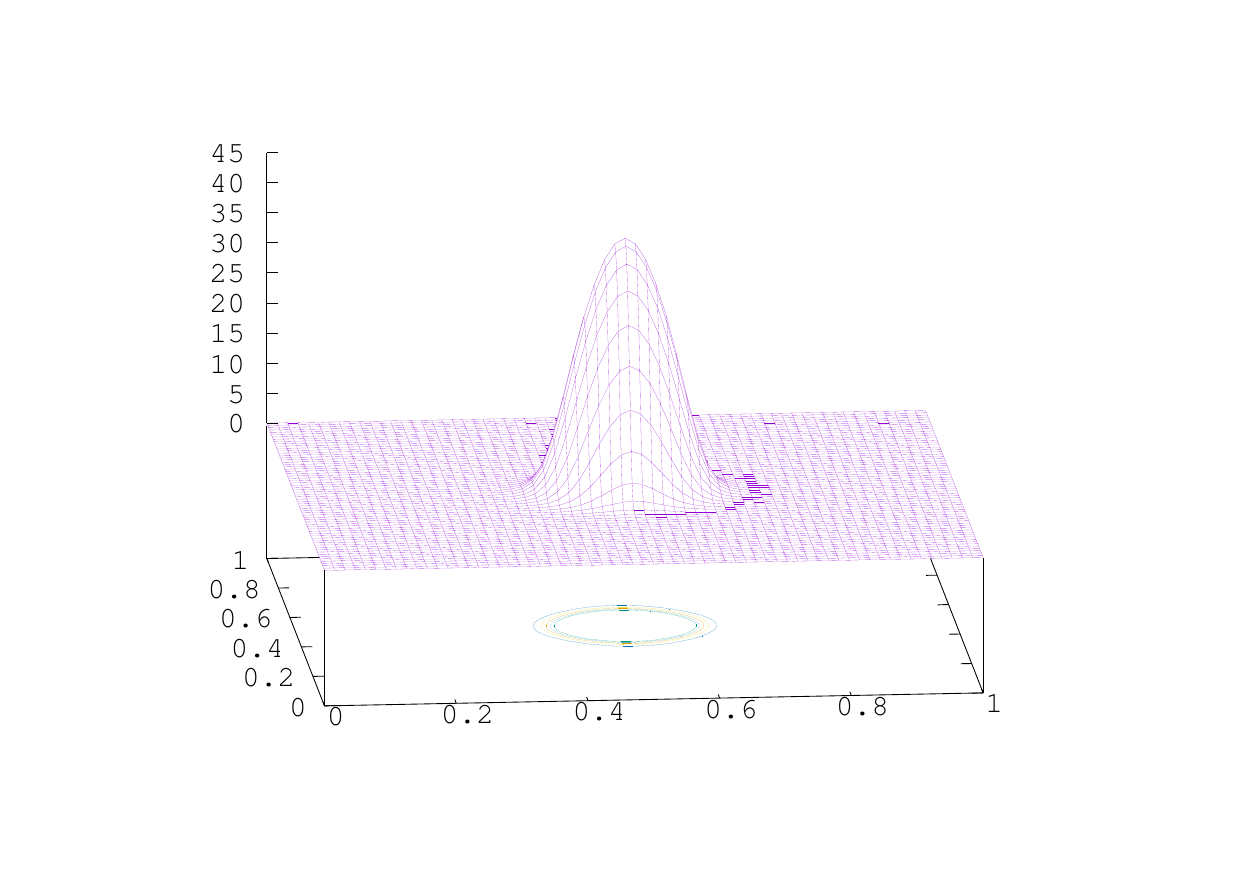}
\end{subfigure}
\caption{Influence of $\ell$ (Test case 3, with periodic boundary conditions and $\alpha = 0.5, \beta = 2$). Left: $\ell(x,m) =  0.01m$, right: $\ell(x,m) =  0.05m$. The rows correspond to $t=0, 1/4, 1/2, 3/4,$ and $T=1$. The contours lines correspond to levels $2,4,6,8,$ and $10$.}
\label{fig:diracbosse-ell-sc}
\end{figure}

\begin{figure}
\centering
\begin{subfigure}{.4\textwidth}
  \centering
  \includegraphics[width=\linewidth]{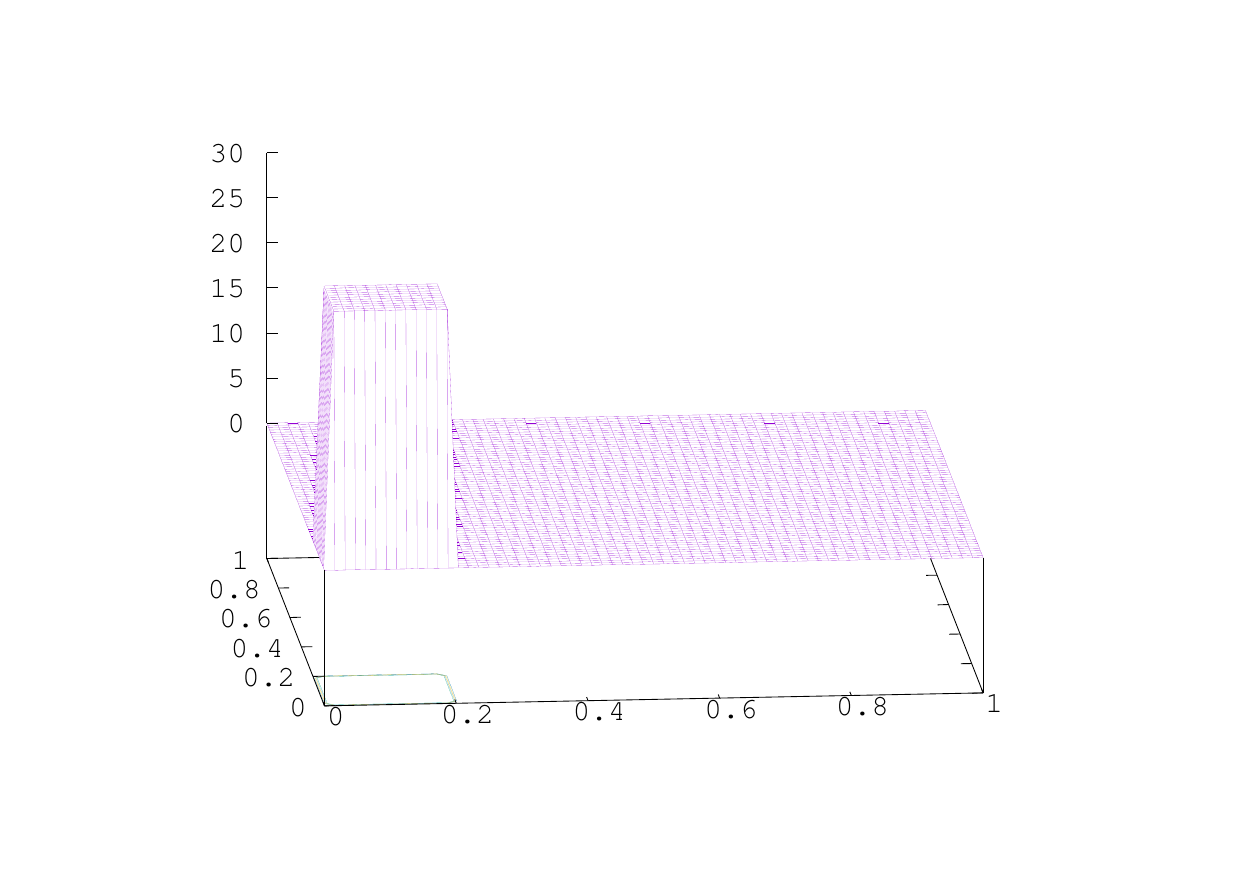}
\end{subfigure}
\vspace{-3em}%
\begin{subfigure}{.4\textwidth}
  \centering
  \includegraphics[width=\linewidth]{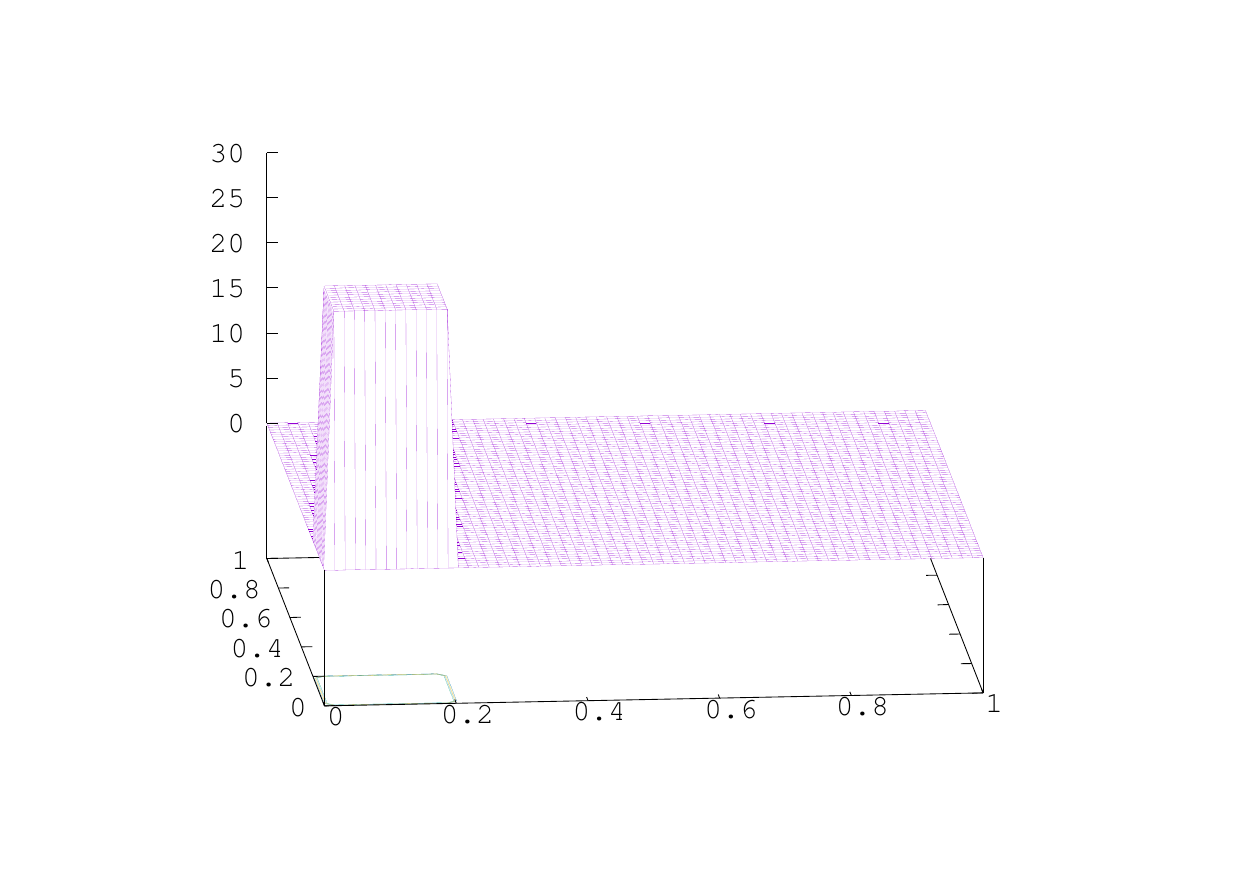}
\end{subfigure}
\vspace{-3em}
\begin{subfigure}{.4\textwidth}
  \centering
  \includegraphics[width=\linewidth]{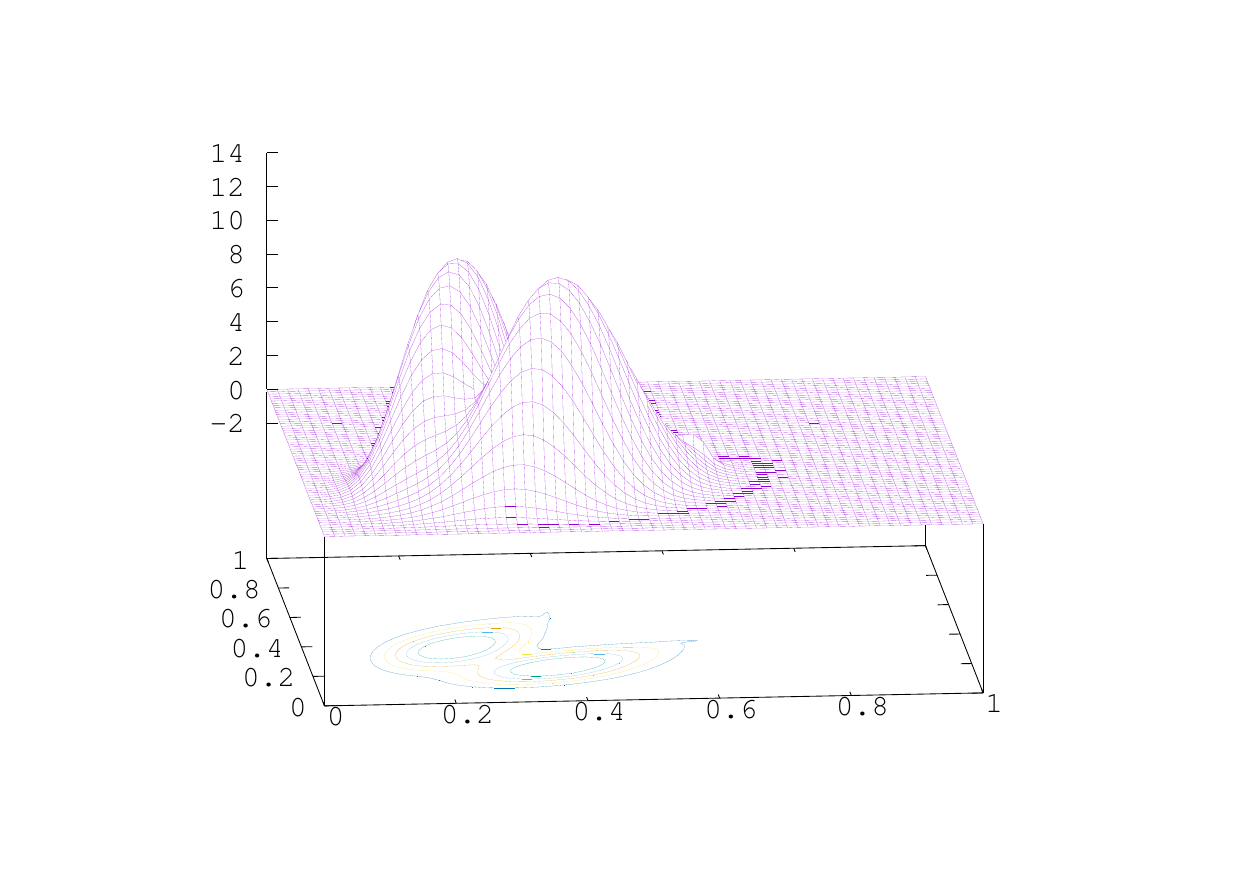}
\end{subfigure}%
\begin{subfigure}{.4\textwidth}
  \centering
  \includegraphics[width=\linewidth]{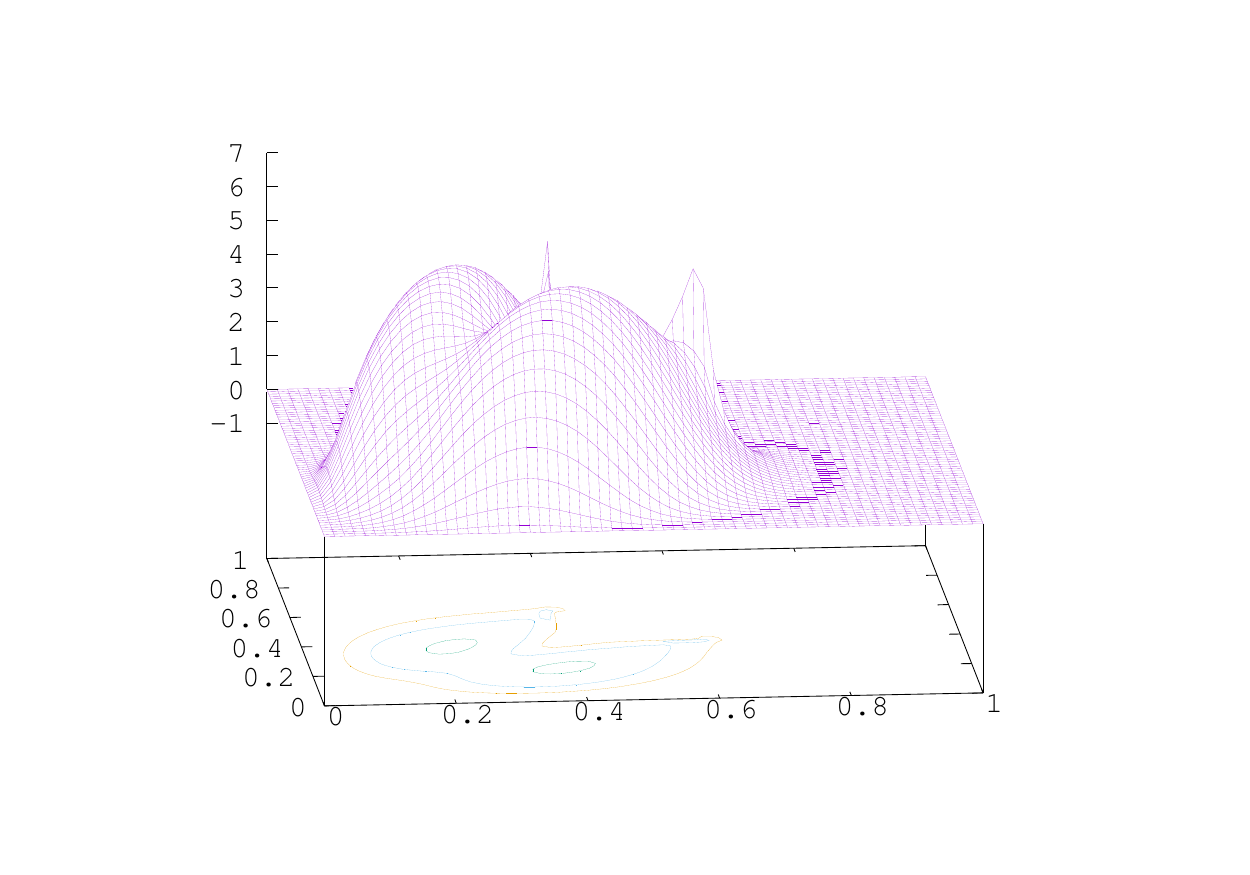}
\end{subfigure}
\vspace{-3em}
\begin{subfigure}{.4\textwidth}
  \centering
  \includegraphics[width=\linewidth]{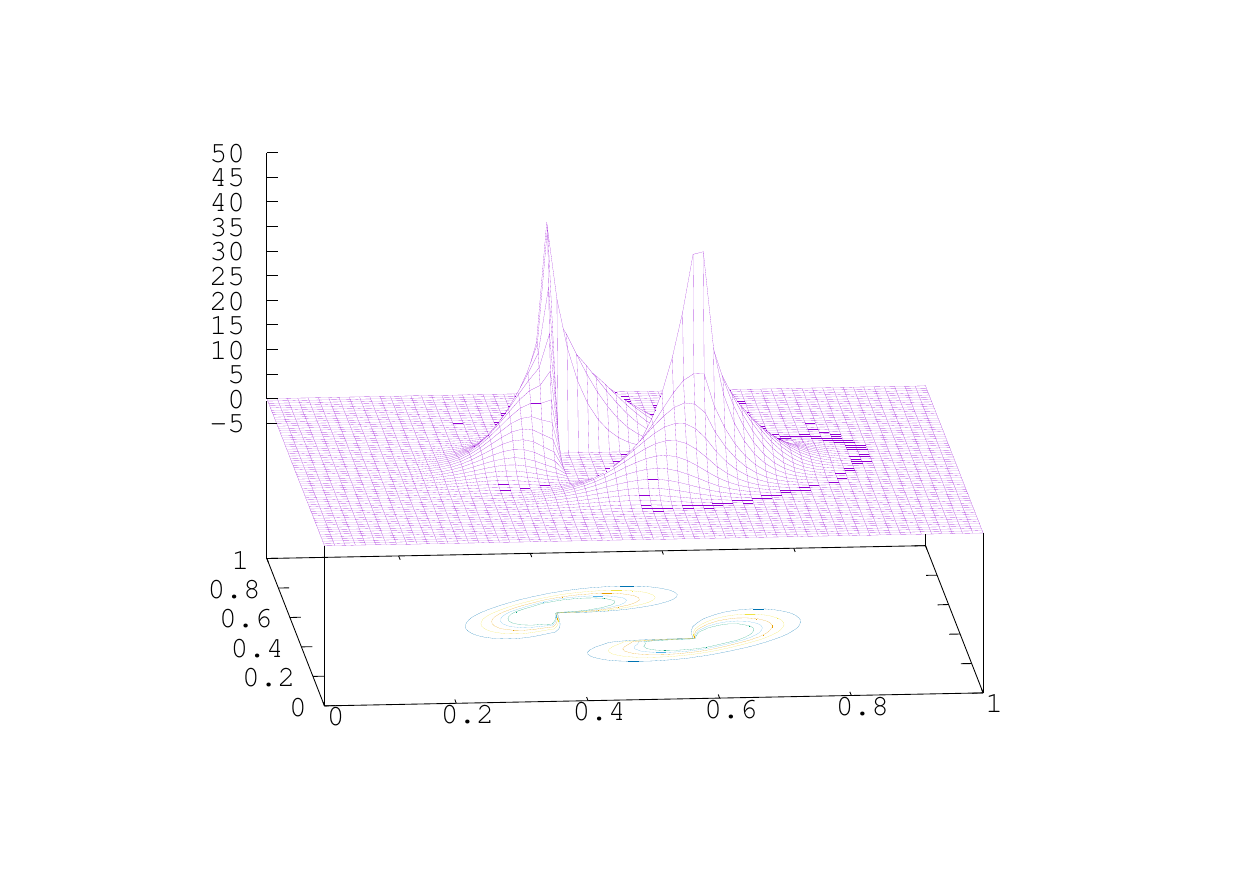}
\end{subfigure}%
\begin{subfigure}{.4\textwidth}
  \centering
  \includegraphics[width=\linewidth]{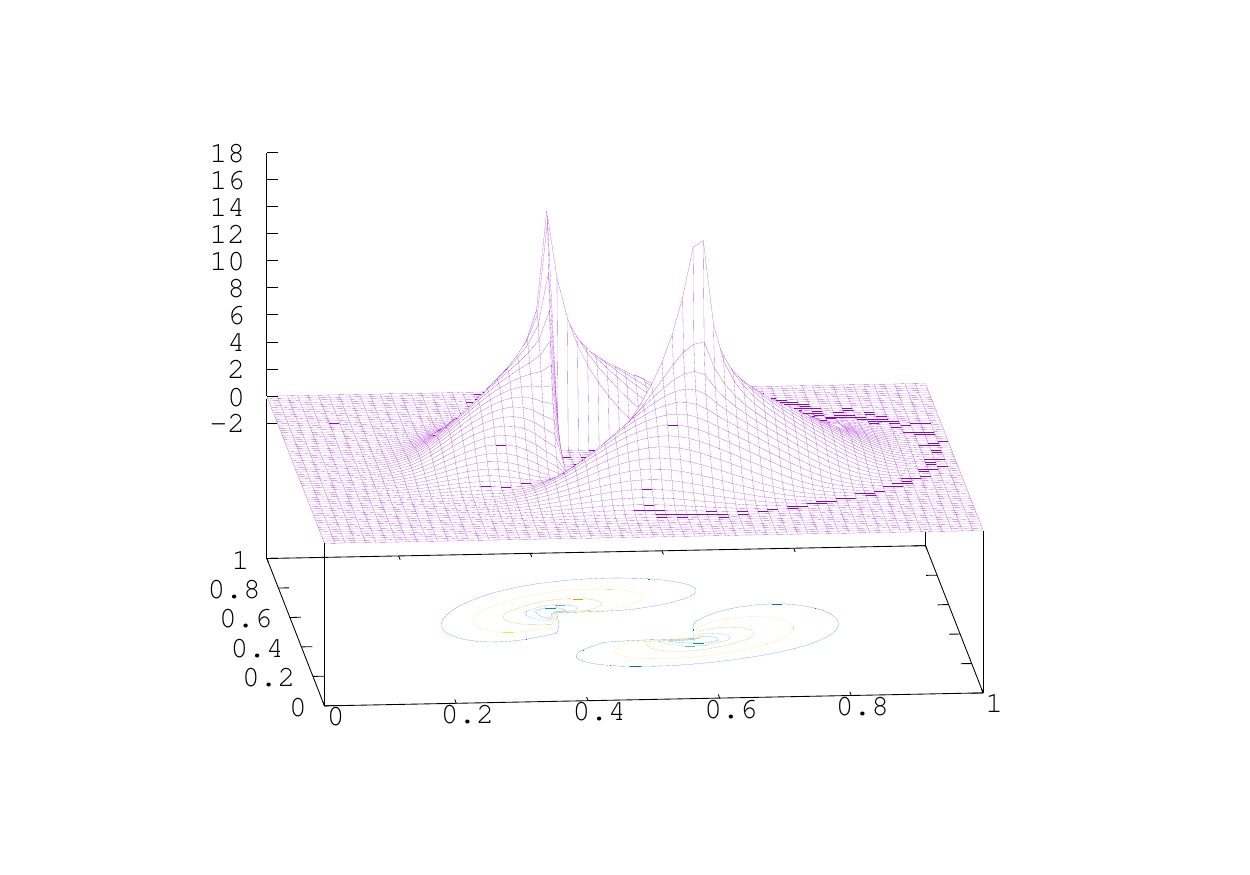}
\end{subfigure}
\vspace{-3em}
\begin{subfigure}{.4\textwidth}
  \centering
  \includegraphics[width=\linewidth]{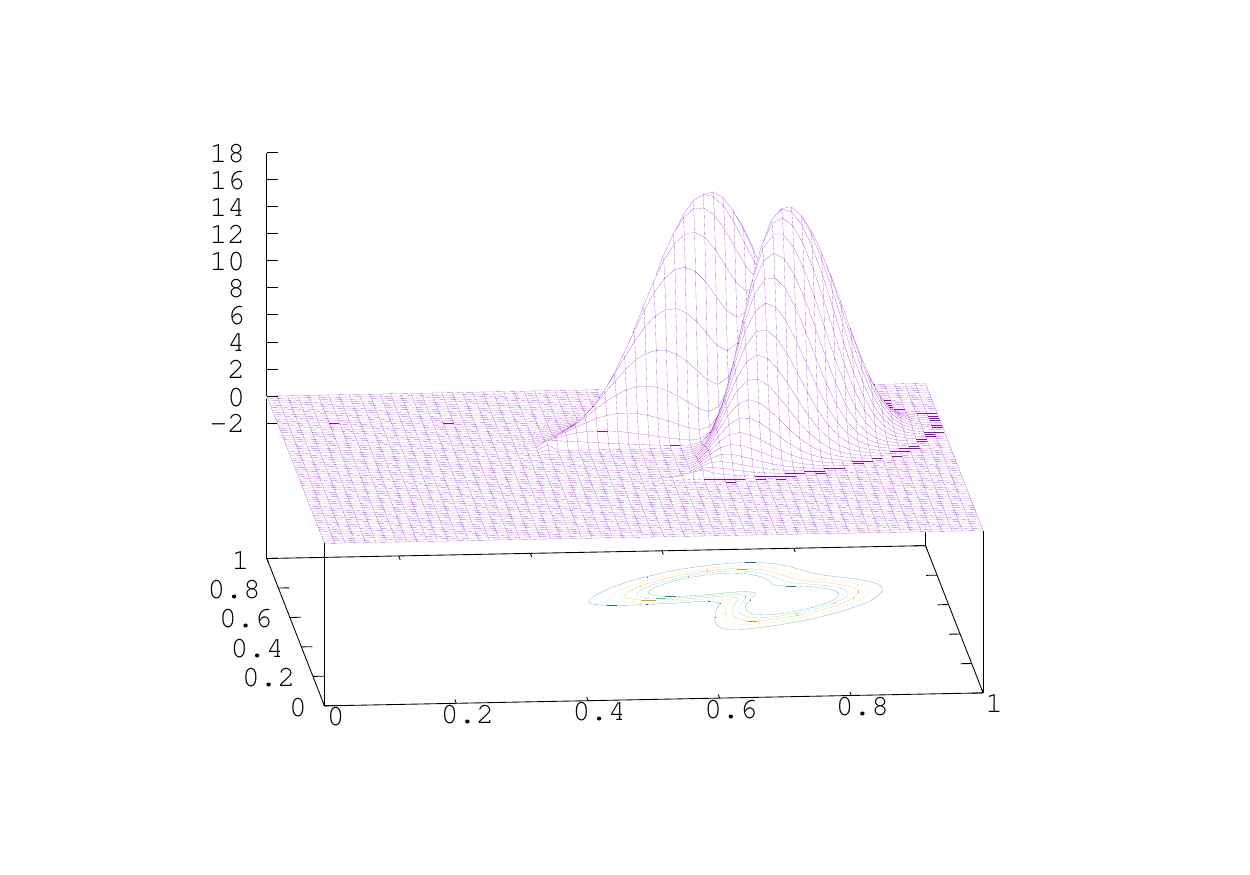}
\end{subfigure}%
\begin{subfigure}{.4\textwidth}
  \centering
  \includegraphics[width=\linewidth]{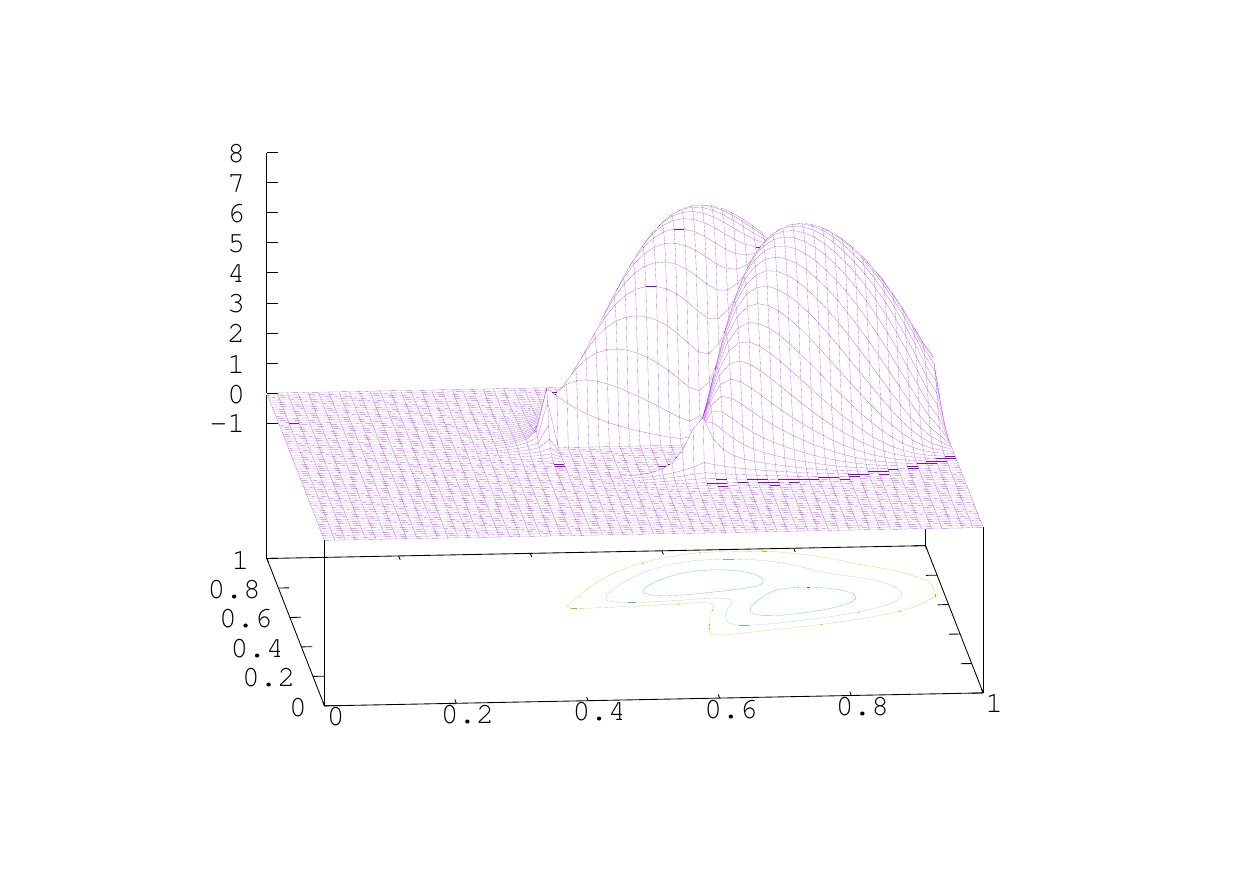}
\end{subfigure}
\begin{subfigure}{.4\textwidth}
  \centering
  \includegraphics[width=\linewidth]{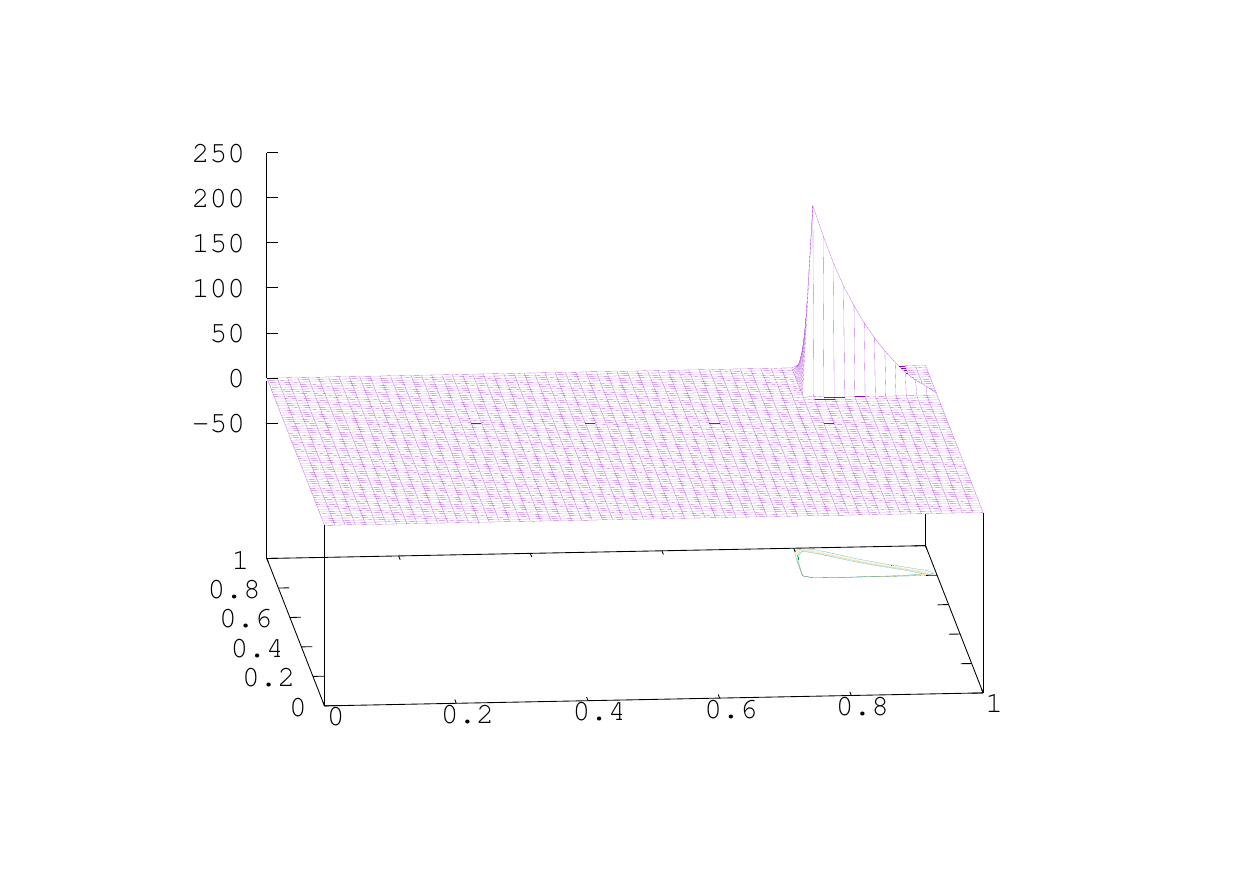}
\end{subfigure}%
\begin{subfigure}{.4\textwidth}
  \centering
  \includegraphics[width=\linewidth]{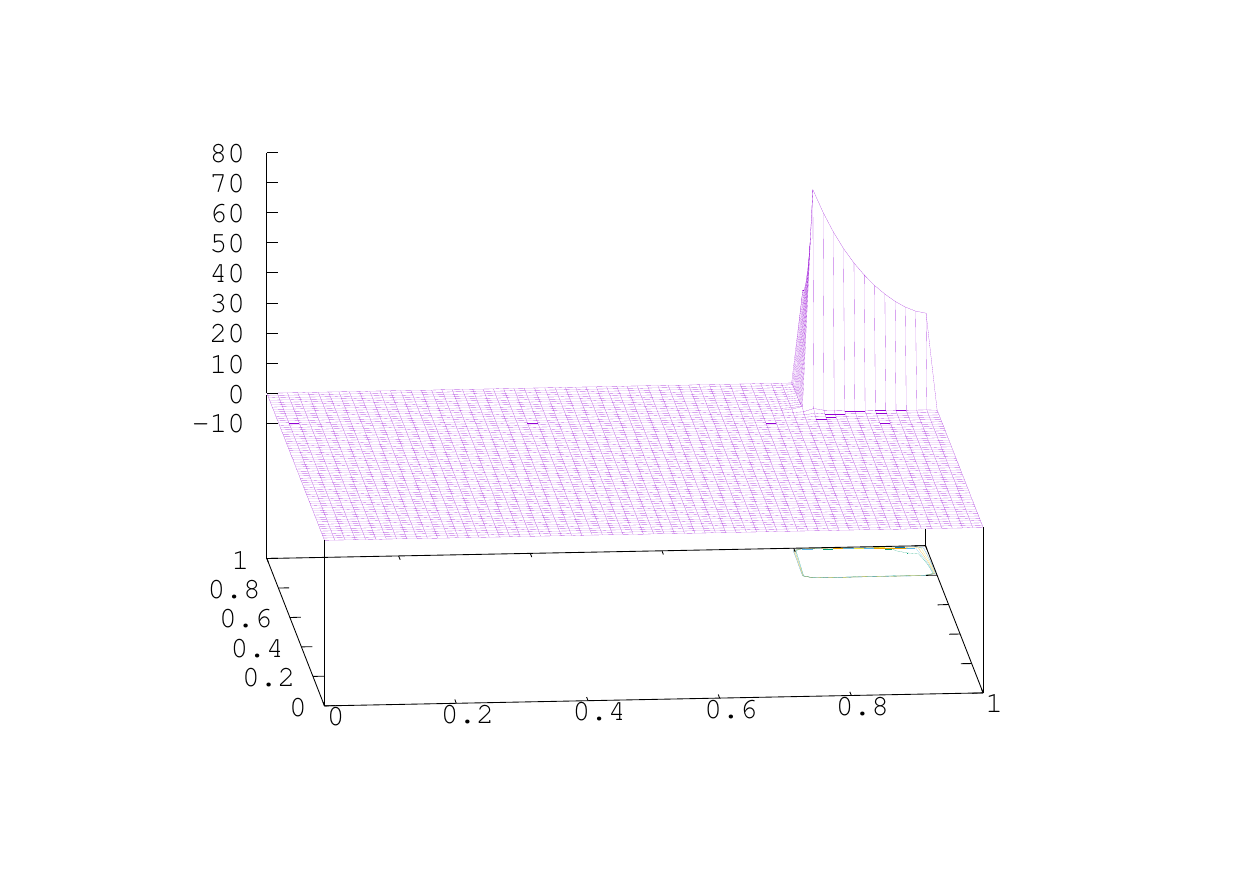}
\end{subfigure}
\caption{Influence of $\ell$ (Test case 2 with state constrained to $[0,1]^2 \setminus (0.4,0.6)^2$, and $\alpha = 0.01, \beta = 2$). Left: $\ell(x,m) =  0.001$, right: $\ell(x,m) =  0.01m$. The rows correspond to $t=0, 1/4, 1/2, 3/4,$ and $T=1$.The contours lines correspond to levels  $2,4,6,8,$ and $10$.}
\label{c2c-ell-sc}
\end{figure}

{\small
\bibliographystyle{amsplain}
\bibliography{mftc-auglag-bib}

\def\cprime{$'$} \def\cprime{$'$}
\providecommand{\bysame}{\leavevmode\hbox to3em{\hrulefill}\thinspace}
\providecommand{\MR}{\relax\ifhmode\unskip\space\fi MR }
\providecommand{\MRhref}[2]{%
  \href{http://www.ams.org/mathscinet-getitem?mr=#1}{#2}
}
\providecommand{\href}[2]{#2}
\begin{thebibliography}{10}

\bibitem{MR3135339}
Y.~Achdou, \emph{Finite difference methods for mean field games},
  Hamilton-{J}acobi equations: approximations, numerical analysis and
  applications (P.~Loreti and N.~A. Tchou, eds.), Lecture Notes in Math., vol.
  2074, Springer, Heidelberg, 2013, pp.~1--47.

\bibitem{AchdouCamilliCapuzzoDolcetta2012Planning}
Y.~Achdou, F.~Camilli, and I.~Capuzzo~Dolcetta, \emph{Mean field games:
  numerical methods for the planning problem}, SIAM Journal on Control and
  Optimization \textbf{50} (2012), no.~1, 77--109.

\bibitem{MR3097034}
Y.~Achdou, F.~Camilli, and I.~Capuzzo-Dolcetta, \emph{Mean field games:
  convergence of a finite difference method}, SIAM J. Numer. Anal. \textbf{51}
  (2013), no.~5, 2585--2612.

\bibitem{MR2679575}
Y.~Achdou and I.~Capuzzo-Dolcetta, \emph{Mean field games: numerical methods},
  SIAM J. Numer. Anal. \textbf{48} (2010), no.~3, 1136--1162.

\bibitem{YAJML}
Y.~Achdou and J.-M. Lasry, \emph{Mean field games for modeling crowd motion},
  preprint.

\bibitem{MR3392611}
Y.~Achdou and M.~Lauri{\`e}re, \emph{On the system of partial differential
  equations arising in mean field type control}, Discrete Contin. Dyn. Syst.
  \textbf{35} (2015), no.~9, 3879--3900. \MR{3392611}

\bibitem{MR3498932}
\bysame, \emph{Mean {F}ield {T}ype {C}ontrol with {C}ongestion}, Appl. Math.
  Optim. \textbf{73} (2016), no.~3, 393--418. \MR{3498932}

\bibitem{YAAP16}
Y.~Achdou and A.~Porretta, \emph{Mean field games with congestion}, in
  preparation.

\bibitem{MR3452251}
\bysame, \emph{Convergence of a {F}inite {D}ifference {S}cheme to {W}eak
  {S}olutions of the {S}ystem of {P}artial {D}ifferential {E}quations {A}rising
  in {M}ean {F}ield {G}ames}, SIAM J. Numer. Anal. \textbf{54} (2016), no.~1,
  161--186. \MR{3452251}

\bibitem{andreev2016preconditioning}
R.~Andreev, \emph{Preconditioning the augmented lagrangian method for
  instationary mean field games with diffusion},  (2016).

\bibitem{BenamouBrenier2000monge}
J.-D. Benamou and Y.~Brenier, \emph{A computational fluid mechanics solution to
  the monge-kantorovich mass transfer problem}, Numerische Mathematik
  \textbf{84}, no.~3, 375--393.

\bibitem{BenamouCarlier2015ALG2}
J.-D. Benamou and G.~Carlier, \emph{Augmented lagrangian methods for transport
  optimization, mean field games and degenerate elliptic equations}, Journal of
  Optimization Theory and Applications \textbf{167} (2015), no.~1, 1--26.

\bibitem{MR3037035}
A.~Bensoussan and J.~Frehse, \emph{Control and {N}ash games with mean field
  effect}, Chin. Ann. Math. Ser. B \textbf{34} (2013), no.~2, 161--192.

\bibitem{MR3134900}
A.~Bensoussan, J.~Frehse, and P.~Yam, \emph{Mean field games and mean field
  type control theory}, Springer Briefs in Mathematics, Springer, New York,
  2013.

\bibitem{MR3116016}
P.~Cardaliaguet, G.~Carlier, and B.~Nazaret, \emph{Geodesics for a class of
  distances in the space of probability measures}, Calc. Var. Partial
  Differential Equations \textbf{48} (2013), no.~3-4, 395--420.

\bibitem{cardaliaguet2014second}
P.~Cardaliaguet, J.~Graber, A.~Porretta, and D.~Tonon, \emph{Second order mean
  field games with degenerate diffusion and local coupling}, NoDEA Nonlinear
  Differential Equations Appl. \textbf{22} (2015), no.~5, 1287--1317.

\bibitem{MR3091726}
R.~Carmona and F.~Delarue, \emph{Mean field forward-backward stochastic
  differential equations}, Electron. Commun. Probab. \textbf{18} (2013), no.
  68, 15.

\bibitem{MR3045029}
R.~Carmona, F.~Delarue, and A.~Lachapelle, \emph{Control of {M}c{K}ean-{V}lasov
  dynamics versus mean field games}, Math. Financ. Econ. \textbf{7} (2013),
  no.~2, 131--166.

\bibitem{EcksteinBertsekas1992douglas}
J.~Eckstein and D.~P Bertsekas, \emph{On the douglas-rachford splitting method
  and the proximal point algorithm for maximal monotone operators},
  Mathematical Programming \textbf{55} (1992), no.~1-3, 293--318.

\bibitem{eckstein2012augmented}
J.~Eckstein and W.~Yao, \emph{Augmented lagrangian and alternating direction
  methods for convex optimization: A tutorial and some illustrative
  computational results}, RUTCOR Research Reports \textbf{32} (2012).

\bibitem{FortinGlowinski2000augmented}
M.~Fortin and R.~Glowinski, \emph{Augmented lagrangian methods: applications to
  the numerical solution of boundary-value problems}, Elsevier, 2000.

\bibitem{GabayMercier1976dual}
D.~Gabay and B.~Mercier, \emph{A dual algorithm for the solution of nonlinear
  variational problems via finite element approximation}, Computers \&
  Mathematics with Applications \textbf{2} (1976), no.~1, 17--40.

\bibitem{MR3195844}
D.~A. Gomes and J.~Sa{\'u}de, \emph{Mean field games models---a brief survey},
  Dyn. Games Appl. \textbf{4} (2014), no.~2, 110--154.

\bibitem{MR0271809}
M.~R. Hestenes, \emph{Multiplier and gradient methods}, J. Optimization Theory
  Appl. \textbf{4} (1969), 303--320. \MR{0271809}

\bibitem{MR2269875}
J-M. Lasry and P-L. Lions, \emph{Jeux \`a champ moyen. {I}. {L}e cas
  stationnaire}, C. R. Math. Acad. Sci. Paris \textbf{343} (2006), no.~9,
  619--625.

\bibitem{MR2271747}
\bysame, \emph{Jeux \`a champ moyen. {II}. {H}orizon fini et contr\^ole
  optimal}, C. R. Math. Acad. Sci. Paris \textbf{343} (2006), no.~10, 679--684.

\bibitem{MR2295621}
\bysame, \emph{Mean field games}, Jpn. J. Math. \textbf{2} (2007), no.~1,
  229--260.

\bibitem{PLL}
P-L. Lions, \emph{Cours du {C}oll{\`e}ge de {F}rance},
  http://www.college-de-france.fr/default/EN/all/equ$_-$der/, 2007-2011.

\bibitem{LionsMercier1979splitting}
P-L. Lions and B.~Mercier, \emph{Splitting algorithms for the sum of two
  nonlinear operators}, SIAM Journal on Numerical Analysis \textbf{16} (1979),
  no.~6, 964--979.

\bibitem{MR0221595}
H.~P. McKean, Jr., \emph{A class of {M}arkov processes associated with
  nonlinear parabolic equations}, Proc. Nat. Acad. Sci. U.S.A. \textbf{56}
  (1966), 1907--1911.

\bibitem{BoydParikh2014proximal}
N.~Parikh and S.~Boyd, \emph{Proximal algorithms}, Foundations and Trends in
  Optimization \textbf{1} (2014), no.~3, 127--239.

\bibitem{porretta2014}
A.~Porretta, \emph{Weak solutions to {F}okker-{P}lanck equations and mean field
  games}, Archive for Rational Mechanics and Analysis (2014), 1--62 (English).

\bibitem{MR0272403}
M.~J.~D. Powell, \emph{A method for nonlinear constraints in minimization
  problems}, Optimization ({S}ympos., {U}niv. {K}eele, {K}eele, 1968), Academic
  Press, London, 1969, pp.~283--298. \MR{0272403}

\bibitem{MR1451876}
R.~T. Rockafellar, \emph{Convex analysis}, Princeton Landmarks in Mathematics,
  Princeton University Press, Princeton, NJ, 1997, Reprint of the 1970
  original, Princeton Paperbacks.

\bibitem{MR1108185}
A-S. Sznitman, \emph{Topics in propagation of chaos}, \'{E}cole d'\'{E}t\'e de
  {P}robabilit\'es de {S}aint-{F}lour {XIX}---1989, Lecture Notes in Math.,
  vol. 1464, Springer, Berlin, 1991, pp.~165--251.

\bibitem{MR1149111}
H.~A. van~der Vorst, \emph{Bi-{CGSTAB}: a fast and smoothly converging variant
  of {B}i-{CG} for the solution of nonsymmetric linear systems}, SIAM J. Sci.
  Statist. Comput. \textbf{13} (1992), no.~2, 631--644. \MR{1149111
  (92j:65048)}

\end{thebibliography}
}

\end{document}